\documentclass{amsart}
\usepackage{amssymb}
\usepackage{amscd}
\usepackage{epsfig}
\usepackage{graphicx}

\usepackage{pgf,tikz}
\usepackage{braids}
\usepackage{caption}

\usetikzlibrary{arrows}
\usetikzlibrary{decorations.pathreplacing}

\title[Galois action on knots I]
{Galois action on knots I: \\
Action of the absolute Galois group}
\author{Hidekazu Furusho}

\address{Graduate School of Mathematics, Nagoya University, 
Furo-cho, Chikusa-ku, Nagoya, 464-8602,  Japan}

\email{furusho@math.nagoya-u.ac.jp}

\date{June 27, 2015.}


\newtheorem{thm}{Theorem}[section]
\newtheorem{lem}[thm]{Lemma}

\newtheorem{prop}[thm]{Proposition}  

\theoremstyle{remark}

\theoremstyle{definition}
\newtheorem{defn}[thm]{Definition}
\newtheorem{rem}[thm]{Remark}
\newtheorem{nota}[thm]{Notation}     
\newtheorem{note}[thm]{Note}
\newtheorem{eg}[thm]{Example}       
\newtheorem{conj}[thm]{Conjecture}    
   
\newtheorem{prob}[thm]{Problem}
\newtheorem{project}[thm]{Project}   

\numberwithin{equation}{section}
\numberwithin{figure}{section}

\newcommand{\creation}{\rotatebox[origin=c]{180}{$\curvearrowleft$}}
\newcommand{\opcreation}{\rotatebox[origin=c]{180}{$\curvearrowright$}}
\newcommand{\annihilation}{\curvearrowright}
\newcommand{\opannihilation}{\curvearrowleft}

\def\orientedcircle{\unitlength.2ex
 \begin{minipage}{8\unitlength}
   \begin{tikzpicture}
     \draw[->] (0,0.125) arc (90:450:0.125);
   \end{tikzpicture}
 \end{minipage}
}

\newcommand{\opT}{\rotatebox[origin=c]{180}{$T$}}

\def\diaCrossP{\unitlength.2ex
  \begin{minipage}{15\unitlength}
    \begin{picture}(15,15)
      \put(0,0){\vector(1,1){15}}
      \qbezier(15,0)(15,0)(10,5)
      \qbezier(5,10)(0,15)(0,15)
      \put(0,15){\vector(-1,1){0}}
    \end{picture}
  \end{minipage}
}
\def\diaCrossN{\unitlength.2ex
  \begin{minipage}{15\unitlength}
    \begin{picture}(15,15)
      \put(15,0){\vector(-1,1){15}}
      \qbezier(0,0)(0,0)(5,5)
      \qbezier(10,10)(15,15)(15,15)
      \put(15,15){\vector(1,1){0}}
    \end{picture}
  \end{minipage}
}

\def\diaProjected{\unitlength.1em
  \begin{minipage}{14\unitlength}
    \begin{picture}(14,14)
      \put(0,0){\vector(1,1){14}}
      \put(14,0){\vector(-1,1){14}}
    \end{picture}
  \end{minipage}
}
\begin{document}
\bibliographystyle{amsalpha+}
\maketitle

\begin{abstract}
Our aim of this and subsequent papers is to 
enlighten  (a part of,  presumably) arithmetic structures of knots.
This paper introduces a notion of profinite knots which extends topological knots
and shows its various basic properties.
Particularly an action of the  absolute Galois group
of the rational number field
on profinite knots
is rigorously established,
which is attained by our extending 
the action of Drinfeld's Grothendieck-Teichm\"{u}ller group on profinite braid groups into on profinite knots.
\end{abstract}

\tableofcontents

\setcounter{section}{-1}
\section{Introduction}\label{introduction}
It is known that there are analogies between algebraic number theory and
3-dimensional topology.
It is said that  Mazur and Manin among others spotted them in 1960's
and, after a long silence, in 1990's,
Kapranov and Reznikov
(\cite{Kap, R} and their lecture in MPI)
who took up their ideas and
explored them jointly
and Morishita \cite{Mo}
whose works started independently in a more sophisticated  aspect,
settled the new area of mathematics, arithmetic topology.
Lots of analogies 
between algebraic number theory and 3-dimensional topology
are suggested in  arithmetic topology,
however, as far as we know,  no direct 
 relationship seems to be known.
Our attempt of
this and subsequent \cite{F13}
papers is to  give a direct one particularly between  Galois groups and knots.

{\it A profinite tangle diagram}, a profinite analogue of an oriented tangle diagram, 
is introduced in Definition \ref{definition of profinite tangle}
as a consistent finite sequence of 
fundamental profinite tangle diagrams; symbols of three types $A$, $\widehat{B}$ and $C$
(Definition \ref{definition of fundamental profinite tangle}).
{\it A profinite knot diagram}, a profinite analogue of an oriented knot diagram,
 is defined by a profinite tangle diagram without endpoints and 
with a single connected component in Definition \ref{definition of profinite knot}.
A profinite version of Turaev moves is given in our Definition \ref{profinite Turaev moves}
(T1)-(T6),
which determines an equivalence, called isotopy, for profinite tangle diagrams.
The set $\widehat{\mathcal T}$ of profinite tangles means the quotient of
the set of profinite tangle diagrams by the equivalence
and the set $\widehat{\mathcal K}$ of profinite knots
is the subspace of $\widehat{\mathcal T}$ consisting of isotopy classes of profinite knot diagrams  
(Definition \ref{definition of isotopy}).
Our first theorem is

{\bf Theorem A  (Theorem \ref{profinite realization theorem} and \ref{topological monoid}).}

{\it
(1). The space $\widehat{\mathcal K}$ carries a structure of a topological commutative monoid
whose product is given by the connected sum \eqref{connected sum}.

(2). Let $\mathcal K$ denote the set of isotopy classes of  (topological) oriented knots.
Then there is a natural map
$$
h: {\mathcal K}\to \widehat{\mathcal K}.
$$
The map is with dense image and
is a monoid homomorphism with respect to the connected sum.
}\\
The map $h$  is naturally conjectured to be injective (Conjecture \ref{injectivity on T}).
It is because, if it is not injective,
the Kontsevich invariant fail to be a perfect knot invariant
(cf. Remark \ref{perfect remark} and \ref{Kontsevich factorization}).

The topological group $\mathrm{Frac}\widehat{\mathcal K}$ of profinite knots is introduced as
the group of fraction of the topological monoid  $\widehat{\mathcal K}$
in Definition \ref{definition of GK}.
A continuous action of the profinite Grothendieck-Teichm\"{u}ller group $\widehat{GT}$ \cite{Dr}
on 
$\mathrm{Frac}\widehat{\mathcal K}$
is rigorously established in Definition \ref{GT-action on profinite knots} - Theorem \ref{GT-action theorem}.
By using the inclusion from the absolute Galois group $G_{\mathbb Q}$
of the rational number field $\mathbb Q$ into $\widehat{GT}$,
our second theorem is obtained.

{\bf Theorem B  (Theorem \ref{Galois representation theorem}).}

{\it
Fix an embedding from the algebraic closure $\overline{\mathbb Q}$ of 
the rational number field $\mathbb Q$
into the complex number field $\mathbb C$.
Then the group $\mathrm{Frac}\widehat{\mathcal K} $ of profinite knots admits
a non-trivial topological $G_{\mathbb Q}$-module structure.
Namely, 
there is a non-trivial continuous Galois representation on profinite knots
$$
\rho_0:G_{\mathbb Q}\to \mathrm{Aut}\ \mathrm{Frac}\widehat{\mathcal K} .
$$
}\\
It is explained that particularly the complex conjugation sends each knot to its mirror image
(Example \ref{mirror image}).
The validity of a knot analogue of the so-called Bely\u \i's theorem, i.e.
the injectivity of $\rho_0$, is posed in Problem \ref{Belyi-type problem}.
Several projects  and problems on the Galois action are posted in the end of this paper.

Our construction of the Galois action could be said as a lift
of the action of Kassel-Turaev \cite{KT98}  given in a proalgebraic setting
into a profinite setting.
Our Galois action on knots might also be related to the \lq Galois relations'
suggested in  Gannon-Walton \cite{GW}.
Our discussion on profinite knots in this paper may be linked to 
Mazur's discussion on profinite equivalence of knots in \cite{Maz}

The contents of the paper is as follows.
\S \ref{Galois action on profinite braids} is devoted to a review of
Drinfeld's work on $\widehat{GT}$ and $G_{\mathbb Q}$-actions on profinite braids.
Main results are presented in \S \ref{Galois action on profinite knots}:
The ABC-construction of profinite knots and their basic properties 
are introduced in \S \ref{definition and properties}.
We also introduce and discuss
the notion of  pro-$l$ knots in \S \ref{pro-l knots}.
It serves for our arguments in \S \ref{absolute Galois action} where
an action of the  absolute Galois group on profinite knots is established.
In Appendix \ref{Two-bridge profinite knots},
a two-bridge profinite knot, a profinite knot with a specific type,
is introduced and its Galois behavior is investigated.
A profinite analogue of a framed knot is introduced and the above two theorems are  extended into
those for profinite framed knots
in Appendix \ref{profinite framed knots}.

\section{Profinite braids}\label{Galois action on profinite braids}
This section is a review mainly  on Drinfeld's work \cite{Dr} of
his profinite Grothendieck-Teichm\"{u}ller group $\widehat{GT}$
and its action on profinite braids.
Definitions of
the profinite braid group $\widehat{B}_n$ and the absolute Galois group $G_{\mathbb Q}$ are recalled
in Example \ref{examples of profinite groups}.
The definition of  $\widehat{GT}$ is presented in Definition \ref{definition of GT}.
In Theorem \ref{GQ to GT}, it is explained that
$\widehat{GT}$ contains $G_{\mathbb Q}$.
The action of $\widehat{GT}$ on $\widehat{B}_n$ is explained in
Theorem \ref{GT-action theorem on braids}.
Specific properties of the action which will be required to next section 
are shown in Proposition \ref{action-change-basepoint-braids}
and Proposition \ref{action-evaluation}.

\begin{defn}
The {\it Artin braid group $B_n$ with $n$-strings}  ($n\geqslant 2$)
is the group generated by $\sigma_i$ 
($1\leqslant i \leqslant n-1$) with two relations
\begin{align*}
&\sigma_i\sigma_{i+1}\sigma_i=\sigma_{i+1}\sigma_i\sigma_{i+1}, \\
&\sigma_i\sigma_j=\sigma_j\sigma_i
\quad \text{if} \quad |i-j|>1.
\end{align*}
The unit of $B_n$ is denoted by $e_n$.
For $n=1$, put $B_1=\{e_1\}$: the trivial group.
The {\it pure braid group $P_n$ with $n$-strings}
is the kernel of the natural projection from $B_n$ to 
the symmetric group ${\frak S}_n$ of degree $n$.
\end{defn}

When $n=2$, there is an identification $B_2\simeq {\mathbb Z}$
and the subgroup $P_2$ corresponds to $2{\mathbb Z}$ under the identification.
When $n=3$, $P_3$ contains a free group $F_2$ generated by $\sigma_1^2$ and $\sigma_2^2$.

\begin{nota}\label{several braid figures}
The generator $\sigma_i$ in $B_n$ is 
depicted as in Figure \ref{sigma}. 
And for $b$ and $b'\in B_n$, 
we draw the product $b\cdot b'\in B_n$ as in 
Figure \ref{product picture} 
(the order of product $b\cdot b'$ is chosen to combine the bottom endpoints of $b$
with the top endpoints of $b'$).
\begin{figure}[h]
\begin{tabular}{c}
  \begin{minipage}{0.5\hsize}
      \begin{center}
          \begin{tikzpicture}
                  \draw[-] (-0.2,0) --(-0.2, 0.5) ;
                    \draw[-] (0,0) --(0, 0.5) ;
                    \draw[dotted] (0.1,0.3) --(0.6, 0.3) ;
                    \draw[-] (0.7,0) --(0.7, 0.5) ;
                     \draw[decorate,decoration={brace,mirror}] (-0.3,0) -- (0.8,0) node[midway,below]{$i-1$};
                   \draw [-] (1.3,0)--(0.9,0.5);
               \draw[color=white, line width=5pt]  (0.9,0)--(1.3,0.5);
                   \draw [-] (0.9,0)--(1.3,0.5);

                   \draw[-] (1.5,0)--(1.5,0.5) ;
                    \draw[-] (1.7,0) --(1.7, .5)  ;

                    \draw[-] (2.6,0) --(2.6, .5)  ;
                   \draw[dotted] (1.8,0.3) --(2.5, 0.3)  ;
                    \draw[decorate,decoration={brace,mirror}] (1.4,0) -- (2.7,0) node[midway,below]{$n-i-1$};
        \end{tikzpicture}
            \caption{$\sigma_i$}
            \label{sigma}
     \end{center}
 \end{minipage}

 \begin{minipage}{0.5\hsize}
      \begin{center}
        \begin{tikzpicture}
                    \draw[-] (4.0,0.2) --(4.0, 0.5) (4.0,1.0) --(4.0, 1.5) (4.0,2.0) --(4.0, 2.3);
                    \draw[-] (4.1,0.2) --(4.1, 0.5) (4.1,1.0) --(4.1, 1.5) (4.1,2.0) --(4.1, 2.3);
                    \draw[dotted] (4.2,0.3) --(4.9, 0.3) (4.2,1.1)--(4.9,1.1)  (4.2,1.35)--(4.9,1.35) (4.2,2.2) --(4.9, 2.2); 
                    \draw[-] (5.0,0.2) --(5.0, 0.5) (5.0,1.0) --(5.0, 1.5) (5.0,2.0) --(5.0, 2.3) ;
                    \draw (3.9,0.5) rectangle (5.1,1);
                    \draw (4.5,0.7) node{$b'$};
                    \draw (3.9,1.5) rectangle (5.1,2.0);
                    \draw (4.5,1.7) node{$b$};
                    \draw[decorate,decoration={brace,mirror}] (3.9,0.2) -- (5.1,0.2) node[midway,below]{$n$};
                    \draw[decorate,decoration={brace}]  (3.9,2.3) -- (5.1,2.3) node[midway,above]{$n$};
                   \draw[color=white, line width=3pt] (3.8,1.25)--(5.2,1.25);

           \end{tikzpicture}
               \caption{$b\cdot b'$}
               \label{product picture}
      \end{center}
  \end{minipage}
\end{tabular}
\end{figure}

For $1\leqslant i<j\leqslant n$,
special elements
$$x_{i,j}=x_{j,i}=(\sigma_{j-1}\cdots\sigma_{i+1})\sigma_i^2(\sigma_{j-1}\cdots\sigma_{i+1})^{-1}$$
generate $P_n$.
For $1\leqslant a\leqslant a+\alpha<b\leqslant b+\beta\leqslant n$,
we define
\begin{align*}
x_{a\cdots a+\alpha,b\cdots b+\beta}:=
&(x_{a,b}x_{a,b+1}\cdots x_{a,b+\beta})\cdot
(x_{a+1,b}x_{a+1,b+1}\cdots x_{a+1,b+\beta}) \\
&\cdots(x_{a+\alpha,b}x_{a+\alpha,b+1}\cdots x_{a+\alpha,b+\beta})
\in P_n.
\end{align*}
They are drawn in Figure \ref{pure generator 1} and \ref{pure generator 2}.
\begin{figure}[h]
\begin{tabular}{c}
  \begin{minipage}{0.5\hsize}
      \begin{center}
           \begin{tikzpicture}
                 \draw[-] (-0.5,0)--(-0.5,1) (-0.2,0)--(-0.2,1) (0.2,0)--(0.2,1) (0.3,0)--(0.3,1) (0.6,0)--(0.6,1) (1.2,0)--(1.2,1) (1.5,0)--(1.5,1);
                 \draw[dotted] (-0.5,0.5)--(-0.2,0.5) (0.3,0.5)--(0.6,0.5) (1.2,0.5)--(1.5,0.5);
                 \draw[-] (0.8,0)--(0.8,0.5);
                 \draw[color=white, line width=7pt](0,0) ..controls(0.1,0.5) and (0.9,0)        ..(1,0.5);
                 \draw[-] (0,0) ..controls(0.1,0.5) and (0.9,0)        ..(1,0.5);
                 \draw[color=white, line width=7pt](1,0.5) ..controls(0.9,1.0) and (0.1,0.5)       ..(0,1);
                 \draw[-] (1,0.5) ..controls(0.9,1.0) and (0.1,0.5)       ..(0,1);
                 \draw[color=white, line width=5pt] (0.8,0.5)--(0.8,1);
                 \draw[-] (0.8,0.5)--(0.8,1);
\draw[decorate,decoration={brace,mirror}] (-0.6,-0.1) -- (-0.1,-0.1) node[midway,below]{$i-1$};
\draw[decorate,decoration={brace}] (-0.6,1.1) -- (0.6,1.1) node[midway,above]{$j-1$};
            \end{tikzpicture}
               \caption{$x_{ij}$}
               \label{pure generator 1}     \end{center}
 \end{minipage}

 \begin{minipage}{0.5\hsize}
     \begin{center}
          \begin{tikzpicture}
                 \draw[-] (-0.5,-0.1)--(-0.5,1.1) (-0.2,-0.1)--(-0.2,1.1)  (0.6,-0.1)--(0.6,1.1)  (0.7,-0.1)--(0.7,1.1) (0.8,-0.1)--(0.8,1.1) (1.9,-0.1)--(1.9,1.1) (2.3,-0.1)--(2.3,1.1);
                \draw[dotted] (-0.5,0.5)--(-0.2,0.5) (0.6,0.95)--(0.8,0.95) (1.9,0.5)--(2.3,0.5);
                 \draw[-] (1.0,0)--(1.0,0.5) (1.1,0)--(1.1,0.5) (1.3,0)--(1.3,0.5) ;
                 \draw[color=white, line width=10pt](0.15,0) ..controls(0.25,0.15) and (1.55,0.25)        ..(1.65,0.5);
                 \draw[-] (0,0) ..controls(0.1, 0.5) and (1.4, 0.25)        ..(1.5,0.5);
                 \draw[-] (0.3,0) ..controls(0.4,0.25) and (1.7,0.25)        ..(1.8,0.5);
                 \draw[color=white, line width=7pt] (1,0.5) ..controls(0.9,1.0) and (0.1,0.5)       ..(0,1);
                 \draw[-] (1.5,0.5) ..controls(1.4,0.75) and (0.1,0.5)       ..(0,1);
                 \draw[-] (1.8,0.5) ..controls(1.7,0.75) and (0.4,0.75)       ..(0.3,1);
                 \draw[color=white, line width=12pt] (1.15,0.5)--(1.15,1);
                 \draw[-] (1,0.5)--(1,1.1) (1.1,0.5)--(1.1,1.1) (1.3,0.5)--(1.3,1.1) ;
                \draw[dotted] (0.1,0.1)--(0.3,0.1) (0.1,0.9)--(0.3,0.9) (1,0.8)--(1.3,0.8) ;
\draw[decorate,decoration={brace,mirror}] (0,-0.1) -- (0.3,-0.1) node[midway,below]{\tiny{$\alpha+1$}};
\draw[decorate,decoration={brace,mirror}] (1,-0.1) -- (1.3,-0.1) node[midway,below]{\tiny{$\beta+1$}};
\draw[decorate,decoration={brace}] (-0.5,1.15) -- (-0.2,1.15) node[midway,above]{\tiny{$a-1$}};
\draw[decorate,decoration={brace}] (0.55,1.3) -- (0.8,1.3) node[midway,above]{\tiny{$b-a-\alpha-1$}};
         \end{tikzpicture}
               \caption{$x_{a\cdots a+\alpha,b\cdots b+\beta}$}
               \label{pure generator 2}
      \end{center}
  \end{minipage}
\end{tabular}
\end{figure}
\end{nota}

Next we briefly review the definition and a few  examples of profinite groups.

\begin{defn}
A topological group $G$ is called a {\it profinite group}
if it is a projective limit $\varprojlim G_i$ of a projective system of finite groups $\{G_i\}_{i\in I}$.
For a discrete group $\Gamma$,
its {\it profinite completion} $\widehat{\Gamma}$ is the profinite group
defined by the projective limit
$$
\widehat{\Gamma}=\varprojlim \Gamma/N
$$
where $N$ runs  over all normal subgroups of $\Gamma$ with finite indices.
\end{defn}

For profinite groups, consult \cite{RZ} for example.
We note there is a natural homomorphism $\Gamma\to\widehat{\Gamma}$.
In the paper, we employ the same symbol 
when we express the image of elements of $\Gamma$ by the map
if there is no confusion.

\begin{eg}\label{examples of profinite groups}
(1).
The set $\widehat{\mathbb Z}$ of {\it profinite integers} is the profinite completion of $\mathbb Z$.
There is an identification
$\widehat{\mathbb Z}\simeq \prod_p{\mathbb Z_p}.$
Here $p$ runs over all primes and $\mathbb Z_p$ stands for the ring of $p$-adic integers.

(2).
The {\it absolute Galois group}  $G_{\mathbb Q}$ of the rational number field $\mathbb Q$ is the profinite group
$$
G_{\mathbb Q}=Gal(\overline{\mathbb Q}/{\mathbb Q}):=\varprojlim Gal(F/{\mathbb Q}).
$$
Here the limit runs over all finite Galois extension $F$ of $\mathbb Q$
and  $Gal(F/{\mathbb Q})$ means its Galois group.

(3).
The {\it profinite braid group} $\widehat{B}_n$ means the profinite completion of $B_n$.
It contains
the {\it profinite pure braid group} $\widehat{P}_n$
(the profinite completion of $P_n$),
which is equal to the kernel of the natural projection $\widehat{B}_n\to {\frak S}_n$.
It  is known that both $B_n$ and $P_n$ are {\it residually finite}, i.e,
their natural maps  are both injective;
\begin{equation}\label{residually finiteness}
B_n\hookrightarrow \widehat{B}_n \quad \text{and} \quad P_n\hookrightarrow \widehat{P}_n.
\end{equation}
When $n=2$, we have $\widehat{B}_2\simeq \widehat{\mathbb Z}$.
For profinite braid groups, we employ the same figures as in Notation \ref{several braid figures}
to express such elements.
Further we also depict $\sigma_i^c\in \widehat{B}_n$ ($c\in\widehat{\mathbb Z}$)
as Figure \ref{sigma-i-c}.
\end{eg}
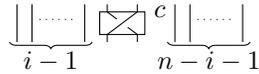
\begin{figure}[h]
          \begin{tikzpicture}
                  \draw[-] (-0.2,0) --(-0.2, 0.5) ;
                    \draw[-] (0,0) --(0, 0.5) ;
                    \draw[dotted] (0.1,0.3) --(0.6, 0.3) ;
                    \draw[-] (0.7,0) --(0.7, 0.5) ;
                     \draw[decorate,decoration={brace,mirror}] (-0.3,0) -- (0.8,0) node[midway,below]{$i-1$};

                   \draw [-] (1.4,0.1)--(1.0,0.4);
               \draw[color=white, line width=5pt]  (1.0,0.1)--(1.4,0.4);
                   \draw [-] (1.0,0.1)--(1.4,0.4);
           \draw  (0.9,0.1) rectangle (1.5,0.4)node[right]{$c$};
                    \draw[-] (1.0,0) --(1.0,0.1) (1.0,0.4)--(1.0,0.5)  (1.4,0) --(1.4,0.1) (1.4,0.4)--(1.4,0.5) ;

                   \draw[-] (1.9,0)--(1.9,0.5) ;
                    \draw[-] (2.1,0) --(2.1, .5)  ;

                    \draw[-] (2.8,0) --(2.8, .5)  ;
                   \draw[dotted] (2.2,0.3) --(2.7, 0.3)  ;
                    \draw[decorate,decoration={brace,mirror}] (1.8,0) -- (2.9,0) node[midway,below]{$n-i-1$};
        \end{tikzpicture}
\caption{$\sigma_i^c\in \widehat{B}_n$ ($c\in\widehat{\mathbb Z}$)}
\label{sigma-i-c}
\end{figure}

To state the definition of $\widehat{GT}$,
we need fix several notations.

\begin{nota}
Let ${F}_2$ be the free  group of rank $2$
with two variables $x$ and $y$ and $\widehat{F}_2$ be its profinite completion.
For any $f\in\widehat{F}_2$ and any group homomorphism 
$\tau:\widehat{F}_2\to G$ sending $x\mapsto \alpha$ and $y\mapsto \beta$,
the symbol $f(\alpha,\beta)$ stands for the image $\tau(f)$.
Particularly for the (actually injective) group homomorphism $\widehat{F}_2\to \widehat{P}_n$
sending $x\mapsto x_{a\cdots a+\alpha,b\cdots b+\beta}$ and 
$y\mapsto x_{b\cdots b+\beta,c\cdots c+\gamma}$
($1\leqslant a\leqslant a+\alpha<b\leqslant b+\beta<c\leqslant c+\gamma
\leqslant n$),
the image of $f\in\widehat{F}_2$ is denoted by 
$f_{a\cdots a+\alpha,b\cdots b+\beta,c\cdots c+\gamma}$.
\end{nota}

The profinite Grothendieck-Teichm\"{u}ller group $\widehat{GT}$
which is a main character of our paper
is defined by  Drinfeld \cite{Dr} to be a profinite subgroup of the topological automorphism
group of ${\widehat F}_2$.

\begin{defn}[\cite{Dr}]\label{definition of GT}
The {\it profinite Grothendieck-Teichm\"{u}ller group}
\footnote{
It is named after Grothendieck's project of {\it un jeu de Teichm\"{u}ller-Lego}
posted 
in {\it Esquisse d'un programme} \cite{G}.
A construction of such group was suggested there
though Drinfeld came to the group independently in his subsequent works \cite{Dr90, Dr} on
deformation of specific type of quantum groups.
}
$\widehat{GT}$ is
the profinite subgroup of  $\mathrm{Aut}\widehat{F}_2$
defined by
\begin{equation*}
\widehat{GT}:=\Bigl\{
\sigma\in \mathrm{Aut}\widehat{F}_2\Bigm |
{\begin{array}{l}
\sigma(x)=x^\lambda, \sigma(y)=f^{-1}y^\lambda f
\text{ for some } (\lambda,f)\in \widehat{\mathbb Z}^\times\times \widehat{F}_2\\
\text{satisfying the three relations below. }
\end{array}}
\Bigr\}
\end{equation*}
\end{defn}
\begin{equation}\label{2-cycle}
f(x,y)f(y,x)=1
\quad\text{ in } \widehat{F}_2,
\end{equation}
\begin{equation}\label{hexagon equation}
f(z,x)z^m f(y,z)y^m f(x,y)x^m=1
\quad\text{ in } \widehat{F}_2
\text{ with } z=(xy)^{-1} \text{ and } m=\frac{\lambda-1}{2},
\end{equation}
\begin{equation}\label{pentagon equation}
f_{1,2,34}f_{12,3,4}=f_{2,3,4}f_{1,23,4}f_{1,2,3}
\quad\text{ in } \widehat{P}_4.
\end{equation}

\begin{rem}
(1).
In some literatures, \eqref{2-cycle}, \eqref{hexagon equation} and \eqref{pentagon equation}
are called {\it 2-cycle, 3-cycle and 5-cycle relation} respectively.
The author often calls  \eqref{2-cycle} and \eqref{hexagon equation} 
by {\it two hexagon equations}
and \eqref{pentagon equation} by  {\it one pentagon equation}
because they reflect the three axioms, two hexagon and one pentagon axioms,
of braided monoidal (tensor) categories \cite{JS}.
We remind that \eqref{pentagon equation} represents
$$
f(x_{12},x_{23}x_{24})f(x_{13}x_{23},x_{34})=f(x_{23},x_{34})f(x_{12}x_{13},x_{24}x_{34})f(x_{12},x_{23})
\quad\text{ in } \widehat{P}_4.
$$
In  several literatures such as \cite{Ih} and \cite{S}, the equation
\eqref{pentagon equation} is replaced by a different (more symmetric) formulation.

(2).
For the {\it proalgebraic} Grothendieck-Teichm\"{u}ller group $GT({\bf k})$ in \cite{Dr}
($\bf k$: an algebraically closed field of characteristic $0$),
it is shown that its one pentagon equation  implies its two hexagon equations in \cite{F10}.
But as for our profinite group $\widehat{GT}$, it is not known
if the pentagon equation \eqref{pentagon equation} implies
two hexagon equations \eqref{2-cycle} and \eqref{hexagon equation}.

(3).
We remark that each $\sigma\in \widehat{GT}$ determines a pair $(\lambda, f)$ uniquely
because the pentagon equation \eqref{pentagon equation} implies that
$f$ belongs to the topological commutator subgroup of $\widehat{F}_2$.
By abuse of notation,
we occasionally express the pair $(\lambda,f)$ 
to represent $\sigma$ and denote as
$\sigma=(\lambda,f)\in \widehat{GT}$. 
The above set-theoretically defined  $\widehat{GT}$ forms indeed a profinite group
whose product is induced from that of $\mathrm{Aut}\widehat{F}_2$
and is given by
\footnote{
For our purpose, we change the order of the product in \cite{Dr}.
}
\begin{equation}\label{product of GT}
(\lambda_2,f_2)\circ (\lambda_1,f_1)=
\Bigl(\lambda_2\lambda_1,f_2\cdot f_1(x^{\lambda_2},f_2^{-1}y^{\lambda_2}f_2)\Bigr).
\end{equation}
\end{rem}

The next lemma will be used later.

\begin{lem}\label{projection of f}
Let $p_i:\widehat{P}_3\to \widehat{P}_2(=\widehat{\mathbb Z})$ ($i=1,2,3$) denote the map
of omission of $i$-th strand in $\widehat{P}_3$.
Let $(\lambda, f)\in\widehat{GT}$.
Then $p_i(f_{1,2,3})=0$  ($i=1,2,3$).
\end{lem}

\begin{proof}
It is easy because $f$ belongs to the commutator of $\widehat{F}_2$ as mentioned above.
\end{proof}

One of the important property of the profinite Grothendieck-Teichm\"{u}ller group
$\widehat{GT}$
is that it contains the absolute Galois group $G_{\mathbb Q}$.

\begin{thm}[\cite{Dr, Ih94}]\label{GQ to GT}
Fix an embedding from $\overline{\mathbb Q}$ into $\mathbb C$.
Then there is an embedding
\begin{equation}\label{GQ into GT}
G_{\mathbb Q} \hookrightarrow \widehat{GT}.
\end{equation}
\end{thm}

We briefly review its proof below.

\begin{proof}
As is explained in \cite{Ih94}, an action of the absolute Galois group on $\widehat{F}_2$, i.e.
\begin{equation}\label{G_Q action on F_2}
G_{\mathbb Q} \to \mathrm{Aut}\ \widehat{F}_2
\end{equation}
is derived from the action on the profinite (scheme-theoretical, cf. \cite{SGA1})
fundamental group
$
\widehat{\pi}_1({\bf P}^1_{\overline{\mathbb Q}}\backslash\{0,1,\infty\},\overrightarrow{01})
$
of the algebraic curve ${\bf P}^1_{\overline{\mathbb Q}}\backslash\{0,1,\infty\}$
with Deligne's \cite{De} tangential base point $\overrightarrow{01}$ 
(this is achieved in the method explained in Remark \ref{algebraic geometry interpretation} below).
The so-called {\it Bely\u \i 's theorem} \cite{Be} claims that
the map \eqref{G_Q action on F_2} is injective:
\begin{equation}\label{Belyi theorem}
G_{\mathbb Q} \hookrightarrow \mathrm{Aut}\ \widehat{F}_2.
\end{equation}
The equations  \eqref{2-cycle}-\eqref{pentagon equation} are checked
for $\sigma\in G_{\mathbb Q}$ in  \cite{Dr, IM},
which means that $G_{\mathbb Q}$ is contained in 
$\widehat{GT}\subset \mathrm{Aut} \widehat{F}_2$.
\end{proof}

\begin{eg}\label{complex conjugation example}
Fix an embedding from $\overline{\mathbb Q}$ into $\mathbb C$.
Then the complex conjugation sending $z\in\mathbb C$ to $\bar z\in\mathbb C$
determines an element $\varsigma_0\in G_{\mathbb Q}$.
It is mapped to the pair $(-1,1)\in \widehat{GT}$
by \eqref{GQ into GT}.
\end{eg}

Asking its surjectivity on the  injection \eqref{GQ into GT} is open for many years:

\begin{prob}
Is $G_{\mathbb Q}$ equal to $\widehat{GT}$ ?
\end{prob}

The following Drinfeld's \cite{Dr}  $\widehat{GT}$-action on $\widehat{B}_n$ 
(a detailed description is also given in \cite{IM})
plays a fundamental role of our results, $\widehat{GT}$-action on knots.

\begin{thm}[\cite{Dr, IM}]\label{GT-action theorem on braids}
Let $n\geqslant 2$.
There is a continuous $\widehat{GT}$-action on $\widehat{B}_n$
\begin{equation*}\label{GT to Aut B_n}
\rho_n:\widehat{GT}\to \mathrm{Aut}\ \widehat{B}_n
\end{equation*}
given by
\begin{equation*}
\sigma=(\lambda,f):
\begin{cases}
\sigma_1 \quad \mapsto \quad&\sigma_1^\lambda, \\
\sigma_i  \quad \mapsto  \quad f_{1\cdots i-1,i,i+1}^{-1}&\sigma_i^\lambda f_{1\cdots i-1,i,i+1}
\qquad (2\leqslant i\leqslant n-1). \\
\end{cases}
\end{equation*}
\end{thm}

We denote $\rho_n(\sigma)(b)$ simply by $\sigma(b)$ 
when there is no confusion.

\begin{rem}
(1).
According to the method to calculate the action in \cite{Dr}
(explicitly presented in the appendix of \cite{IM}),
particularly we have
\begin{align}\label{GT-action on eta}
(\sigma_1\cdots\sigma_i)\mapsto & \
f_{[1],[i],[n-i-1]}^{-1}\cdot
(\sigma_1\cdots\sigma_i)\cdot x_{1\cdots i,i+1}^m ,\\ \notag
(\sigma_i\cdots\sigma_1)\mapsto  & \
x_{1\cdots i,i+1}^m \cdot (\sigma_i\cdots\sigma_1) \cdot
f_{[1],[i],[n-i-1]}.
\end{align}
Here 
$m=\frac{\lambda-1}{2}$
and for $f_{[1],[i],[n-i-1]}$, see \eqref{f[]}.

(2).
We note that $\rho_n$ is injective when $n\geqslant 4$. 
\end{rem}

By Theorem \ref{GQ to GT} and Theorem \ref{GT-action theorem on braids},
we obtain the absolute Galois representation
\begin{equation}\label{GQ to Aut Bn}
\rho_n:  G_{\mathbb Q}\to \mathrm{Aut}\ \widehat{B}_n .
\end{equation}
The below is an algebraic-geometrical interpretation of the Galois action
in terms of Grothendieck's  theory \cite{SGA1} on profinite (scheme-theoretical) fundamental groups.

\begin{rem}\label{algebraic geometry interpretation}
We have a well-known identification between 
the braid group $B_n$
with the topological fundamental group
$\pi_1(X_n(\mathbb C),*)$.
Here  $X_n({\mathbb C})=\mathrm{Conf}^n_{{\frak S}_n}({\mathbb C})$ means  the quotient of
the {\it configuration space} 
$$
\mathrm{Conf}^n({\mathbb C})=\{(z_1,\dots,z_n)\in{\mathbb C}^n| z_i\neq z_j (i\neq j)\}
$$
by the symmetric group ${\frak S}_n$ action
and $*$ is a basepoint. 
%

Let 
$\widehat{\pi}_1(X_n\times{\overline{\mathbb Q}},*)$
denote the profinite (scheme-theoretical) fundamental group of 
$X_n\times{\overline{\mathbb Q}}$ 
in the sense of Grothendieck \cite{SGA1}.
Here the scheme $X_n$ means the $\mathbb Q$-structure of $X_n({\mathbb C})$
and $*$ is a basepoint defined over $\overline{\mathbb Q}$ in the sense of loc.cit.
%
Fix an embedding from $\overline{\mathbb Q}$ into $\mathbb C$,
then, by the so-called  Riemann's existence theorem (lo.cit. VII.Th\'{e}or\`{e}me 5.1),
the group $\widehat{\pi}_1(X_n\times{\overline{\mathbb Q}},*)$ is identified with
the profinite completion of $\pi_1(X_n(\mathbb C),*)$.
Hence we have an identification
\begin{equation}\label{profinite identification}
\widehat{B}_n\simeq \widehat{\pi}_1(X_n\times{\overline{\mathbb Q}},*).
\end{equation}
%
%

Next assume that  $*$ is defined over $\mathbb Q$.
Then by \cite{SGA1} IX.Th\'{e}or\`{e}me 6.1,
we have the homotopy exact sequence of the profinite fundamental group
$$
1\to \widehat{\pi}_1(X_n\times{\overline{\mathbb Q}}, *)\to
\widehat{\pi}_1(X_n, *)\to
\widehat{\pi}_1(Spec{\mathbb Q},*)\to 1.
$$
The last $\widehat{\pi}_1(Spec \ {\mathbb Q},*)$ is nothing but the absolute Galois group $G_{\mathbb Q}$.
A point here is that each basepoint $*$ 
determines a section $s_*$ of the exact sequence.
By \eqref{profinite identification},
the section $s_*$ yields a continuous Galois representation  on  
$\widehat{B}_n$
$$
\rho_{n,*}:G_{\mathbb Q}\to \mathrm{Aut}\ \widehat{B}_n
$$
by inner conjugation, i.e.,
$\rho_{n,*}(\sigma)(b)=s_*(\sigma)\cdot b\cdot s_*(\sigma)^{-1}$
($\sigma\in G_{\mathbb Q}$ and $b\in\widehat{B}_n$).
A specific (tangential  in the sense of Deligne \cite{De})
basepoint $t_n$  is constructed in \cite{IM},
where they showed that the resulting $\rho_{n,t_n}$ is equal to our 
$\rho_n$ in \eqref{GQ to Aut Bn}.
\end{rem}

Special properties
of  the $\widehat{GT}$-action in Theorem \ref{GT-action theorem on braids}
are presented in 
the following Proposition \ref{action-change-basepoint-braids}
and Proposition \ref{action-evaluation}
(though they are implicitly suggested in \cite{Dr}).
They will be employed several times in our paper.

\begin{nota}
Put $n>0$ and
$m_1,m_2\geqslant 0$. 
On the continuous homomorphism 
$$
e_{m_1}\otimes\cdot\otimes e_{m_2}:\widehat{B}_n\to \widehat{B}_{m_1+n+m_2}
$$
which is defined by
$\sigma_i\mapsto \sigma_{m_1+i}$
(obtained by placing the trivial braids $e_{m_1}$ and $e_{m_2}$ on the  left and right respectively),
we denote the image of $b\in\widehat{B}_n$ by
$e_{m_1}\otimes b\otimes e_{m_2}$.
\end{nota}

\begin{prop}\label{action-change-basepoint-braids}
Put $n>0$ and $m_1,m_2\geqslant 0$. 
Let $\sigma=(\lambda,f)\in \widehat{GT}$ and $b\in\widehat{B}_n$.
Then 
\begin{equation}\label{acbb}
\sigma(e_{m_1}\otimes b\otimes e_{m_2})=
f_{[m_1],[n],[m_2]}^{-1}
\cdot (e_{m_1}\otimes \sigma(b)\otimes e_{m_2})\cdot
f_{[m_1],[n],[m_2]}.
\end{equation}
Here
\begin{align}\label{f[]}
f&_{[m_1],[n],[m_2]}:=
f_{1\cdots m_1, m_1+1\cdots m_1+n-1,m_1+n}\cdot \\ \notag
&\quad f_{1\cdots m_1, m_1+1\cdots m_1+n-2,m_1+n-1} \cdots
f_{1\cdots m_1, m_1+1,m_1+2}
\in\widehat{B}_{m_1+n+m_2}. 
\end{align}
\end{prop}

\begin{proof}
It is enough to check \eqref{acbb} for $b=\sigma_i$ ($1\leqslant i \leqslant n-1$).
By Theorem \ref{GT-action theorem on braids},
\begin{align}\label{GT-action on tensor for braid}
&\qquad f_{[m_1],[n],[m_2]}^{-1}\cdot
(e_{m_1}\otimes\sigma(\sigma_i)\otimes e_{m_2})
\cdot f_{[m_1],[n],[m_2]}= \\ \notag
f_{[m_1],[n],[m_2]}^{-1}\cdot
&f_{m_1+1\cdots  m_1+i-1, m_1+i, m_1+i+1}^{-1}\cdot 
\sigma_{m_1+i}^\lambda 
\cdot  f_{m_1+1\cdots m_1+i-1, m_1+i, m_1+i+1} 
\cdot f_{[m_1],[n],[m_2]}.
\notag
\end{align}

\begin{itemize}
\item
When $M\geqslant i+2$,
both
$f_{m_1+1\cdots m_1+i-1, m_1+i, m_1+i+1} $
and $\sigma_{m_1+i}$ commute with
$f_{1\cdots m_1, m_1+1\cdots m_1+M-1,m_1+M}$
because $x_{m_1+1\cdots m_1+i-1,m_1+i}$, $x_{m_1+i, m_1+i+1}$ and
$\sigma_{m_1+i}$ commute with
$x_{1\cdots m_1, m_1+1\cdots m_1+M-1}$ and
$x_{m_1+1\cdots m_1+M-1,m_1+M}$.
Therefore
\begin{align*}
\eqref{GT-action on tensor for braid}=
f_{[m_1],[i+1],[m_2]}^{-1}\cdot
&f_{m_1+1\cdots m_1+i-1, m_1+i, m_1+i+1}^{-1}\cdot 
\sigma_{m_1+i}^\lambda  \\ \notag
&\qquad\qquad  \qquad
\cdot  f_{m_1+1\cdots m_1+i-1, m_1+i, m_1+i+1} 
\cdot f_{[m_1],[i+1],[m_2]}.
\end{align*}

\item
When $M=i, i+1$, our calculation goes as follows.
\begin{align*}
\eqref{GT-action on tensor for braid}&
=f_{[m_1],[i-1],[m_2]}^{-1}\cdot
f_{1\cdots m_1, m_1+1\cdots m_1+i-1,m_1+i}^{-1}\cdot
f_{1\cdots m_1, m_1+1\cdots m_1+i,m_1+i+1}^{-1}\cdot \\ \notag
&\quad f_{m_1+1\cdots m_1+i-1, m_1+i, m_1+i+1}^{-1}\cdot 
\sigma_{m_1+i}^\lambda 
\cdot  f_{m_1+1\cdots m_1+i-1, m_1+i, m_1+i+1}\cdot \\ \notag
&\qquad f_{1\cdots m_1, m_1+1\cdots m_1+i,m_1+i+1}\cdot
f_{1\cdots m_1, m_1+1\cdots m_1+i-1,m_1+i}\cdot
f_{[m_1],[i-1],[m_2]}, \\
\intertext{by the pentagon equation \eqref{pentagon equation}, }
&=f_{[m_1],[i-1],[l+m_2]}^{-1}\cdot
f_{1\cdots m_1+i-1, m_1+i, m_1+i+1}^{-1} \cdot \\
&\quad f_{1\cdots m_1, m_1+1\cdots m_1+i-1, m_1+i\ m_1+i+1}^{-1}\cdot
\sigma_{m_1+i}^\lambda \cdot 
f_{1\cdots m_1, m_1+1\cdots m_1+i-1, m_1+i\ m_1+i+1}\cdot \\
&\qquad f_{1\cdots m_1+i-1, m_1+i, m_1+i+1}\cdot
f_{[m_1],[i-1],[l+m_2+2]}. \\
\intertext{Since $\sigma_{m_1+i}$ commutes with $f_{1\cdots m_1, m_1+1\cdots m_1+i-1, m_1+i\ m_1+i+1}$,}
&=f_{[m_1],[i-1],[l+m_2]}^{-1}\cdot
f_{1\cdots m_1+i-1, m_1+i, m_1+i+1}^{-1} \cdot 
\sigma_{m_1+i}^\lambda \cdot  \\
&\qquad \qquad \qquad
f_{1\cdots m_1+i-1, m_1+i, m_1+i+1}\cdot
f_{[m_1],[i-1],[l+m_2+2]}. 
\end{align*}

\item
When $M\leqslant i-1$, both 
$f_{1\cdots m_1+i-1, m_1+i, m_1+i+1}$ and $\sigma_{m_1+i}$ commute with
$f_{1\cdots m_1, m_1+1\cdots m_1+M-1,m_1+M}$.
Therefore
\begin{align*}
\eqref{GT-action on tensor for braid}
&=f_{1\cdots m_1+i-1, m_1+i, m_1+i+1}^{-1} \cdot 
\sigma_{m_1+i}^\lambda \cdot 
f_{1\cdots m_1+i-1, m_1+i, m_1+i+1} \\
&=\sigma(\sigma_{m_1+i})
=\sigma(e_{m_1}\otimes\sigma_i\otimes e_{m_2}).
\end{align*}
\end{itemize}
Hence we get the equality \eqref{acbb}.
\end{proof}

\begin{nota}\label{ev-nota}
Let $l,n \geqslant 1$ and $1\leqslant k\leqslant l$.
We consider the continuous group homomorphism
\begin{equation}\label{ev}
P_l\to P_{l+n-1}
\end{equation}
sending, for $1\leqslant i<j\leqslant l$,
\begin{equation*}
x_{i,j}\mapsto
\begin{cases}
x_{i+k-1,j+k-1} & (k<i), \\
x_{i\cdots i+k-1, j+k-1} & (k=i), \\
x_{i,j+k-1} &(i<k<j), \\
x_{i,j\cdots j+k-1} & (k=j), \\
x_{i,j} & (j<k).
\end{cases}
\end{equation*}
We obtain the map by replacing the $k$-th string (from the left)
by the trivial braid $e_n$ with $n$ strings,
hence it naturally extends to two maps (not homomorphisms)
$$
\mathrm{ev}_{k, e_n}:{B}_l\to {B}_{l+n-1}\qquad
\text{ and } \qquad
\mathrm{ev}^{k, e_n}:{B}_l\to {B}_{l+n-1}
$$
which replaces the $k$-th string (from the bottom and  the above left respectively)
by the trivial braid $e_n$ with $n$ strings.
Both of their restrictions into $P_l$ are equal to the above map \eqref{ev}.
Since the map \eqref{ev} also continuously extends into the homomorphism
$\widehat{P}_l\to\widehat{P}_l$, our two maps naturally extend to the maps (not homomorphisms)
$$
\mathrm{ev}_{k, e_n}:\widehat{B}_l\to \widehat{B}_{l+n-1}\qquad
\text{ and } \qquad
\mathrm{ev}^{k, e_n}:\widehat{B}_l\to \widehat{B}_{l+n-1}.
$$
Here we employ the same symbols because there would be no confusion.
\end{nota}

\begin{prop}\label{action-evaluation}
Put $l\geqslant 1$.
Let $\sigma=(\lambda,f)\in\widehat{GT}$ and $b\in \widehat{B}_l$.
Set $k'=b (k)$.
Here $b(k)$ stands for the image of $k$ by the permutation
corresponding to $b$
by the projection $B_l\to{\frak S}_l$.
Then,  for each $k$ with $1\leqslant k\leqslant l$,
we have
\begin{align} \label{ae}
\sigma(\mathrm{ev}_{k,e_n}(b))=&
f_{[k'-1],[n],[l-k']}^{-1}\cdot \mathrm{ev}_{k,e_n}(\sigma(b))\cdot f_{[k-1],[n],[l-k]}, \\
\sigma(\mathrm{ev}^{k',e_n}(b))=&
f_{[k'-1],[n],[l-k']}^{-1}\cdot \mathrm{ev}^{k',e_n}(\sigma(b))\cdot f_{[k-1],[n],[l-k]}.\notag
\end{align}
\end{prop}

\begin{proof}
It suffices to prove the first equality, for
the validity of the second equality is immediate once we have the first equality.
Firstly we prove \eqref{ae} for $b=\sigma_i$ ($1\leqslant i\leqslant l-1$).

\begin{itemize}
\item When $k<i$, we have $\mathrm{ev}_{k,e_n}(\sigma_i)=\sigma_{i+n-1}$ and $k'=k$. Therefore
\begin{align*}
\mathrm{RHS}
&=f_{[k-1],[n],[l-k]}^{-1}\cdot 
\mathrm{ev}_{k,e_n}(f_{1\cdots i-1,i,i+1}^{-1}\sigma_i^\lambda f_{1\cdots i-1,i,i+1})
\cdot f_{[k-1],[n],[l-k]} \\
&=f_{[k-1],[n],[l-k]}^{-1}\cdot 
(f_{1\cdots i+n-2,i+n-1,i+n}^{-1}\sigma_{i+n-1}^\lambda f_{1\cdots i+n-2,i+n-1,i+n})
\cdot f_{[k-1],[n],[l-k]}. \\
\intertext{By $k-1+n\leqslant i+n-2$, $ f_{[k-1],[n],[l-k]}$ commutes with
$f_{1\cdots i+n-2,i+n-1,i+n}$ and $\sigma_{i+n-1}$. Thus}
&=
f_{1\cdots i+n-2,i+n-1,i+n}^{-1}\sigma_{i+n-1}^\lambda f_{1\cdots i+n-2,i+n-1,i+n}
= \sigma(\sigma_{i+n-1}) \\
&=\sigma(\mathrm{ev}_{k,e_n}(\sigma_i))=\mathrm{LHS}.
\end{align*}

\item When $k=i$, we have
$\mathrm{ev}_{k,e_n}(\sigma_i)=\mathrm{ev}_{i,e_n}(\sigma_i)
=\sigma_i\sigma_{i+1}\cdots\sigma_{i+n-1}$
and $k'=k+1=i+1$. Therefore
\begin{align*}
\mathrm{RHS}
&=f_{[i],[n],[l-i-1]}^{-1}\cdot 
\mathrm{ev}_{i,e_n}(f_{1\cdots i-1,i,i+1}^{-1}\sigma_i^\lambda f_{1\cdots i-1,i,i+1})
\cdot f_{[i-1],[n],[l-i]} \\
&=f_{[i],[n],[l-i-1]}^{-1}\cdot 
f_{1\cdots i-1,i,i+1\cdots i+n}^{-1}\cdot
\mathrm{ev}_{i,e_n}(\sigma_i^\lambda) \cdot f_{1\cdots i-1,i\cdots i+n-1,i+n}
\cdot f_{[i-1],[n],[l-i]} \\
&=f_{[i],[n],[l-i-1]}^{-1}\cdot 
f_{1\cdots i-1,i,i+1\cdots i+n}^{-1}\cdot
\mathrm{ev}_{i,e_n}(\sigma_i\cdot x_{i,i+1}^m) 
\cdot f_{[i-1],[n+1],[l-i-1]} \\
&= f_{[i],[n],[l-i-1]}^{-1}\cdot 
f_{1\cdots i-1,i,i+1\cdots i+n}^{-1}\cdot
(\sigma_i\cdots \sigma_{i+n-1})\cdot x_{i\cdots i+n-1,i+n}^m \\
&\qquad\qquad \cdot f_{[i-1],[n+1],[l-i-1]} \\
&=
f_{[i],[n],[l-i-1]}^{-1}\cdot 
f_{1\cdots i-1,i,i+1\cdots i+n}^{-1}\cdot
(e_{i-1}\otimes (\sigma_1\cdots\sigma_n)\cdot x_{1\cdots n,n+1}^m\otimes e_{l-i-1}) \\
&\qquad\qquad \cdot f_{[i-1],[n+1],[l-i-1]}.\\
\intertext{By \eqref{GT-action on eta},}
&=
f_{[i],[n],[l-i-1]}^{-1}\cdot 
f_{1\cdots i-1,i,i+1\cdots i+n}^{-1}\cdot
(e_{i-1}\otimes f_{[1],[n],[0]}\cdot
\sigma(\sigma_1\cdots\sigma_n)\otimes e_{l-i-1}) \\
&\qquad\qquad \cdot f_{[i-1],[n+1],[l-i-1]}\\
&=
f_{[i],[n],[l-i-1]}^{-1}\cdot 
f_{1\cdots i-1,i,i+1\cdots i+n}^{-1}\cdot
(e_{i-1}\otimes f_{[1],[n],[0]}\cdot
\otimes e_{l-i-1}) \\
&\qquad\qquad \cdot (e_{i-1}\otimes 
\sigma(\sigma_1\cdots\sigma_n)\otimes e_{l-i-1}) 
\cdot f_{[i-1],[n+1],[l-i-1]}.\\
\intertext{By Lemma \ref{aux-lem} below,}
&=f_{[i-1],[n+1],[l-i-1]}^{-1}\cdot
(e_{i-1}\otimes \sigma(\sigma_1\cdots\sigma_n)\otimes e_{l-i-1}) 
\cdot f_{[i-1],[n+1],[l-i-1]}.\\
\intertext{By Proposition \ref{action-change-basepoint-braids},}
&=\sigma(e_{i-1}\otimes (\sigma_1\cdots\sigma_n) \otimes e_{l-i-1}) 
=\sigma(\sigma_i\cdots\sigma_{i+n-1}) 
=\sigma(\mathrm{ev}_{i,e_n}(\sigma_i))=\mathrm{LHS}.
\end{align*}

\item When $k=i+1$, we have
$\mathrm{ev}_{k,e_n}(\sigma_i)=\mathrm{ev}_{i+1,e_n}(\sigma_i)
=\sigma_{i+n-1}\cdots\sigma_{i+1}\sigma_{i}$
and $k'=k-1=i$. Therefore
\begin{align*}
\mathrm{RHS}
&=f_{[i-1],[n],[l-i]}^{-1}\cdot 
\mathrm{ev}_{i+1,e_n}(f_{1\cdots i-1,i,i+1}^{-1}\sigma_i^\lambda f_{1\cdots i-1,i,i+1})
\cdot f_{[i],[n],[l-i-1]}\\
&=f_{[i-1],[n],[l-i]}^{-1}\cdot 
 f_{1\cdots i-1,i\cdots i+n-1,i+n}^{-1}\cdot
\mathrm{ev}_{i+1,e_n}(\sigma_i^\lambda) \cdot f_{1\cdots i-1,i,i+1\cdots i+n}
\cdot f_{[i],[n],[l-i-1]} \\
&= f_{[i-1],[n+1],[l-i-1]}^{-1}\cdot
\mathrm{ev}_{i+1,e_n}(x_{i,i+1}^m\cdot\sigma_i) 
\cdot f_{1\cdots i-1,i,i+1\cdots i+n} \cdot
f_{[i],[n],[l-i-1]} \\
&=f_{[i-1],[n+1],[l-i-1]}^{-1}\cdot
x_{i\cdots i+n-1, i+n}^m\cdot
(\sigma_{i+n-1}\cdots\sigma_{i+1}\sigma_{i}) \\
&\qquad\qquad 
\cdot f_{1\cdots i-1,i,i+1\cdots i+n} 
\cdot f_{[i],[n],[l-i-1]} \\
&=f_{[i-1],[n+1],[l-i-1]}^{-1}\cdot
(e_{i-1}\otimes x_{1\cdots n,n+1}^m\cdot (\sigma_n\cdots\sigma_1)\otimes e_{l-i-1}) \\
&\qquad\qquad \cdot f_{1\cdots i-1,i,i+1\cdots i+n} \cdot f_{[i],[n],[l-i-1]}.
\intertext{By \eqref{GT-action on eta},}
&=f_{[i-1],[n+1],[l-i-1]}^{-1}\cdot
(e_{i-1}\otimes \sigma(\sigma_n\cdots\sigma_1)\cdot
f_{[1],[n],[0]}^{-1}\otimes e_{l-i-1}) \\
&\qquad\qquad \cdot f_{1\cdots i-1,i,i+1\cdots i+n}\cdot 
f_{[i],[n],[l-i-1]}\\
&=f_{[i-1],[n+1],[l-i-1]}^{-1}\cdot
(e_{i-1}\otimes\sigma(\sigma_n\cdots\sigma_1)\otimes e_{l-i-1}) \cdot
(e_{i-1}\otimes f_{[1],[n],[0]}^{-1}\cdot \otimes e_{l-i-1}) \\
&\qquad\qquad \cdot
f_{1\cdots i-1,i,i+1\cdots i+n}\cdot
f_{[i],[n],[l-i-1]}. \\
\intertext{By Lemma \ref{aux-lem} below,}
&=f_{[i-1],[n+1],[l-i-1]}^{-1}\cdot
(e_{i-1}\otimes \sigma(\sigma_n\cdots\sigma_1)\otimes e_{l-i-1}) 
\cdot f_{[i-1],[n+1],[l-i-1]}.\\
\intertext{By Proposition \ref{action-change-basepoint-braids},}
&=\sigma(e_{i-1}\otimes (\sigma_n\cdots\sigma_1) \otimes e_{l-i-1}) 
=\sigma(\sigma_{i+n-1}\cdots\sigma_{i}) 
=\sigma(\mathrm{ev}_{i+1,e_n}(\sigma_i))=\mathrm{LHS}.
\end{align*}

\item When $k>i+1$, we have
$\mathrm{ev}_{k,e_n}(\sigma_i)=\sigma_i$ and $k'=k$. Therefore
\begin{align*}
\mathrm{RHS}
&=f_{[k-1],[n],[l-k]}^{-1}\cdot 
\mathrm{ev}_{k,e_n}(f_{1\cdots i-1,i,i+1}^{-1}\sigma_i^\lambda f_{1\cdots i-1,i,i+1})
\cdot f_{[k-1],[n],[l-k]} \\
&=f_{[k-1],[n],[l-k]}^{-1}\cdot 
(f_{1\cdots i-1,i,i+1}^{-1}\sigma_i^\lambda f_{1\cdots i-1,i,i+1})
\cdot f_{[k-1],[n],[l-k]}. \\
\intertext{By $i+1\leqslant k-1$, $ f_{[k-1],[n],[l-k]}$ commutes with
$f_{1\cdots i-1,i,i+1}$ and $\sigma_{i}$. Thus}
&=f_{1\cdots i-1,i,i+1}^{-1}\sigma_i^\lambda f_{1\cdots i-1,i,i+1}
= \sigma(\sigma_{i})
=\sigma(\mathrm{ev}_{k,e_n}(\sigma_i))=\mathrm{LHS}.
\end{align*}
\end{itemize}
Whence the equation \eqref{ae} for $b=\sigma_i$ is obtained.

The validity for $b=\sigma_i$ implies the validity for $b\in B_l$
because each element of $B_l$ is a finite product of $\sigma_i$'s.
Whence particularly  we have  the validity for  $P_l$.
Then by continuity we have for $\widehat{P}_l$.
Since we have the validity for $B_l$ and $\widehat{P}_l$,
we have the validity for $\widehat{B}_l$.
%
\end{proof}

The auxiliary lemma below is required to prove the above proposition.

\begin{lem}\label{aux-lem}
For $\sigma=(\lambda,f)\in\widehat{GT}$ and $i,n,l>0$ with $l>i$, the following equation holds
in $\widehat{B}_{l+n-1}$:
\begin{equation*}
(e_{i-1}\otimes f_{[1],[n],[0]}\otimes e_{l-i-1})\cdot f_{[i-1],[n+1],[l-i-1]}
=f_{1\cdots i-1,i,i+1\cdots i+n}\cdot f_{[i],[n],[l-i-1]}.
\end{equation*}
\end{lem}

\begin{proof}
The above equation can be read as
\begin{align*}
(f_{i,i+1\cdots i+n-1,i+n}\cdots f_{i,i+1,i+2})\cdot
f_{1\cdots i-1,i\cdots i+n-1,i+n}\cdots f_{1\cdots i-1,i,i+1} \\
=f_{1\cdots i-1,i,i+1\cdots i+n} \cdot
(f_{1\cdots i,i+1\cdots i+n-1,i+n}\cdots f_{1\cdots i,i+1,i+2}). 
\end{align*}
It can be proved by successive applications of \eqref{pentagon equation}.
\end{proof}


\section{Profinite knots}\label{Galois action on profinite knots}
This section is to present our main results.
Our ABC-construction of profinite knots is introduced and the basic properties of profinite knots
are shown in \S \ref{definition and properties}.
An action of the absolute Galois group on profinite knots is rigorously established in \S \ref{absolute Galois action},
where the notion of  pro-$l$ knots introduced in \S \ref{pro-l knots}
serves to show its property.

\subsection{ABC-construction}\label{definition and properties}
Profinite tangle diagrams,
profinite analogues of tangle
diagrams, 
are introduced as  consistent finite sequences of symbols of
three types $A$, $\widehat{B}$ and $C$ in Definition \ref{definition of profinite tangle}.
Profinite link diagrams mean profinite tangle diagrams without endpoints and profinite knot diagrams mean profinite link diagrams
with a single connected component (Definition \ref{definition of profinite knot}).
The notion of isotopy for them are given by a profinite analogue of Turaev moves in 
Definition \ref{profinite Turaev moves}.
Two fundamental properties for
the set $\widehat{\mathcal{K}}$ of profinite knots (isotopy classes of profinite knot diagrams) are presented;
Theorem \ref{profinite realization theorem}  explains that
there is  a natural map from the set $\mathcal{K}$ of isotopy classes of (topological) knots
to our set $\widehat{\mathcal K}$
and our Theorem \ref{topological monoid} says that  our $\widehat{\mathcal K}$
carries a structure of a topological commutative monoid.

Here is a brief review of tangles and knots.

\begin{nota}
Let $k,l\geqslant 0$.
Let $\epsilon=(\epsilon_1,\dots,\epsilon_k)$ and $\epsilon'=(\epsilon'_1,\dots,\epsilon'_l)$
be sequences (including the empty sequence $\emptyset$)
of  symbols $\uparrow$ and $\downarrow$.
An (oriented)
\footnote{We occasionally omit to mention it.
Throughout the paper all tangles are assumed to be oriented.}
 {\it tangle} of type $(\epsilon,\epsilon')$
means a smooth embedded compact oriented one-dimensional  real manifolds
in $[0,1]\times {\mathbb C}$
(hence it is a finite disjoint union of embedded one-dimensional intervals
and circles),
whose boundaries are $\{(1,1),\dots,(1,k),(0,1)\dots,(0,l)\}$
such that $\epsilon_i$ (resp. $\epsilon_j'$) is $\uparrow$ or $\downarrow$ 
if the tangle is oriented upwards or downwards at $(1,i)$  (resp. at $(0,j)$) respectively.
A {\it link} is a tangle of type $(\emptyset, \emptyset)$ , i.e. $k=l=0$,
and a {\it knot} means a link with a single connected component.
\end{nota}

A tangle is like \lq a braid'  whose each connected component is allowed 
to be a circle and have endpoints on the same plane.
Figure \ref{example of tangles} might help the readers 
to have a good understanding of the definition.
\begin{figure}[h]
\begin{center}
           \begin{tikzpicture}
                \draw[dotted] (0,0)--(4.5,0);

                 \draw[->] (2,0) ..controls(0.6,2.0) and (1.2,2.6)     ..(1.6,3);
                \draw[-] (1.6,3) ..controls(2.0,3.4) and (3,3)   ..(2.6,2.0);
                 \draw[color=white, line width=7pt](2.6,2.0) ..controls(2.2,1.2)  and (0.4,2.)      ..(1,4.0);
                \draw[->] (2.6,2.0) ..controls(2.2,1.2)  and (0.4,2.)      ..(1,4.0);

                 \draw[->] (2.0,4.0)--(1.82,3.34) ;
                 \draw[color=white, line width=7pt] (1.8,2.9)  arc (180:360:0.7);
                  \draw[->] (1.8,2.9)  arc (180:360:0.7);
                  \draw[->] (3.2,2.9)--(3,4) ;

               \draw[->] (1.4,0.6) ..controls(0.2,0.2) and (-0.1,0.5)  ..(0.2,0.8);
                 \draw[color=white, line width=7pt](0.2,0.8) ..controls(0.6,1.2) and (1.2,1.6)  ..(1.6,1.4);
                \draw[->](0.2,0.8) ..controls(0.6,1.2) and (1.2,1.6)  ..(1.6,1.4);
              \draw[-] (1.6,1.4) ..controls(2.0,1.0) and (1.8,0.9)  ..(1.7,0.8);

                 \draw[color=white, line width=7pt](1,0) ..controls(1.,0.8) and (3.8,-0.5)  ..(4,4);
                 \draw[<-] (1,0) ..controls(1.,0.8) and (3.8,-0.5)  ..(4,4);

                \draw[dotted] (0,4)--(4.5,4.0);
            \end{tikzpicture}
\end{center}
\caption{A tangle of type $({\uparrow\downarrow\uparrow\downarrow, \downarrow\uparrow})$}
\label{example of tangles}  
\end{figure}

The following notion of profinite fundamental tangle diagrams
plays a role of composite elements
of the notion of profinite tangles.

\begin{defn}\label{definition of fundamental profinite tangle}
The set of  {\it fundamental profinite 
tangle diagrams}  means the disjoint union of
the following three sets 
$A$, $\widehat{B}$ and $C$
\footnote{
A, B and C stand for
Annihilations, Braids and Creations respectively.
}
of symbols:
\begin{align*}
A&:=\left\{a_{k,l}^\epsilon \bigm| k,l=0,1,2,\dots
, \  \epsilon=(\epsilon_i)_{i=1}^{k+l+1}\in
\{\uparrow, \downarrow\}^{k}
\times \{\annihilation, \opannihilation\}
\times\{\uparrow, \downarrow\}^{l}
\right\},
\\
\widehat{B}&:=\left\{b_n^\epsilon \bigm| b_n^\epsilon=
\left(b_n,\epsilon=(\epsilon_i)_{i=1}^{n}\right)
\in \widehat{B}_n\times \{\uparrow, \downarrow\}^{n}, n=1,2,3,4,\dots
\right\},
\\
C&:=\left\{c_{k,l}^\epsilon \bigm| k,l=0,1,2,\dots, \ \epsilon=(\epsilon_i)_{i=1}^{k+l+1}\in
\{\uparrow, \downarrow\}^{k}
\times \{\creation,\opcreation\}
\times\{\uparrow, \downarrow\}^{l}
\right\}.
\end{align*}
Here all arrows are merely regarded as symbols.
\end{defn}

We occasionally depict 
these fundamental profinite tangle diagrams with ignorance of arrows
(which represent orientation of each strings)
as the pictures in Figure \ref{fundamental profinite tangles},
which we call their topological pictures.

\begin{figure}[h]
\begin{center}
          \begin{tikzpicture}
\draw  (1.3,-0.7)node{$a_{k,l}^\epsilon$};
                    \draw[-] (0,0) --(0, 0.5) ;
                    \draw[dotted] (0.1,0.3) --(0.6, 0.3) ;
                    \draw[-] (0.7,0) --(0.7, 0.5) ;
                     \draw[decorate,decoration={brace,mirror}] (-0.1,0) -- (0.8,0) node[midway,below]{$k$};
                   \draw (0.9,0.1)  arc (180:0:0.4);
                    \draw[-] (1.8,0) --(1.8, .5)  ;
                    \draw[-] (2.5,0) --(2.5, .5)  ;
                   \draw[dotted] (1.9,0.3) --(2.4, 0.3)  ;
                    \draw[decorate,decoration={brace,mirror}] (1.75,0) -- (2.6,0) node[midway,below]{$l$};
\draw  (4.6,-0.7)node{$b_{n}^\epsilon$};
                    \draw[-] (4.0,0) --(4.0, 0.5) (4.0,1.0) --(4.0, 1.5) ;
                    \draw[-] (4.1,0) --(4.1, 0.5) (4.1,1.0) --(4.1, 1.5) ;
                    \draw[dotted] (4.2,0.3) --(4.9, 0.3) (4.2,1.3) --(4.9, 1.3) ;
                    \draw[-] (5.0,0) --(5.0, 0.5) (5.0,1.0) --(5.0, 1.5) ;
                    \draw (3.9,0.5) rectangle (5.1,1);
                    \draw (4.5,0.7) node{$b_n$};
                    \draw[decorate,decoration={brace,mirror}] (3.9,0) -- (5.1,0) node[midway,below]{$n$};
                    \draw[decorate,decoration={brace}]  (3.9,1.5) -- (5.1,1.5) node[midway,above]{$n$};
\draw  (7.9,-0.7)node{$c_{k,l}^\epsilon$};
                    \draw[-] (6.6,0) --(6.6, 0.5) ;
                    \draw[dotted] (6.7,0.3) --(7.2, 0.3) ;
                    \draw[-] (7.3,0) --(7.3, 0.5) ;
                     \draw[decorate,decoration={brace,mirror}] (6.5,0) -- (7.4,0) node[midway,below]{$k$};
                    \draw (7.5,0.5)  arc (180:360:0.4);
                    \draw[-] (8.4,0) --(8.4, .5)  ;
                    \draw[-] (9.1,0) --(9.1, .5)  ;
                   \draw[dotted] (8.5,0.3) --(9, 0.3)  ;
                    \draw[decorate,decoration={brace,mirror}] (8.3,0) -- (9.2,0) node[midway,below]{$l$};
           \end{tikzpicture}
\caption{Topological picture of fundamental profinite tangle diagrams}
\label{fundamental profinite tangles}
\end{center}
\end{figure}
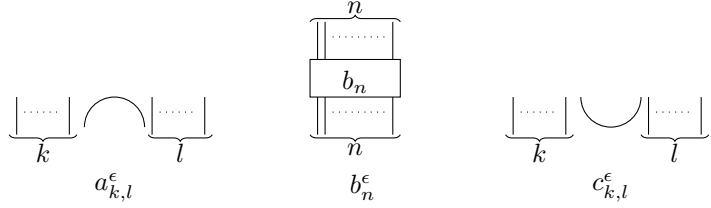

For a fundamental profinite tangle diagram $\gamma$, its {\it source} $s(\gamma)$ and
its {\it target} $t(\gamma)$ are sequences of $\uparrow$ and $\downarrow$ defined below:
\begin{enumerate}
\item 
When $\gamma=a_{k,l}^\epsilon$,  $s(\gamma)$ is the sequence of
$\uparrow$ and $\downarrow$
replacing $\annihilation$ (resp. $\opannihilation$) 
by $\uparrow\downarrow$ (resp.$\downarrow\uparrow$)
in $\epsilon$ and
$t(\gamma)$ is the sequence omitting
$\annihilation$ and $\opannihilation$ in $\epsilon$
(cf. Figure\ref{source and target for a}).
\begin{figure}[h]
\begin{tikzpicture}
                      \draw[-] (0,0) node{$\uparrow$}--(0, 1)node{$\uparrow$};
                      \draw[-] (1,0)node{$\downarrow$} --(1, 1)node{$\downarrow$};
                      \draw (2,0)node{$\downarrow$}  arc (180:0:0.5);
                      \draw[-] (4,0)node{$\uparrow$} --(4, 1)node{$\uparrow$};
                      \draw (3,0) node {$\uparrow$};
\end{tikzpicture}
\caption{$a_{2,1}^{\epsilon}$
with $s(a_{2,1}^{\epsilon})={\uparrow\downarrow\downarrow\uparrow\uparrow}$ and
$t(a_{2,1}^{\epsilon})={\uparrow\downarrow\uparrow}$}
\label{source and target for a}
\end{figure}
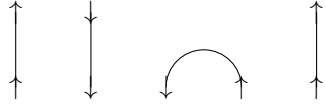

\item
When $\gamma=b_n^\epsilon$, 
$s(\gamma)=\epsilon$ and $t(\gamma)$ is the permutation of $\epsilon$
induced by the image of $b_n^\epsilon$ of the projection $\widehat{B}_n$ to the symmetric group $\frak S_n$  (cf. Figure \ref{source and target}). 
\begin{figure}[h]
\begin{tikzpicture}
\braid  (braid) at (2,0)  a_1^{-1} a_2^{-1};
    \node[at=(braid-1-s)]{$\uparrow$};
    \node[at=(braid-2-s)]{$\downarrow$};
    \node[at=(braid-3-s)] {$\uparrow$};
    \node[at=(braid-1-e)]{$\uparrow$};
    \node[at=(braid-2-e)]{$\downarrow$};
    \node[at=(braid-3-e)] {$\uparrow$};
\end{tikzpicture}
\caption{An example  of $b_3^{\epsilon}$
with $s(b_3^{\epsilon})=\epsilon={\downarrow\uparrow\uparrow}$
and $t(b_3^{\epsilon})={\uparrow\downarrow\uparrow}$}
\label{source and target}
\end{figure}
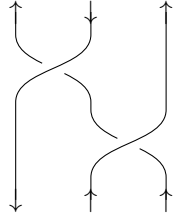

\item
When $\gamma=c_{k,l}^\epsilon$,
$s(\gamma)$ is the set omitting
$\creation$ and $\opcreation$ in $\epsilon$ and
$t(\gamma)$ is the set replacing
$\creation$ (resp. $\opcreation$) 
by $\downarrow\uparrow$ (resp.$\uparrow\downarrow$)
in $\epsilon$.
\end{enumerate}

\begin{defn}\label{definition of profinite tangle}
A {\it profinite  (oriented) tangle diagram} means 
a finite {\it consistent}
\footnote{Here \lq consistent' means successively composable, that is,
$s(\gamma_{i+1})=t(\gamma_i)$ holds for all $i=1,2,\dots, n-1$. 
}
sequence $T=\{\gamma_i\}_{i=1}^n$ of fundamental profinite tangles
(which we denote by $\gamma_n\cdots\gamma_2\cdot\gamma_1$).
Its source and its target are defined by $s(T):=s(\gamma_1)$ and
$t(T):=t(\gamma_n)$. 
A {\it profinite (oriented) link diagram} means a profinite tangle diagram $T$ with $s(T)=t(T)=\emptyset$.
%
\end{defn}

For a fundamental profinite tangle diagram $\gamma$, 
its {\it skeleton} ${\Bbb S}(\gamma)$ is the graph consisting of
finitely many vertices 
and finitely many edges connecting them as follows:
\begin{enumerate}
\item 
When $\gamma=a_{k,l}^\epsilon$ or $c_{k,l}^\epsilon$,
${\Bbb S}(\gamma)$ is nothing but the graph of
the topological picture of $\gamma$ in Figure \ref{fundamental profinite tangles}
whose set of vertices is the collection of its endpoints
and whose set of edges is given by the arrows connecting them
(cf. Figure \ref{skeleton for a}).
\begin{figure}[h]
\begin{center}
\begin{tikzpicture}
                      \draw[-] (0,0) node{$\bullet$}--(0, 1)node{$\bullet$};
                      \draw[-] (1,0)node{$\bullet$} --(1, 1)node{$\bullet$};
                      \draw (2,1)node{$\bullet$}  arc (180:360:0.5);
                      \draw[-] (4,0)node{$\bullet$} --(4, 1)node{$\bullet$};
                      \draw (3,1) node {$\bullet$};
\end{tikzpicture}
\end{center}
\caption{Skeleton of Figure \ref{source and target for a}}
\label{skeleton  for a}
\end{figure}
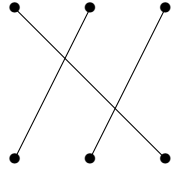

\item
When $\gamma=b_n^\epsilon$, 
${\Bbb S}(\gamma)$ is the graph 
describing the permutation $p(b_n^\epsilon)$ (cf. Figure \ref{skeleton}),
where $p$ means the projection $\widehat{B}_n\to\frak S_n$. 
Namely its set of vertices is the collection of its endpoints and
its set of edges is the set of diagonal lines combining
the corresponding vertices. 
\end{enumerate}
\begin{figure}[h]
\begin{center}
\begin{tikzpicture}
\draw[-] (0,0) node{$\bullet$} --(1, 2) node{$\bullet$};
\draw[-] (1,0) node{$\bullet$}--(2,2) node{$\bullet$};
\draw[-] (2,0) node{$\bullet$} --(0,2) node{$\bullet$};
\end{tikzpicture}
\caption{Skeleton of Figure \ref{source and target}}
\label{skeleton}
\end{center}
\end{figure}
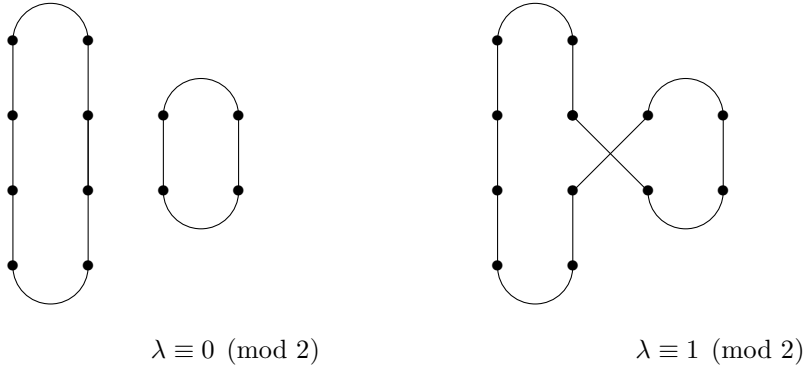

For a profinite tangle diagram $T=\gamma_n\cdots\gamma_2\cdot\gamma_1$,
its skeleton 
${\Bbb S}(T)={\Bbb S}(\gamma_n)\cdots{\Bbb S}(\gamma_2)\cdot{\Bbb S}(\gamma_1)$
means the graph obtained by the composition of 
${\Bbb S}(\gamma_i)$ ($1\leqslant i \leqslant n$)
(cf. Figure \ref{knot or link}-\ref{knot or link 2}).

\begin{defn}\label{definition of profinite knot}
A {\it connected component} of a profinite tangle diagram $T$
means a connected component of the skeleton ${\Bbb S}(T)$ as a graph.
A {\it profinite  (oriented) knot diagram} means a profinite  (oriented) link diagram with a single connected  component.
\end{defn}

It is easy to see that
the number of connected components of any profinite tangle diagram
is always finite.
A profinite knot diagram is a profinite version of 
a (topological) oriented knot diagram.

\begin{prob}
Can we regard
a profinite knot  as a wild knot\footnote{
A  {\it wild knot} means a {\it topological} embedding of the oriented circle into 
$S^3$ (or $\mathbb R^3$).
}?
\end{prob}

It might be nice if we could give any topological meaning for all (or a part of) profinite knot diagrams.

\begin{eg}
The profinite link diagram 
$$
a_{0,0}^{\opannihilation}\cdot 
a_{2,0}^{\downarrow\uparrow\annihilation}\cdot
{(\sigma_2^{\downarrow\uparrow\uparrow\downarrow})}^\lambda\cdot c_{2,0}^{\downarrow\uparrow\opcreation}\cdot 
c^{\creation}_{0,0}
\qquad (\lambda\in\widehat{\mathbb Z})
$$
\begin{figure}[h]
\begin{center}
\begin{tikzpicture}
\draw (0,0)  arc (180:360:0.5);
\draw[-] (0,0) node{$\downarrow$} --(0, 1) node{$\downarrow$}
--(0, 2) node{$\downarrow$}--(0, 3) node{$\downarrow$};
\draw[-] (1,0)node{$\uparrow$} --(1,1)node{$\uparrow$}--(1,1.3);
\draw (2,1)  arc (180:360:0.5);
\draw  (0.9,1.3) rectangle (2.1,1.7)node[right]{$\lambda$};
\draw[-] (1,1.3)--(2,1.7)
             (2,1.3)--(1.75,1.4)
             (1.25,1.6)--(1,1.7);
\draw[-] (3,1) node{$\downarrow$} --(3,2) node{$\downarrow$};
\draw[-] (1,1.7)--(1,2)node{$\uparrow$} --(1,3)node{$\uparrow$};
\draw[-] (2,1)node{$\uparrow$}--(2,1.3);
\draw[-] (2,1.7)--(2,2)node{$\uparrow$};
\draw (1,3)  arc (0:180:0.5);
\draw (3,2)  arc (0:180:0.5);
\draw[color=white, very thick] (-0.1,0)--(1.1,0) (-0.1,1)--(3.1,1) (-0.1,2)--(3.1,2) (-0.1,3)--(1.1,3);
\end{tikzpicture}
\caption{Is this a profinite knot?}
\label{knot or link}
\end{center}
\end{figure}
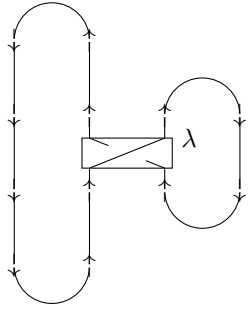
depicted in Figure \ref{knot or link}
(here $\sigma_2^{\downarrow\uparrow\uparrow\downarrow}$
is  the generator $\sigma_2$ of $\widehat{B_4}$
)
is with $2$ connected components if  $\lambda\equiv 0\pmod 2$
and is with a single connected component (hence is a profinite knot)
if   $\lambda\equiv 1\pmod 2$. (cf. Figure \ref{knot or link 2}).
\begin{figure}[h]
\begin{tabular}{cc}
     \begin{minipage}{0.50\hsize}
         \begin{tikzpicture}
                  \draw (0,0)  arc (180:360:0.5);
                  \draw[-] (0,0) node{$\bullet$} --(0, 1) node{$\bullet$}--(0, 2) node{$\bullet$}--(0, 3) node{$\bullet$};
                  \draw[-] (1,0)node{$\bullet$} --(1,1)node{$\bullet$}--(1,2);
                  \draw (2,1)  arc (180:360:0.5);
                  \draw[-] (3,1) node{$\bullet$} --(3,2) node{$\bullet$};
                  \draw[-] (1,1)--(1,2)node{$\bullet$} --(1,3)node{$\bullet$};
                  \draw[-] (2,1)node{$\bullet$}--(2,2)node{$\bullet$};
                  \draw (1,3)  arc (0:180:0.5);
                  \draw (3,2)  arc (0:180:0.5);
           \end{tikzpicture}
                      \caption*{$\lambda \equiv 0\pmod 2$}
    \end{minipage}
    \begin{minipage}{0.50\hsize}
         \begin{tikzpicture}
                  \draw (0,0)  arc (180:360:0.5);
                  \draw[-] (0,0) node{$\bullet$} --(0, 1) node{$\bullet$}--(0, 2) node{$\bullet$}--(0, 3) node{$\bullet$};
                  \draw[-] (1,0)node{$\bullet$} --(1,1)node{$\bullet$}--(2,2)node{$\bullet$};
                  \draw (2,1)  arc (180:360:0.5);
                  \draw[-] (3,1) node{$\bullet$} --(3,2) node{$\bullet$};
                  \draw[-] (1,2)node{$\bullet$} --(1,3)node{$\bullet$};
                  \draw[-] (2,1)node{$\bullet$}--(1,2)node{$\bullet$};
                  \draw (1,3)  arc (0:180:0.5);
                  \draw (3,2)  arc (0:180:0.5);
           \end{tikzpicture}
                       \caption*{$\lambda \equiv 1\pmod 2$\qquad}      
     \end{minipage}
\end{tabular}
    \caption{Skeleton of Figure \ref{knot or link} }
    \label{knot or link 2}
\end{figure} 
\end{eg}

\begin{nota}\label{various notations}
\begin{enumerate}
\item
The symbol $e_n^\epsilon=(e_n,\epsilon)$ stands for 
the fundamental  profinite tangle diagram in $\widehat{B}$
with $s(e_n^\epsilon)=\epsilon$
which corresponds to the  trivial braid $e_n$ in $\hat B_n$.
For a fundamental  profinite tangle diagram $\gamma$,
we mean
$e_{n_1}^{\epsilon_1}\otimes\gamma\otimes e_{n_2}^{\epsilon_2}$
by the  fundamental  profinite tangle diagram obtained by
putting $ e_{n_1}^{\epsilon_1}$ and $ e_{n_2}^{\epsilon_2}$
on the left and on the right of $\gamma$ respectively.
So,   $e_{n_1}^{\epsilon_1}\otimes a_{k,l}^\epsilon\otimes e_{n_2}^{\epsilon_2}
=a_{n_1+k,l+n_2}^{\epsilon_1,\epsilon,\epsilon_2}$,
for instance.
For a profinite  tangle $T=\{\gamma_i\}_{i=1}^n$
($\gamma_i$: a fundamental profinite tangle diagram),
$e_{n_1}^{\epsilon_1}\otimes T\otimes e_{n_2}^{\epsilon_2}$
means the profinite tangle diagram
$\{e_{n_1}^{\epsilon_1}\otimes\gamma_i\otimes e_{n_2}^{\epsilon_2}\}_{i=1}^n$.
\item
Let $\epsilon_0\in\{\uparrow,\downarrow\}^n$ for some $n$.
For a fundamental profinite tangle diagram $b_l^\epsilon\in\widehat{B}$
and $k$ with $1\leqslant k\leqslant l$,
the symbol 
$\mathrm{ev}_{k,\epsilon_0}(b_l^\epsilon)$ (resp. 
$\mathrm{ev}^{k,\epsilon_0}(b_l^\epsilon)$ )
means the element in $\widehat{B}$
which represents the profinite braid replacing the $k$-th string
(from its bottom  (resp. above) left) of $b_l^\epsilon$  by the trivial braid
$e_n^{\epsilon_0}$ with
$n$ strings whose source is $\epsilon_0$
(cf. Notation \ref{ev-nota}).
For instance, the profinite tangle diagram in Figure \ref{source and target}
is described as
$\mathrm{ev}_{1,\downarrow\uparrow} \left( \diaCrossP \right)$
or
$\mathrm{ev}^{2,\downarrow\uparrow} \left( \diaCrossP \right)$.
\end{enumerate}
\end{nota}

\begin{defn}\label{profinite Turaev moves}
For profinite tangle diagrams,  the moves (T1)-(T6) are defined as follow.
\end{defn}

(T1) {\it Trivial braids invariance}:
for a profinite tangle diagram $T$ with $|s(T)|=m$ (resp. $| t(T)|=n$),
\footnote{
For a set $S$, $|S|$ stands for its cardinality.
}
$$
e_n^{t(T)}\cdot T=T=T\cdot e_m^{s(T)}.
$$
For $e_n$, see Notation \ref{various notations}.
Figure \ref{T1} depicts the move.
\begin{figure}[h]
          \begin{tikzpicture}
                      \draw (0,0.3) rectangle (1.0,0.8);
                      \draw (0.5,0.5) node{$T$};
                      \draw[-] (0.1,0) --(0.1, 0.3);
                      \draw[-] (0.2,0) --(0.2, 0.3);
                      \draw[dotted] (0.3,0.1) --(0.8, 0.1);
                      \draw[-] (0.8,0) --(0.8, 0.3);
                      \draw[-] (0.9,0) --(0.9, 0.3);
                      \draw[-] (0.1,0.8) --(0.1, 1.1);
                      \draw[-] (0.25,0.8) --(0.25, 1.1);
                      \draw[dotted] (0.35,1) --(0.8, 1);
                      \draw[-] (0.9,0.8) --(0.9, 1.1);
                      \draw[decorate,decoration={brace,mirror}] (0,0) -- (1,0) node[midway,below]{$m$};
                      \draw[-] (0.1,1.2) --(0.1, 1.5);
                      \draw[-] (0.25,1.2) --(0.25, 1.5);
                      \draw[dotted] (0.35,1.3) --(0.8, 1.3);
                      \draw[-] (0.9,1.2) --(0.9, 1.5);
                      \draw[decorate,decoration={brace}] (0,1.5) -- (1,1.5) node[midway,above]{$n$};
\draw  (1.3,0.5)node{$=$};
                      \draw (1.6,0.3) rectangle (2.6,0.8);
                      \draw (2.1,0.5) node{$T$};
                      \draw[-] (1.7,0) --(1.7, 0.3);
                      \draw[-] (1.8,0) --(1.8, 0.3);
                      \draw[dotted] (1.9,0.1) --(2.4, 0.1);
                      \draw[-] (2.4,0) --(2.4, 0.3);
                      \draw[-] (2.5,0) --(2.5, 0.3);
                      \draw[-] (1.7,0.8) --(1.7, 1.1);
                      \draw[-] (1.85,0.8) --(1.85, 1.1);
                      \draw[dotted] (1.95,1) --(2.4, 1);
                      \draw[-] (2.5,0.8) --(2.5, 1.1);
                      \draw[decorate,decoration={brace,mirror}] (1.6,0) -- (2.6,0) node[midway,below]{$m$};
                      \draw[decorate,decoration={brace}] (1.6,1.1) -- (2.6,1.1) node[midway,above]{$n$};
\draw  (2.9,0.5)node{$=$};
                      \draw (3.2,0.3) rectangle (4.2,0.8);
                      \draw (3.7,0.5) node{$T$};
                      \draw[-] (3.3,-0.1) --(3.3, -0.4);
                      \draw[-] (3.4,-0.1) --(3.4, -0.4);
                      \draw[dotted] (3.5,-0.3) --(4, -0.3);
                      \draw[-] (4,-0.1) --(4, -0.4);
                     \draw[-] (4.1,-0.1) --(4.1, -0.4);
                      \draw[-] (3.3,0) --(3.3, 0.3);
                      \draw[-] (3.4,0) --(3.4, 0.3);
                      \draw[dotted] (3.5,0.1) --(4, 0.1);
                      \draw[-] (4,0) --(4, 0.3);
                      \draw[-] (4.1,0) --(4.1, 0.3);
                      \draw[-] (3.3,0.8) --(3.3, 1.1);
                      \draw[-] (3.45,0.8) --(3.45, 1.1);
                      \draw[dotted] (3.55,1) --(4, 1);
                      \draw[-] (4.1,0.8) --(4.1, 1.1);
                      \draw[decorate,decoration={brace,mirror}] (3.2,-0.4) -- (4.2,-0.4) node[midway,below]{$m$};
                      \draw[decorate,decoration={brace}] (3.2,1.1) -- (4.2,1.1) node[midway,above]{$n$};
           \end{tikzpicture}
\caption{(T1): Trivial braids invariance}
\label{T1}
\end{figure}
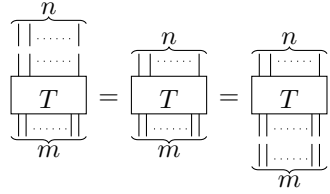 

(T2) {\it Braids composition}:
for  $b_n^{\epsilon_1}, b_n^{\epsilon_2}\in \widehat{B}$ with 
$t(b_n^{\epsilon_1})=s(b_n^{\epsilon_2})$,
$$
b_n^{\epsilon_2}\cdot b_n^{\epsilon_1}=b_n^{\epsilon_3}.
$$
Here $b_n^{\epsilon_3}$ means 
the element in $\widehat{B}$
with $s(b_n^{\epsilon_3})=s(b_n^{\epsilon_1})$ and
$t(b_n^{\epsilon_3})=t(b_n^{\epsilon_2})$
which represents the product $b_2\cdot b_1$ of two braids in $\widehat{B}_n$.
Figure \ref{T2} depicts the move.
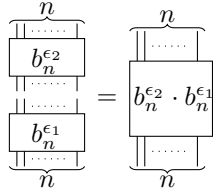
\begin{figure}[h]
       \begin{tikzpicture}
                      \draw[-] (0.1,0) --(0.1, 0.2);
                      \draw[-] (0.2,0) --(0.2, 0.2);
                      \draw[dotted] (0.3,0.1) --(0.8, 0.1);
                      \draw[-] (0.9,0) --(0.9, 0.2);
                      \draw[decorate,decoration={brace,mirror}] (0,0) -- (1,0) node[midway,below]{$n$};
                   \draw (0,0.2) rectangle (1.0,0.7);
                   \draw (0.5,0.4) node{$b_n^{\epsilon_1}$};
                      \draw[-] (0.1,0.7) --(0.1, 0.9);
                      \draw[-] (0.2,0.7) --(0.2, 0.9);
                      \draw[dotted] (0.3,0.8) --(0.8, 0.8);
                      \draw[-] (0.9,0.7) --(0.9, 0.9);
                      \draw[-] (0.1,1) --(0.1, 1.2);
                      \draw[-] (0.2,1) --(0.2, 1.2);
                      \draw[dotted] (0.3,1.1) --(0.8, 1.1);
                      \draw[-] (0.9,1) --(0.9, 1.2);
                   \draw (0,1.2) rectangle (1.0,1.7);
                   \draw (0.5,1.4) node{$b_n^{\epsilon_2}$};
                      \draw[-] (0.1,1.7) --(0.1, 1.9);
                      \draw[-] (0.2,1.7) --(0.2, 1.9);
                      \draw[dotted] (0.3,1.8) --(0.8, 1.8);
                      \draw[-] (0.9,1.7) --(0.9, 1.9);
                      \draw[decorate,decoration={brace}] (0,1.9) -- (1,1.9) node[midway,above]{$n$};
\draw  (1.3,0.9)node{$=$};
                     \draw[-] (1.7,0) --(1.7, 0.4);
                      \draw[-] (1.8,0) --(1.8, 0.4);
                      \draw[dotted] (1.9,0.2) --(2.4, 0.2);
                      \draw[-] (2.5,0) --(2.5, 0.4);
                  \draw (1.6,0.4) rectangle (2.7,1.4);
                  \draw (2.2,0.9) node{$b_n^{\epsilon_2}\cdot b_n^{\epsilon_1}$};
                     \draw[-] (1.7,1.4) --(1.7, 1.8);
                      \draw[-] (1.8,1.4) --(1.8, 1.8);
                      \draw[dotted] (1.9,1.6) --(2.4, 1.6);
                      \draw[-] (2.5,1.4) --(2.5, 1.8);
                      \draw[decorate,decoration={brace,mirror}] (1.6,0) -- (2.6,0) node[midway,below]{$n$};
                     \draw[decorate,decoration={brace}] (1.6,1.8) -- (2.6,1.8) node[midway,above]{$n$};
          \end{tikzpicture}
\caption{(T2): Braids composition}
\label{T2}
\end{figure}

(T3) {\it Independent tangles relation}:
for profinite tangle diagrams $T_1$ and $T_2$ with
$| s(T_1)|=m_1$, $| t(T_1)|=n_1$,
$| s(T_2)|=m_2$ and $ | t(T_2)|=n_2$,
$$
(e^{t(T_1)}_{n_1}\otimes T_2)\cdot (T_1\otimes e^{s(T_2)}_{m_2})=
(T_1\otimes e^{t(T_2)}_{n_2})\cdot (e^{s(T_1)}_{m_1}\otimes T_2).
$$
We occasionally denote both hands side by $T_1\otimes T_2$.
For the symbol $\otimes$, see Notation \ref{various notations}.
Figure \ref{T3} depicts the move.
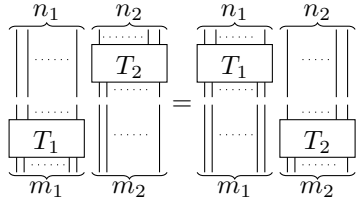
\begin{figure}[h]
           \begin{tikzpicture}
                      \draw[-] (0.1,0) --(0.1, 0.2);
                      \draw[-] (0.2,0) --(0.2, 0.2);
                      \draw[dotted] (0.3,0.1) --(0.8, 0.1);
                      \draw[-] (0.8,0) --(0.8, 0.2);
                      \draw[-] (0.9,0) --(0.9, 0.2);
                      \draw[decorate,decoration={brace,mirror}] (0,0) -- (1,0) node[midway,below]{$m_1$};
                   \draw (0,0.2) rectangle (1.0,0.7);
                   \draw (0.5,0.4) node{$T_1$};
                      \draw[-] (0.1,0.7) --(0.1, 0.9);
                      \draw[-] (0.25,0.7) --(0.25, 0.9);
                      \draw[dotted] (0.35,0.8) --(0.85, 0.8);
                      \draw[-] (0.9,0.7) --(0.9, 0.9);
                      \draw[-] (0.1,1) --(0.1, 1.9);
                      \draw[-] (0.25,1) --(0.25, 1.9);
                      \draw[dotted] (0.35,1.5) --(0.8, 1.5);
                      \draw[-] (0.9,1) --(0.9, 1.9);
                      \draw[decorate,decoration={brace}] (0,1.9) -- (1,1.9) node[midway,above]{$n_1$};
                      \draw[-] (1.2,0) --(1.2, 0.9);
                      \draw[-] (1.3,0) --(1.3, 0.9);
                      \draw[dotted] (1.4,0.5) --(1.9, 0.5);
                     \draw[-] (2,0) --(2, 0.9);
                     \draw[-] (1.2,1.0) --(1.2, 1.2);
                     \draw[-] (1.3,1.0) --(1.3, 1.2);
                      \draw[dotted] (1.4,1.1) --(1.9, 1.1);
                      \draw[-] (2,1) --(2, 1.2);
                   \draw (1.1,1.2) rectangle (2.1,1.7);
                   \draw (1.6,1.4) node{$T_2$};
                      \draw[-] (1.2,1.7) --(1.2, 1.9);     
                     \draw[dotted] (1.25,1.8) --(1.8, 1.8);           
                      \draw[-] (1.85,1.7) --(1.85, 1.9);           
                      \draw[-] (2,1.7) --(2, 1.9);                     
                      \draw[decorate,decoration={brace,mirror}] (1.1,0) -- (2.1,0) node[midway,below]{$m_2$};
                     \draw[decorate,decoration={brace}] (1.1,1.9) -- (2.1,1.9) node[midway,above]{$n_2$};
\draw  (2.3,0.9)node{$=$};
                     \draw[-] (2.6,0) --(2.6, 0.9);
                      \draw[-] (2.7,0) --(2.7, 0.9);
                      \draw[dotted] (2.8,0.5) --(3.3, 0.5);
                      \draw[-] (3.3,0) --(3.3, 0.9);
                     \draw[-] (3.4,0) --(3.4, 0.9);
                     \draw[-] (2.6,1.0) --(2.6, 1.2);
                     \draw[-] (2.7,1.0) --(2.7, 1.2);
                      \draw[dotted] (2.8,1.1) --(3.3, 1.1);
                      \draw[-] (3.3,1) --(3.3, 1.2);
                      \draw[-] (3.4,1) --(3.4, 1.2);
                   \draw (2.5,1.2) rectangle (3.5,1.7);
                   \draw (3,1.4) node{$T_1$};
                      \draw[-] (2.6,1.7) --(2.6, 1.9);     
                     \draw[-] (2.75,1.7) --(2.75, 1.9);     
                     \draw[dotted] (2.8,1.8) --(3.3, 1.8);           
                      \draw[-] (3.4,1.7) --(3.4, 1.9);                     
                      \draw[decorate,decoration={brace,mirror}] (2.5,0) -- (3.5,0) node[midway,below]{$m_1$};
                     \draw[decorate,decoration={brace}] (2.5,1.9) -- (3.5,1.9) node[midway,above]{$n_1$};
                      \draw[-] (3.7,0) --(3.7, 0.2);    
                      \draw[-] (3.8,0) --(3.8, 0.2);
                      \draw[dotted] (3.9,0.1) --(4.4, 0.1);
                      \draw[-] (4.5,0) --(4.5, 0.2);
                      \draw[decorate,decoration={brace,mirror}] (3.6,0) -- (4.6,0) node[midway,below]{$m_2$};
                   \draw (3.6,0.2) rectangle (4.6,0.7);
                   \draw (4.1,0.4) node{$T_2$};
                      \draw[-] (3.7,0.7) --(3.7, 0.9);
                      \draw[dotted] (3.9,0.8) --(4.3, 0.8);
                     \draw[-] (4.35,0.7) --(4.35, 0.9);
                      \draw[-] (4.5,0.7) --(4.5, 0.9);
                      \draw[-] (3.7,1) --(3.7, 1.9);
                      \draw[dotted] (3.9,1.5) --(4.3, 1.5);
                     \draw[-] (4.35,1) --(4.35, 1.9);
                      \draw[-] (4.5,1) --(4.5, 1.9);
                      \draw[decorate,decoration={brace}] (3.6,1.9) -- (4.6,1.9) node[midway,above]{$n_2$};  
          \end{tikzpicture}  
\caption{(T3): Independent tangles relation}
\label{T3}
\end{figure}

(T4) {\it Braid-tangle relations}:
for $b_l^\epsilon\in \widehat{B}$,  $k$ with $1\leqslant k\leqslant l$
and a profinite tangle diagram $T$ with
$| s(T)|=m$ and $| t(T)|=n$,
$$
\mathrm{ev}_{k,t(T)}(b_l^\epsilon)\cdot (e_{k-1}^{s_1}\otimes T\otimes e_{l-k}^{s_2})=
(e_{k'-1}^{t_1}\otimes T\otimes e_{l-k'}^{t_2})\cdot \mathrm{ev}^{k',s(T)}(b_l^\epsilon).
$$
For $\mathrm{ev}$, see Notation \ref{various notations}.
For $s(b_l^\epsilon)=\epsilon=(\epsilon_i)_{i=1}^l$ we put
$s_1:=(\epsilon_i)_{i=1}^{k-1}$ and $s_2:=(\epsilon_i)_{i=k+1}^{l}$.
Put $k'=b_l^\epsilon(k)$. 
Here $b_l^\epsilon(k)$ stands for the image of $k$ by the permutation 
which corresponds to  $b_l^\epsilon$
by the projection $B_l\to{\frak S}_l$.
For $t(b_l^\epsilon)=(\epsilon'_i)_{i=1}^l$ we put
$t_1:=(\epsilon'_i)_{i=1}^{k'-1}$ and $t_2:=(\epsilon'_i)_{i=k'+1}^{l}$.
Figure \ref{T4} depicts the move.
\begin{figure}[h]
         \begin{tikzpicture}
                      \draw[-] (0.1,0) --(0.1, 0.2)  (0.1,-0.1) --(0.1, -0.9);
                      \draw[-] (0.2,0) --(0.2, 0.2) (0.2,-0.1) --(0.2, -0.9);
                      \draw[dotted] (0.3,0.1) --(0.4, 0.1);
                      \draw[dotted] (0.3,-0.5) --(0.4, -0.5);
                      \draw[-]  (0.5,-0.1) --(0.5, -0.9) (0.5,0)--(0.5,0.2) ;
                      \draw[decorate,decoration={brace,mirror}] (0,-0.9) -- (0.6,-0.9) node[midway,below]{\tiny $k-1$};

                      \draw[-] (0.75,0) --(0.75, 0.2) (0.75,-0.1) --(0.75, -0.3) ;
                      \draw[-] (0.78,0) --(0.78, 0.2) (0.78,-0.1) --(0.78, -0.3) ;
                     \draw[-]  (0.78,-0.7) --(0.78, -0.9)  (0.81,-0.7) --(0.81, -0.9) (0.89,-0.7) --(0.89, -0.9)  (0.92,-0.7) --(0.92, -0.9)  (0.96,-0.7)--(0.96,-0.9)  ;
                     \draw[-] (0.81,0) --(0.81, 0.2) (0.81,-0.1) --(0.81, -0.3) ;
                     \draw[dotted] (0.75,0.1) --(0.91, 0.1)   (0.75,-0.2) --(0.91, -0.2) (0.78,-0.8)--(0.96,-0.8) ;
                     \draw[-] (0.88,0) --(0.88, 0.2) (0.88,-0.1) --(0.88, -0.3) ;
                     \draw[-] (0.91,0) --(0.91, 0.2)  (0.91,-0.1)--(0.91,-0.3);
                  \draw (0.65,-0.3) rectangle (1.05,-0.7);
                   \draw (0.85,-0.5) node{$T$};

                     \draw[-]  (1.2,0) --(1.2, 0.2)  (1.2,-0.1) --(1.2, -0.9);
                     \draw[dotted] (1.3,0.1) --(1.9, 0.1)  (1.3,-0.5) --(1.9, -0.5);
                     \draw[-]  (2.0,0) --(2.0, 0.2)  (2.0,-0.1) --(2.0, -0.9) (2.0,0.7) --(2.0, 0.9);
                   \draw[-]  (2.1,0) --(2.1, 0.2)  (2.1,-0.1) --(2.1, -0.9)   (2.1,0.7) --(2.1,0.9);
                   \draw[decorate,decoration={brace,mirror}] (0.65,-0.9) -- (1.05,-0.9) node[midway,below]{\small $m$};
                   \draw[decorate,decoration={brace,mirror}] (1.1,-0.9) -- (2.2,-0.9) node[midway,below]{\tiny $l-k$};
                   \draw (0,0.2) rectangle (2.3,0.7);
                   \draw (1.2,0.4) node{$\mathrm{ev}_{k,t(T)}(b_l^\epsilon)$};
 
                      \draw[-] (0.1,0.7) --(0.1, 0.9);
                      \draw[-] (0.2,0.7) --(0.2, 0.9);
                      \draw[dotted] (0.3,0.8) --(1.1, 0.8);
                     \draw[-] (1.15, 0.7) --(1.15, 0.9);

           \draw[-]  (1.38, 0.9) --(1.38, 0.7)  (1.41, 0.9) --(1.41, 0.7) (1.44, 0.9) --(1.44, 0.7)  (1.52,0.9) --(1.52, 0.7)  (1.56, 0.9)--(1.56, 0.7)  ;
                     \draw[dotted]    (1.38, 0.8)--(1.56, 0.8) ;

                    \draw[-] (1.75,0.7) --(1.75, 0.9);
              \draw[dotted]  (1.8,0.8) --(2.0, 0.8)  ;   
              \draw[-] (2.,0.7) --(2., 0.9) ;
                      \draw[-] (2.1,0.7) --(2.1, 0.9) ;

              \draw[decorate,decoration={brace}] (0,.9) -- (2.3,.9) node[midway,above]{$l-1+n$};

\draw  (2.7,0.4)node{$=$};

                   \draw (3.1,0.2) rectangle (5.4,0.7);
                   \draw (4.3,0.45) node{$\mathrm{ev}^{k',s(T)}(b_l^\epsilon)$};
                      \draw[-] (3.2,0) --(3.2, 0.2);
                      \draw[-] (3.3,0) --(3.3, 0.2);
                      \draw[dotted] (3.4,0.1) --(3.5, 0.1);
                     \draw[-] (3.6,0) --(3.6, 0.2);

                      \draw[-] (3.85,0) --(3.85, 0.2) ;
                      \draw[-] (3.88,0) --(3.88, 0.2) ;
                     \draw[dotted] (3.85,0.1) --(4.01, 0.1) ;
                     \draw[-] (3.95,0) --(3.95, 0.2)  ;
                     \draw[-] (3.98,0) --(3.98, 0.2)  ;
                     \draw[-] (4.01,0) --(4.01, 0.2)  ;

                      \draw[-] (4.2,0) --(4.2, 0.2);
                     \draw[dotted] (4.3,0.1) --(5.0, 0.1) ;
                      \draw[-] (5.1,0) --(5.1, 0.2);
                      \draw[-] (5.2,0) --(5.2, 0.2);
                   \draw[decorate,decoration={brace,mirror}] (3.1,-0.1) -- (5.3,-0.1) node[midway,below]{$l-1+m$};
                      \draw[-] (3.2,0.7) --(3.2, 0.9)  (3.2,1.0)--(3.2,1.9);
                      \draw[-] (3.3,0.7) --(3.3, 0.9)  (3.3,1.0)--(3.3,1.9);
                      \draw[dotted] (3.4,0.8) --(4.1, 0.8)  (3.4,1.5) --(4.1, 1.5);
                      \draw[-] (4.2,0.7) --(4.2, 0.9)  (4.2,1.0)--(4.2,1.9)  ;

                    \draw[-] (4.41, 1.9) --(4.41, 1.7) ;
                      \draw[-]  (4.45,1.9) --(4.45, 1.7) ;
                      \draw[-]  (4.48, 1.9) --(4.48, 1.7) ;
                     \draw[-]  (4.48, 1.3) --(4.48, 1)  (4.51, 1.3) --(4.51, 1) (4.59, 1.3) --(4.59, 1)  (4.62,1.3) --(4.62, 1)  (4.66, 1.3)--(4.66, 1)  ;
                     \draw[dotted]  (4.45,1.8) --(4.61, 1.8)  (4.48, 1.15)--(4.66, 1.15) ;
                     \draw[-]  (4.48, 0.9) --(4.48, 0.7)  (4.51, 0.9) --(4.51, 0.7) (4.59, 0.9) --(4.59, 0.7)  (4.62,0.9) --(4.62, 0.7)  (4.66, 0.9)--(4.66, 0.7)  ;
                     \draw[dotted]    (4.48, 0.8)--(4.66, 0.8) ;
                     \draw[-]  (4.58, 1.9) --(4.58, 1.7) ;
                     \draw[-]  (4.61, 1.9)--(4.61,1.7);
                  \draw (4.35,1.7) rectangle (4.75,1.3);
                   \draw (4.55, 1.5) node{$T$};

               \draw[-] (4.85,0.7) --(4.85, 0.9)  (4.85,1.0)--(4.85,1.9);
               \draw[dotted]  (4.9,0.8) --(5.1, 0.8)  (4.9,1.5) --(5.1, 1.5) ;

                     \draw[-] (5.1,0.7) --(5.1, 0.9)  (5.1,1.0)--(5.1,1.9);
                      \draw[-] (5.2,0.7) --(5.2, 0.9)  (5.2,1.0)--(5.2,1.9);

              \draw[decorate,decoration={brace}] (3.1,1.9) -- (4.3,1.9) node[midway,above]{\small $k'-1$};
\draw[decorate,decoration={brace}] (4.35,1.9) -- (4.7,1.9) node[midway,above]{\small $n$};
             \draw[decorate,decoration={brace}] (4.78,1.9) -- (5.3,1.9) node[midway,above]{\tiny $l-k'$};
          \end{tikzpicture}
\caption{(T4): Braid-tangle relation }
\label{T4}
\end{figure}
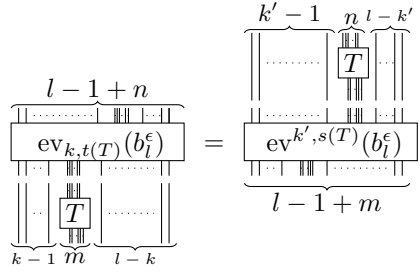

(T5) {\it Creation-annihilation relation}:
for $c_{k,l}^{\epsilon}\in C$ and $a_{k+1,l-1}^{\epsilon'}\in A$
with $t(c_{k,l}^{\epsilon})=s(a_{k+1,l-1}^{\epsilon'})$
$$a_{k+1,l-1}^{\epsilon'}\cdot c_{k,l}^{\epsilon}=e_{k+l}^{s(c_{k,l}^{\epsilon})}.$$
And
for $c_{k,l}^{\epsilon}\in C$ and $a_{k-1,l+1}^{\epsilon'}\in A$
with $t(c_{k,l}^{\epsilon})=s(a_{k-1,l+1}^{\epsilon'})$
$$a_{k-1,l+1}^{\epsilon'}\cdot c_{k,l}^{\epsilon}=e_{k+l}^{s(c_{k,l}^{\epsilon})}.$$
Figure \ref{T5} depicts the move.
\begin{figure}[h]
\begin{tabular}{cc}
\begin{minipage}{0.5\hsize}
\begin{center}
           \begin{tikzpicture}
                    \draw[-] (0,0) --(0, 0.5)  (0,0.6)--(0,1.1);
                    \draw[dotted] (0.1,0.3) --(0.3, 0.3) (0.1,0.9)--(0.3,.9);
                    \draw[-] (0.4,0) --(0.4, 0.5) (0.4,0.6)--(0.4,1.1);
                    \draw[-] (0.5,0.6)--(0.5,1.1);
                     \draw[decorate,decoration={brace,mirror}] (-0.1,-0.1) -- (0.4,-0.1) node[midway,below]{$k$};
\draw (0.5,0.5)  arc (180:360:0.2);
 \draw (0.9,0.6)  arc (180:0:0.2);
                     \draw[decorate,decoration={brace,mirror}] (1.2,-0.1) -- (1.9,-0.1) node[midway,below]{$l$};
                    \draw[-] (1.3,0) --(1.3, 0.5);
                    \draw[-] (1.4,0) --(1.4, 0.5)  (1.4,0.6) --(1.4, 1.1);
                    \draw[dotted] (1.5,0.3) --(1.7, 0.3)  (1.5,0.9) --(1.7, 0.9);
                    \draw[-] (1.8,0) --(1.8, 0.5) (1.8,0.6) --(1.8, 1.1);
\draw  (2.2,0.5)node{$=$};
                    \draw[-] (2.5,0) --(2.5, 1.1);
                    \draw[-] (2.6,0) --(2.6, 1.1);
                   \draw[dotted] (2.7,0.5) --(3.4, 0.5) ;
                    \draw[-] (3.5,0) --(3.5, 1.1);
                    \draw[decorate,decoration={brace,mirror}] (2.4,-0.1) -- (3.6,-0.1) node[midway,below]{$k+l$};
           \end{tikzpicture}
\end{center}
\end{minipage}
\begin{minipage}{0.5\hsize}
\begin{center}
           \begin{tikzpicture}
                    \draw[-] (0,0) --(0, 0.5)  (0,0.6)--(0,1.1);
                    \draw[dotted] (0.1,0.3) --(0.3, 0.3) (0.1,0.9)--(0.3,.9);
                    \draw[-] (0.4,0) --(0.4, 0.5) (0.4,0.6)--(0.4,1.1);
                    \draw[-] (0.5,0)--(0.5,.5);
                     \draw[decorate,decoration={brace,mirror}] (0.,-0.1) -- (0.55,-0.1) node[midway,below]{$k$};
\draw (0.5,0.6)  arc (180:0:0.2);
 \draw (0.9,0.5)  arc (180:360:0.2);
                     \draw[decorate,decoration={brace,mirror}] (1.4,-0.1) -- (1.8,-0.1) node[midway,below]{$l$};
                    \draw[-] (1.3,0.6) --(1.3, 1.1);
                    \draw[-] (1.4,0) --(1.4, 0.5)  (1.4,0.6) --(1.4, 1.1);
                    \draw[dotted] (1.5,0.3) --(1.7, 0.3)  (1.5,0.9) --(1.7, 0.9);
                    \draw[-] (1.8,0) --(1.8, 0.5) (1.8,0.6) --(1.8, 1.1);
\draw  (2.2,0.5)node{$=$};
                    \draw[-] (2.5,0) --(2.5, 1.1);
                    \draw[-] (2.6,0) --(2.6, 1.1);
                   \draw[dotted] (2.7,0.5) --(3.4, 0.5) ;
                    \draw[-] (3.5,0) --(3.5, 1.1);
                    \draw[decorate,decoration={brace,mirror}] (2.5,-0.1) -- (3.5,-0.1) node[midway,below]{$k+l$};
           \end{tikzpicture}
\end{center}
\end{minipage}
\end{tabular}
\caption{(T5): Creation-annihilation relations}
\label{T5}
\end{figure}
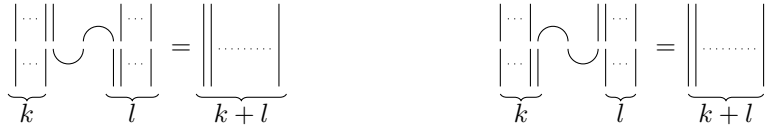

(T6) {\it  First Reidemeister move}:
for $c\in\widehat{\mathbb Z}$
\footnote{
It should be worthy to emphasize that
$c$ is assumed to be not in $\mathbb Z$ but in  $\widehat{\mathbb Z}$.}, 
$c_{k,l}^{\epsilon}\in C$ and
$\sigma_{k+1}^{\epsilon'}\in \widehat{B}$
which represents $\sigma_{k+1}\in {\widehat B}_{k+l+2}$
and $t(c_{k,l}^{\epsilon})=s\bigl((\sigma_{k+1}^{\epsilon'})^c\bigr)$
$$
(\sigma_{k+1}^{\epsilon'})^{c}\cdot c_{k,l}^{\epsilon}=c_{k,l}^{\bar\epsilon}
$$
where $\bar\epsilon$ is chosen to be 
$t(\bar\epsilon)=t((\sigma_{k+1}^{\epsilon'})^{c})$.

For $c\in\widehat{\mathbb Z}$,
$a_{k,l}^{\epsilon}\in A$ and
$\sigma_{k+1}^{\epsilon'}\in \widehat{B}$
which represents $\sigma_{k+1}\in {\widehat B}_{k+l+2}$
and $s(a_{k,l}^{\epsilon})=t\bigl((\sigma_{k+1}^{\epsilon'})^c\bigr)$
$$
a_{k,l}^{\epsilon}\cdot (\sigma_{k+1}^{\epsilon'})^{c}=a_{k,l}^{\bar\epsilon}.
$$
where $\bar\epsilon$ is chosen to be 
$s(\bar\epsilon)=s((\sigma_{k+1}^{\epsilon'})^{c})$.
Figure \ref{T6} depicts the moves.

\begin{figure}[h]
\begin{tabular}{cc}
\begin{minipage}{0.5\hsize}
\begin{center}
          \begin{tikzpicture}
                    \draw[-] (0,0) --(0, 0.5)  (0,0.6)--(0,1.5);
                    \draw[dotted] (0.1,0.3) --(0.3, 0.3) (0.1,1)--(0.3,1);
                    \draw[-] (0.4,0) --(0.4, 0.5) (0.4,0.6)--(0.4,1.5);
                     \draw[decorate,decoration={brace,mirror}] (-0.1,0) -- (0.5,0) node[midway,below]{$k$};
\draw (0.6,0.5)  arc (180:360:0.2);
                    \draw[-] (.6,1.2) --(1, .8);
                    \draw[draw=white,double=black, very thick] (.6,.8) --(1, 1.2);
\draw  (0.5,0.8) rectangle (1.1,1.2) node[right]{$c$};
                    \draw[-] (0.6,0.6)--(0.6,.8) (1.0,0.6)--(1.0,.8);
                    \draw[-] (0.6,1.2)--(0.6,1.5) (1.0,1.2)--(1.0,1.5);
                    \draw[-] (1.5,0) --(1.5, .5) (1.5,.6)--(1.5,1.5) ;
                    \draw[-] (2.0,0) --(2.0, .5)  (2.0,.6)--(2.0,1.5) ;
                   \draw[dotted] (1.6,0.3) --(1.9, 0.3)   (1.6,1) --(1.9, 1) ;
                    \draw[decorate,decoration={brace,mirror}] (1.4,0) -- (2.1,0) node[midway,below]{$l$};
\draw  (2.5,0.7)node{$=$};
                    \draw[-] (3.,0) --(3.,1.5);
                    \draw[dotted] (3.1,0.7) --(3.3, 0.7) ;
                    \draw[-] (3.4,0) --(3.4,1.5);
                     \draw[decorate,decoration={brace,mirror}] (2.9,0) -- (3.5,0) node[midway,below]{$k$};
\draw (3.5,1.5)  arc (180:360:0.2);
                    \draw[-] (4.,0) --(4.,1.5) ;
                    \draw[-] (4.4,0) --(4.4,1.5) ;
                   \draw[dotted] (4.1,0.7) --(4.3, 0.7) ;
                    \draw[decorate,decoration={brace,mirror}] (3.9,0) -- (4.5,0) node[midway,below]{$l$};
\draw (4.7,0) node{,};
           \end{tikzpicture}
\end{center}
\end{minipage}
\begin{minipage}{0.5\hsize}
\begin{center}
          \begin{tikzpicture}
                    \draw[-] (0,0) --(0, 0.9)  (0,1)--(0,1.5);
                    \draw[dotted] (0.1,0.5) --(0.3, 0.5) (0.1,1.3)--(0.3,1.3);
                    \draw[-] (0.4,0) --(0.4, 0.9) (0.4,1)--(0.4,1.5);
                     \draw[decorate,decoration={brace,mirror}] (-0.1,0) -- (0.5,0) node[midway,below]{$k$};
\draw (0.6,1)  arc (180:0:0.2);
                    \draw[-] (.6,.6) --(1, .2);
                    \draw[draw=white,double=black, very thick] (.6,.2) --(1, .6);
\draw  (0.5,0.2) rectangle (1.1,.6) node[right]{$c$};
                    \draw[-] (0.6,0)--(0.6,.2) (1.0,0)--(1.0,.2);
                    \draw[-] (0.6,.6)--(0.6,.9)  (1.0,.6)--(1.0,.9) ;
                    \draw[-] (1.5,0) --(1.5, .9) (1.5,1)--(1.5,1.5) ;
                    \draw[-] (2.,0) --(2., .9)  (2.,1)--(2.,1.5) ;
                   \draw[dotted] (1.6,0.5) --(1.9, 0.5)   (1.6,1.3) --(1.9, 1.3) ;
                    \draw[decorate,decoration={brace,mirror}] (1.4,0) -- (2.1,0) node[midway,below]{$l$};
\draw  (2.5,0.7)node{$=$};
                    \draw[-] (3.,0) --(3.,1.5);
                    \draw[dotted] (3.1,0.7) --(3.3, 0.7) ;
                    \draw[-] (3.4,0) --(3.4,1.5);
                     \draw[decorate,decoration={brace,mirror}] (2.9,0) -- (3.5,0) node[midway,below]{$k$};
\draw (3.5,0.1)  arc (180:0:0.2);
                    \draw[-] (4.,0) --(4.,1.5) ;
                    \draw[-] (4.4,0) --(4.4,1.5) ;
                   \draw[dotted] (4.1,0.7) --(4.3, 0.7) ;
                    \draw[decorate,decoration={brace,mirror}] (3.9,0) -- (4.5,0) node[midway,below]{$l$};
           \end{tikzpicture}
\end{center}
\end{minipage}
\end{tabular}
\caption{(T6): First Reidemeister move}
\label{T6}
\end{figure}
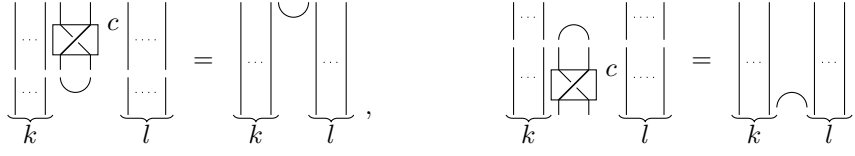
We note that in the first (resp. second) equation
$ c_{k,l}^{\epsilon}=c_{k,l}^{\bar\epsilon}$
(resp. $a_{k,l}^{\epsilon}=a_{k,l}^{\bar\epsilon}$)
if and only if $c \equiv 0\pmod 2$.

These moves (T1)-(T6) are profinite analogues of the  so-called
{\it Turaev moves} \cite{T}
for oriented tangles (consult also \cite{CDM, K, O}).
Our above formulation is stimulated by
the moves presented (R1)-(R11) in \cite{B}. 

\begin{defn}\label{definition of isotopy}
Two profinite (oriented) tangle diagrams $T_1$ and $T_2$ are {\it isotopic},
denoted $T_1=T_2$ by abuse of notation, if 
they are related by a {\it finite} number of the moves (T1)-(T6). 
A {\it profinite tangle}, {\it profinite link} and {\it profinite knot} means an isotopy class of profinite tangle diagrams, link diagrams, knot diagrams respectively.
The set of profinite tangles is denoted by  $\widehat{\mathcal T}$. 
The set $\widehat{\mathcal L}$ of profinite links
(resp. the set $\widehat{\mathcal K}$ of profinite knots)
is the subset of $\widehat{\mathcal T}$ which consists of
isotopy classes of  profinite links
(resp. of profinite knots).
\end{defn}

\begin{note}\label{topology on profinite knots}
Profinite topology on $\widehat{B}_n$ ($n=1,2,\dots$)
and the discrete topology on $A$ and on $C$ 
yield a topology on the space of profinite tangles.
Hence  $\widehat{\mathcal T}$ carries a structure of topological space.
\end{note}

\begin{thm}\label{profinite realization theorem}
(1). Let   $\mathcal T$ be
the set of isotopy classes of  (topological)  oriented tangles.
There is a natural map 
$$h:\mathcal T\to\widehat{\mathcal T},$$
which we call the profinite realization map.

(2). The above profinite realization map induces the map
$$h:\mathcal K\to\widehat{\mathcal K}.$$
Here  $\mathcal K$  stands for the set of  isotopy classes of topological oriented knots.
\end{thm}

\begin{proof}
(1).
The result in \cite{B} indicates that 
the set $\mathcal T$ is described by the set of consistent  finite sequences of  
fundamental tangles,
elements of $A$, 
$$B:=\left\{b_n^\epsilon \bigm| b_n^\epsilon=
\left(b_n,\epsilon=(\epsilon_i)_{i=1}^{n}\right)
\in {B}_n\times \{\uparrow, \downarrow\}^{n}, n=1,2,3,4,\dots
\right\}
$$
and $C$,
modulo the (discrete) Turaev moves.
Since the {\it (discrete) Turaev moves}  in this case mean the moves
replacing profinite tangles and braids by (discrete) tangles and braids
in (T1)-(T6) and $c\in\widehat{\mathbb Z}$ by $c\in{\mathbb Z}$ in (T6).
Because we have a natural map $B_n\to\widehat{B}_n$ and
the Turaev moves are special case of our 6 moves,
we have a natural map $h:\mathcal T\to\widehat{\mathcal T}$.

(2).
It is easy because the set of profinite tangles isotopic to a given profinite knot
consists of only profinite knots.
\end{proof}

We notice that the number of connected components
is an isotopic invariant of profinite tangles.
As a knot analogue of residually-finiteness
\eqref{residually finiteness} of the braid group $B_n$,
we raise the conjecture below.

\begin{conj}\label{injectivity on T}
The map $h:{\mathcal K}\to\widehat{\mathcal K}$ is injective.
\end{conj}

\begin{rem}\label{perfect remark}
The Kontsevich invariant is an invariant of oriented knots
which is conjectured to be  a perfect invariant,
i.e. an invariant detecting all oriented knots (cf. \cite{O-Prob} etc).
We note that if the invariant is perfect,
then the map $h$ is injective
by the arguments given in Remark \ref{Kontsevich factorization} below.
\end{rem}
%

Since last century, lots of knot invariants have been investigated;
such as  the unknotting number, the Alexander polynomial, the  Jones polynomial, 
quandles, the Khovanov homology, etc.
It might be interesting to pose the following.

\begin{prob}\label{profinite knot invariant}
Extend various known knot invariants to those for profinite knots.
\end{prob}

Profinite analogues of  finite type knot invariants will be discussed
in Remark \ref{finite-type remark}.
 

Below we remind one of the most elementary results for oriented knots.

\begin{prop}\label{knot forms a monoid}
The space $\mathcal K$ of topological oriented knots
carries a structure of  a commutative associative monoid
by the connected sum (knot sum).
\end{prop}

Here the connected sum (knot sum)  is a natural way to fuse
two oriented knots, with an appropriate position of orientation,
into one 
(an example is illustrated in Figure \ref{Example of connected sum}).
\begin{figure}[h]
\begin{center}
\begin{tikzpicture}
                 \draw[<-] (0.8,2.2) arc (0:-90:0.6);
                  \draw[<-](0.2,1.6) arc (-90:-180:0.6);
                 \draw [<-] (0,2.0)  arc (270:180:0.2);
                 \draw[color=white, line width=5pt] (0,2.4) ..controls(0.4,2.4) and (0.4,2.0)..(0.8,2.0);
                 \draw[<-] (0,2.4) ..controls(0.4,2.4) and (0.4,2.0)..(0.8,2.0);
                 \draw[color=white, line width=5pt]  (0,2.0) ..controls(0.4,2.0) and (0.4,2.4) ..(0.8, 2.4);
                \draw[-] (0,2.0) ..controls(0.4,2.0) and (0.4,2.4) ..(0.8, 2.4);
                 \draw (0.8,2.0)  arc (-90:90:0.2);
                 \draw[color=white, line width=5pt] (-0.2,2.2) arc (180:0:0.5);  
                 \draw (-0.2,2.2) arc (180:0:0.5);  
                \draw[color=white, line width=5pt]  (-0.4,2.2) ..controls (-0.4,2.4) and (-0.1,2.4) ..(0,2.4);
                  \draw[<-] (-0.4,2.2) ..controls (-0.4,2.4) and (-0.1,2.4) ..(0,2.4);
              \draw[color=white, line width=5pt] (-0.1,0.1).. controls  (-0.3,0.1)  and (-0.1,0.8)..(0.2,0.8);
                \draw [<-] (-0.1,0.1).. controls  (-0.3,0.1)  and (-0.1,0.8)..(0.2,0.8); 
              \draw[color=white, line width=5pt]  (0.5,0.1).. controls  (0.1,0.1)  and (-0.1,1.4)..(0.2,1.4);
                \draw[<-]  (0.5,0.1).. controls  (0.1,0.1)  and (-0.1,1.4)..(0.2,1.4);
              \draw[color=white, line width=5pt] (-0.1,0.1).. controls  (0.3,0.1)  and (0.5,1.4)..(0.2,1.4);
                \draw[->] (-0.1,0.1).. controls  (0.3,0.1)  and (0.5,1.4)..(0.2,1.4);
               \draw[color=white, line width=5pt](0.5,0.1).. controls  (0.7,0.1)  and (0.5,0.8)..(0.2,0.8);
                \draw  (0.5,0.1).. controls  (0.7,0.1)  and (0.5,0.8)..(0.2,0.8);
  \draw[dotted]   (0.2,1.5) circle (0.3);
                \draw (2.2,1.2) node{$\Longrightarrow$};

                 \draw[<-] (4.8,2.2) .. controls (4.8,1.9) and  (4.7,1.6) ..(4.5,1.6);
                 \draw[->] (3.6,2.2) .. controls (3.6,1.9) and  (3.7,1.6) ..(3.9,1.6);
                 \draw [<-] (4.0,2.0)  arc (270:180:0.2);
                 \draw[color=white, line width=5pt] (4.0,2.4) ..controls(4.4,2.4) and (4.4,2.0)..(4.8,2.0);
                 \draw[<-] (4.0,2.4) ..controls(4.4,2.4) and (4.4,2.0)..(4.8,2.0);
                 \draw[color=white, line width=5pt]  (4.0,2.0) ..controls(4.4,2.0) and (4.4,2.4) ..(4.8, 2.4);
                \draw[-] (4.0,2.0) ..controls(4.4,2.0) and (4.4,2.4) ..(4.8, 2.4);
                 \draw (4.8,2.0)  arc (-90:90:0.2);
                 \draw[color=white, line width=5pt] (3.8,2.2) arc (180:0:0.5);  
                 \draw (3.8,2.2) arc (180:0:0.5);  
                \draw[color=white, line width=5pt]  (4-0.4,2.2) ..controls (4-0.4,2.4) and (4-0.1,2.4) ..(4.0,2.4);
                  \draw[<-] (4-0.4,2.2) ..controls (4-0.4,2.4) and (4-0.1,2.4) ..(4.0,2.4);
  \draw[-] (4.5,1.6) .. controls (4.7,1.6) and  (4.3,1.6) ..(4.3,1.4);
  \draw[-] (4-0.1,1.6) .. controls (4-0.3,1.6) and  (4.1,1.6) ..(4.1,1.4);
              \draw[color=white, line width=5pt] (4-0.1,0.1).. controls  (4-0.3,0.1)  and (4-0.1,0.8)..(4.2,0.8);
                \draw [<-] (4-0.1,0.1).. controls  (4-0.3,0.1)  and (4-0.1,0.8)..(4.2,0.8); 
              \draw[color=white, line width=3pt] (4.5,0.1).. controls  (4.1,0.1)  and (4.1,1.4)..(4.1,1.4); 
                \draw[<-]  (4.5,0.1).. controls  (4.1,0.1)  and (4.1,1.4)..(4.1,1.4); 
              \draw[color=white, line width=3pt](4-0.1,0.1).. controls  (4.3,0.1)  and (4.3,1.4)..(4.3,1.4); 
                \draw[->] (4-0.1,0.1).. controls  (4.3,0.1)  and (4.3,1.4)..(4.3,1.4); 
               \draw[color=white, line width=5pt](4.5,0.1).. controls  (4.7,0.1)  and (4.5,0.8)..(4.2,0.8);
                \draw  (4.5,0.1).. controls  (4.7,0.1)  and (4.5,0.8)..(4.2,0.8); 
\end{tikzpicture}
\end{center}
\caption{Connected sum (knot sum)}
\label{Example of connected sum}
\end{figure}
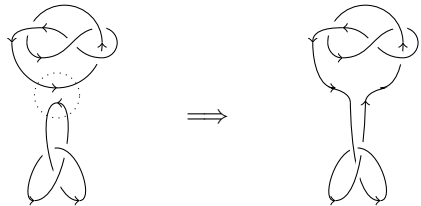

It can be done any points.
It is actually well-defined. 
It yields a commutative monoid structure on 
$\mathcal K$.
For more,
consult the standard text book of knot theory.

The notion of connected sum can be extended into profinite knots.

\begin{thm}\label{topological monoid}
For  any two profinite knots $K_1=\alpha_m\cdots\alpha_1$ and $K_2=\beta_n\cdots\beta_1$
with
$(\alpha_m,\alpha_1)=(\opannihilation,\creation)$
and
$(\beta_n,\beta_1)=(\opannihilation,\creation)$,
their {\sf connected sum} means the  profinite tangle defined by
\begin{equation}\label{connected sum}
K_1\sharp K_2:=\alpha_{m}\cdots\alpha_2\cdot\beta_{n-1}\cdots\beta_1.
\end{equation}
Then

(1). the above connected sum
induces a well-defined product
$$
\sharp:\widehat{\mathcal K}\times\widehat{\mathcal K}\to\widehat{\mathcal K}.
$$

(2). By the product $\sharp$,
the set $\widehat{\mathcal K}$ forms a topological
(that is, the map $\sharp$ is continuous with respect to the topology given above)
commutative associative monoid,
whose unit is given by the oriented circle 
$\orientedcircle:=\opannihilation\cdot\creation$. 

(3). The profinite realization map $h:\mathcal K\to\widehat{\mathcal K}$ forms a
monoid homomorphism 
whose image is dense in $\widehat{\mathcal K}$.
\end{thm}

\begin{proof}
(1).
Since each isotopy class of profinite knot contains by (T6) a profinite knot $K$
of the above type;
a profinite knot
starting with $\creation$ and ending at $\opannihilation$,
we can show that the connected sum extends to $\widehat{\mathcal K}$
once we have the well-definedness.

Firstly we prove that $K_1\sharp K_2$ is isotopic to $K_1'\sharp K_2$
if $K_1'$ is 
isotopic to $K_1$, i.e., $K_1'$ is obtained by a finite number of
our moves (T1)-(T6) from $K_1$.
We may assume that $K_1'$ is obtained from $K_1$
by a single operation of one of the 6 moves.
In the case where this move effects only on 
$\alpha_i$'s for $i>1$,  it is easy to see our claim.
If the moves effects on $\alpha_1$, it must be  (T3) or (T6).
Consider the latter case (T6).
It suffices to show that $K_1\sharp K_2$ is isotopic to
$K_3=\alpha_{m}\cdots\alpha_2\cdot\sigma^{2c}\cdot\beta_{n-1}\cdots\beta_1$
for $c\in\widehat{\mathbb Z}$.
The proof is depicted in Figure \ref{proofT6}.
Here $S_1=\alpha_{m-1}\cdots\alpha_2$ and $S_2=\beta_{n}\cdots\beta_2$.
We note that the first and the fourth equalities follow from  (T3) and (T5).
We use (T4) in the second equality for $\sigma^{2c}$ and the dashed box.
We derived the third equality from (T6).
\begin{figure}[h]
\begin{center}
         \begin{tikzpicture}
                    \draw[->] (0,0)  arc (180:360:0.3);
                     \draw[<-] (0,0)--(0,0.3)  ;
                     \draw[->] (0.6,0)--(0.6,0.3) ;
                  \draw (-0.1,0.3) rectangle (0.7,0.8);
    \draw  (0.3,0.5)node{$S_2$};
                     \draw[<-] (0,0.8)--(0,1.4)  ;
                     \draw[->]  (0.6,0.8)--(0.6,1.4) ;
                     \draw[-] (0.6,1.4)--(0,1.7)  ;
                    \draw[color=white, line width=5pt] (0,1.4)--(0.6,1.7)  ;
                     \draw[-] (0,1.4)--(0.6,1.7)  ;
                      \draw (-0.1,1.4) rectangle (0.7,1.7);                    
                     \draw  (0.9,1.7)node{$2c$};
                    \draw[<-] (0,1.7)--(0,2.3);
                    \draw[->]  (0.6,1.7)--(0.6,2.3) ;
                 \draw (-0.1,2.3) rectangle (0.7,2.8);
    \draw  (0.3,2.5)node{$S_1$};
                     \draw[-] (0,2.8)--(0,3.1) (0.6,2.8)--(0.6,3.1) ;
                    \draw[<-] (0,3.1)  arc (180:0:0.3);
\draw[color=white, very thick] (-0.1,0)--(0.7,0) (-0.1,1.2)--(0.7,1.2) (-0.1,1.9)--(0.7, 1.9)   (-0.1,3.1)--(0.7,3.1) ;
\draw  (1.3,1.4)node{$=$};
                     \draw[-]  (2,0.3)--(2,.4) (2,0.9)--(2,1.4) (2.2,0.3)--(2.2,.4)  (2.2,0.9)--(2.2,1.) ;
                    \draw[->] (2.,0.3)  arc (180:360:0.1);
                     \draw[-] (2.4,0)--(2.4,1.0) ;
                    \draw[->] (2.2,1.0)  arc (180:0:0.1);
                 \draw (1.9,0.4) rectangle (2.3,0.9);
    \draw  (2.1,.6)node{$S_2$};
                 \draw[dotted] (1.8,0.1) rectangle (2.5,1.2);
                    \draw[->] (2.4,0)  arc (180:360:0.3);
                     \draw[-] (3.,0)--(3.,1.4) ;
                     \draw[-] (3.0,1.4)--(2.0,1.7)  ;
                    \draw[color=white, line width=5pt] (2.0,1.4)--(3.0,1.7)  ;
                     \draw[-] (2.0,1.4)--(3.0,1.7)  ;
                 \draw (1.9,1.4) rectangle (3.1,1.7);
                     \draw  (3.3,1.7)node{$2c$};
                    \draw[-] (2.0,1.7)--(2.0,2.3) (3.0,1.7)--(3.0,2.3) ;
                 \draw (1.9,2.3) rectangle (3.1,2.8);
    \draw  (2.5,2.5)node{$S_1$};
                     \draw[-] (2.2,2.8)--(2.2,3.1) (2.8,2.8)--(2.8,3.1) ;
                    \draw[<-] (2.2,3.1)  arc (180:0:0.3);
\draw[color=white, very thick] (1.9,0)--(2.7,0) (1.9,0.3)--(3.1,0.3) (1.9,1.0)--(3.1,1.0)
(1.9,1.3)--(3.1,1.3) (1.9,1.9)--(3.1, 1.9)   (1.9,3.1)--(2.9,3.1) ;
\draw  (3.8,1.4)node{$=$};
                     \draw[-]  (4.5,1.1)--(4.5,1.2) (4.5,1.7)--(4.5,2.3) (4.7,1.1)--(4.7,1.2)  (4.7,1.7)--(4.7,1.8) ;
                    \draw[->] (4.5,1.1)  arc (180:360:0.1);
                     \draw[-] (4.9,0.5)--(4.9,1.8) ;
                    \draw[->] (4.7,1.8)  arc (180:0:0.1);
                 \draw (4.4,1.2) rectangle (4.8,1.7);
    \draw  (4.6,1.4)node{$S_2$};
                 \draw[dotted] (4.3,0.9) rectangle (5,2);
                     \draw[-] (4.9,0)--(4.9,.2) ;
                    \draw[->] (4.9,0)  arc (180:360:0.3);
                     \draw[-] (5.5,0.2)--(4.9,0.5) ;
                    \draw[color=white, line width=5pt] (4.9,.2)--(5.5,.5)  ;
                     \draw[-] (4.9,.2)--(5.5,.5)  ;
                 \draw (4.8,.2) rectangle (5.6,.5);
                     \draw  (5.8,.5)node{$2c$};
                     \draw[-] (5.5,0)--(5.5,0.2) (5.5,0.5)--(5.5,2.3);
                 \draw (4.4,2.3) rectangle (5.6,2.8);
    \draw  (5,2.5)node{$S_1$};
                     \draw[-] (4.7,2.8)--(4.7,3.1) (5.3,2.8)--(5.3,3.1) ;
                    \draw[<-] (4.7,3.1)  arc (180:0:0.3);
\draw[color=white, very thick] (4.4,0)--(5.2,0) (4.4,0.7)--(5.2,0.7)
(4.4,1.1)--(5.6,1.1) (4.4,1.8)--(5.6,1.8) (4.4,2.1)--(5.6, 2.1)  (4.4,3.1)--(5.4,3.1) ;
\draw  (6.2,1.4)node{$=$};
                     \draw[-] (6.6,2.3)--(6.6,1.7)  (6.6,1.2)--(6.6,1.1);
                     \draw (6.5,1.2) rectangle (7.1,1.7);
                     \draw  (6.8,1.4)node{$S_2$};
                     \draw[->] (6.6,1.1)  arc (180:360:0.2);
                     \draw (6.5,2.3) rectangle (7.9,2.8); 
                     \draw[->] (7,1.8)  arc (180:0:0.2);
                     \draw[-] (7,1.8)--(7,1.7)  (7,1.2)--(7,1.1);
                     \draw[-] (7.4,0)--(7.4,1.8) ;
                     \draw[->] (7.4,0.0)  arc (180:360:0.2);
                     \draw[-] (7.8,0.0)--(7.8,2.3);
                      \draw  (7.2,2.5)node{$S_1$};
                     \draw[-] (6.9,2.8)--(6.9,3.1) (7.5,2.8)--(7.5,3.1) ;
                    \draw[<-] (6.9,3.1)  arc (180:0:0.3);
\draw[color=white, very thick] (6.5,0)--(7.1,0) (6.5,1.1)--(7.9,1.1) (6.5,1.8)--(7.9,1.8) 
(6.5,2.1)--(7.9, 2.1)  (6.5,3.1)--(7.9,3.1) ;
\draw  (8.4,1.4)node{$=$};
                     \draw[-] (9.2,0.9)--(9.2,1.2) (9.2,1.7)--(9.2,2.3) (9.2,2.8)--(9.2,3.1);
                     \draw[-] (9.8,0.9)--(9.8,1.2) (9.8,1.7)--(9.8,2.3) (9.8,2.8)--(9.8,3.1);
                     \draw (9,1.2) rectangle (10.0,1.7); 
                      \draw  (9.6,1.4)node{$S_2$};
                     \draw (9,2.3) rectangle (10.0,2.8); 
                      \draw  (9.6,2.5)node{$S_1$};
                    \draw[->] (9.2,0.9)  arc (180:360:0.3);
                    \draw[<-] (9.2,3.1)  arc (180:0:0.3);
\draw[color=white, very thick] (9.1,1)--(9.9,1) (9.1,2)--(9.9,2) (9.1,3)--(9.9,3)  ;
         \end{tikzpicture}
 
\caption{$K_3=K_1\sharp K_2$}
\label{proofT6}
\end{center}
\end{figure}

Next consider the former case (T3).  
Since $K_1$ is a profinite knot, $m_1$ and $m_2$ are both $0$.
By the above argument in case (T6),
we may assume that both $T_1$  and $T_2$  in Figure \ref{T3}
should start from $\creation$
(i.e. $\alpha_1=\beta_1=\creation$).
Define  $T$  as in Figure \ref{definition of T}.
A successive application of
commutativity of  profinite braids with $T$ shown in (T4)
and that of creations and annihilations with $T$ shown in Lemma \ref{CCAwithT}
lead the isotopy equivalence shown in Figure \ref{proofT3}.
Here $T_1'$ and $T_2'$ stand for emissions of the lowest $\creation$ from
$T_1$ and $T_2$.
We note that emission of $T$ in the figure represents $K_1$ and $K'_1$.
\begin{figure}[h]
\begin{tabular}{c}
  \begin{minipage}{0.5\hsize}
      \begin{center}
         \begin{tikzpicture}
                      \draw (0,0.3) rectangle (0.4,0.8);
                      \draw (0.2,0.5) node{$T$};
                      \draw[->] (0.2,0)--(0.2, 0.3)  ;
                      \draw[->] (0.2,0.8)--(0.2,1.1) ;
\draw  (0.7,0.5)node{$=$};
                      \draw[->] (1.0,-0.2)--(1.0,1.2);
                       \draw (1.0,-0.2) node{$\uparrow$};
                      \draw (1.2,0.5) rectangle (2.0,0.9);
                     \draw (1.6,0.7) node{$S_2$};
                     \draw[<-] (1.4,0.2)--(1.4,0.5);
                    \draw[<-] (1.4,0.9)--(1.4,1.2);
                     \draw[->] (1.4,0.2)  arc (180:360:0.2);
                     \draw[<-] (1.4,1.2)  arc (0:180:0.2);
                     \draw[->] (1.8,0.2)--(1.8,0.5);
                     \draw[->] (1.8,0.9)--(1.8,1.4);
\draw[color=white, very thick] (0.9,0.2)--(2.0,0.2) (0.9,1.2)--(2.0,1.2) ;
         \end{tikzpicture}
\caption{definition of $T$}
\label{definition of T}
     \end{center}
 \end{minipage}

 \begin{minipage}{0.5\hsize}
     \begin{center}
         \begin{tikzpicture}
                      \draw[->] (0,0)  arc (180:360:0.3);
                      \draw[->] (0,1.1)--(0,0) ;
                      \draw (0.4,0.3) rectangle (0.8,0.8);
                      \draw[->] (0.6,0)--(0.6,0.3) ;
                      \draw (0.6,0.5) node{$T$};
                      \draw[<-] (0.6,1.1)--(0.6,0.8) ;
                      \draw (-0.2,1.1) rectangle (0.8,1.6);
                      \draw (0.3,1.3) node{$T'_1$};
                      \draw[-] (-0.1,1.6)--(-0.1,3.4) (0.1,1.6)--(0.1,3.4) (0.7,1.6)--(0.7,3.4);
                  \draw[dotted] (0.2,2.5) --(0.6, 2.5) ;
                      \draw[->] (1.2,2.1)  arc (180:360:0.3);
                       \draw[<-] (1.2,2.1)--(1.2,2.4) ; 
                     \draw[->] (1.8,2.1)--(1.8,2.4) ;
                      \draw (1.0,2.4) rectangle (2.0,2.9);
                      \draw (1.5,2.6) node{$T'_2$};
                      \draw[-] (1.1,2.9)--(1.1,3.4) (1.2,2.9)--(1.2,3.4)
(1.9,2.9)--(1.9,3.4);
                  \draw[dotted] (1.3,3.1) --(1.8, 3.1) ;
                      \draw (-0.2,3.4) rectangle (2.0,4.4);
\draw  (2.5,2.2)node{$=$};
                      \draw[->] (3,2.1)  arc (180:360:0.3);
                      \draw[<-] (3,2.1)--(3,2.4) ;
                     \draw[->] (3.6,2.1)--(3.6,2.4) ;
                      \draw (2.8,2.4) rectangle (3.8,2.9);
                      \draw (3.3,2.6) node{$T'_1$};
                      \draw[-] (2.9,2.9)--(2.9,3.4) (3.1,2.9)--(3.1,3.4) (3.7,2.9)--(3.7,3.4);
                     \draw[->] (4.2,0)  arc (180:360:0.3);
                       \draw[<-] (4.2,0)--(4.2,1.1) ; 
                       \draw (4.6,0.3) rectangle (5,0.8);
                      \draw (4.8,0.5) node{$T$};
                      \draw[<-] (4.8,1.1)--(4.8,0.8) ;
                     \draw[->] (4.8,0)--(4.8,0.3) ;
                      \draw (4.0,1.1) rectangle (5.0,1.6);
                      \draw (4.5,1.3) node{$T'_2$};
                      \draw[-] (4.1,1.6)--(4.1,3.4) (4.2,1.6)--(4.2,3.4)
(4.9,1.6)--(4.9,3.4);
                  \draw[dotted] (4.3,2.5) --(4.8, 2.5) ;
                     \draw (2.8,3.4) rectangle (5.0,4.4);
\draw[color=white, very thick] (-0.1,0)--(4.9,0) (-0.1,0.9)--(4.9,0.9) (-0.1,2.1)--(5, 2.1)   (-0.1,3.2)--(5.,3.2) ;
         \end{tikzpicture}
\caption{$K_1\sharp K_2=K_1'\sharp K_2$}
\label{proofT3}
      \end{center}
  \end{minipage}
\end{tabular}
\end{figure}

Secondly we must prove that $K_1\sharp K_2$ is isotopic to $K_1\sharp K_2'$
if $K_2'$ is isotopic to $K_2$.
It can be proved in a completely same way to the above arguments.
Thus our proof is finally completed.

(2). Associativity, i.e. $(K_1\sharp K_2)\sharp K_3=K_1\sharp (K_2\sharp K_3)$,
is easy to see.
A proof of commutativity is illustrated in Figure \ref{proof of commutativity}.
We note that we use (T3) and (T5) in the first, the third and the fifth equalities
and we apply (T4) and Lemma \ref{CCAwithT} for the dashed boxes in
the second and the fourth equalities.

\begin{figure}[h]
\begin{center}
        \begin{tikzpicture}
                    \draw[->] (0,-0.1)  arc (180:360:0.3);
                     \draw[<-] (0,-0.1)--(0,0.3)  ;
                     \draw[->] (0.6,-0.1)--(0.6,0.3) ;
                  \draw (-0.1,0.3) rectangle (0.7,0.8);
    \draw  (0.3,0.5)node{$S_2$};
                     \draw[<-] (0,0.8)--(0,1.4)  ;
                     \draw[->]  (0.6,0.8)--(0.6,1.4) ;
                 \draw (-0.1,1.4) rectangle (0.7,1.9);
    \draw  (0.3,1.6)node{$S_1$};
                     \draw[-] (0,1.9)--(0,2.3) (0.6,1.9)--(0.6,2.3) ;
                    \draw[<-] (0,2.3)  arc (180:0:0.3);
\draw[color=white, very thick] (-0.1,-0.1)--(0.7,-0.1) (-0.1,1.1)--(0.7,1.1)  (-0.1,2.3)--(0.7,2.3) ;
\draw  (1.3,1.1)node{$=$};
                     \draw[-] (2,-0.1)--(2,1.4) ;
                    \draw[->] (2,-0.1)  arc (180:360:0.3);
                     \draw[-] (2.6,-0.1)--(2.6,1.0) ;
                    \draw[->] (2.6,1.0)  arc (180:0:0.1);
                     \draw[-] (2.8,1.0)--(2.8,0.9) (3,1.4)--(3,0.9);
                 \draw (2.7,0.4) rectangle (3.1,0.9);
    \draw  (2.9,.6)node{$S_2$};
                     \draw[-] (2.8,0.4)--(2.8,0.3) (3,0.4)--(3,0.3);
                    \draw[->] (2.8,0.3)  arc (180:360:0.1);
                 \draw[dotted] (2.5,0.1) rectangle (3.2,1.2);
                 \draw (1.9,1.4) rectangle (3.1,1.9);
    \draw  (2.5,1.6)node{$S_1$};
                     \draw[-] (2.0,1.9)--(2.0,2.3) (3.0,1.9)--(3.0,2.3) ;
                    \draw[<-] (2.0,2.3)  arc (180:0:0.5);
\draw[color=white, very thick] (1.9,-0.1)--(2.7,-0.1) (1.9,0.3)--(3.1,0.3) (1.9,1.0)--(3.1,1.0)
(1.9,1.3)--(3.1,1.3)  (1.9,2.3)--(3.1,2.3) ;
\draw  (3.8,1.1)node{$=$};
                     \draw[-] (4.5,-0.1)--(4.5,0.3) (4.5,0.8)--(4.5,2.3);
                     \draw[->] (4.5,-0.1)  arc (180:360:0.3);
                      \draw (4.4,0.3) rectangle (5.2,0.8);
                      \draw  (4.8,0.5)node{$S_1$};
                    \draw[-] (5.1,-0.1)--(5.1,0.3) (5.1,0.8)--(5.1,2.0) ;
                    \draw[->] (5.1,2.0)  arc (180:0:0.1);
                     \draw[-] (5.3,2.0)--(5.3,1.9) (5.5,2.3)--(5.5,1.9);
                 \draw (5.2,1.4) rectangle (5.6,1.9);
    \draw  (5.4,1.6)node{$S_2$};
                     \draw[-] (5.3,1.3)--(5.3,1.4) (5.5,1.3)--(5.5,1.4);
                    \draw[->] (5.3,1.3)  arc (180:360:0.1);
                 \draw[dotted] (5,1.1) rectangle (5.7,2.2);
                    \draw[<-] (4.5,2.3)  arc (180:0:0.5);
\draw[color=white, very thick] (4.4,-0.1)--(5.2,-0.1) (4.4,1.0)--(5.2,1.0) 
(4.4,1.3)--(5.6,1.3) (4.4,2.0)--(5.6, 2.0)  (4.4,2.3)--(5.6,2.3) ;
\draw  (6.2,1.1)node{$=$};
                     \draw (6.8,1.4) rectangle (7.2,1.9); 
                      \draw  (7.0,1.6)node{$S_1$};
                     \draw[->] (6.9,1.3)  arc (180:360:0.1);
                     \draw[-] (6.9,1.3)--(6.9,1.4) (6.9,1.9)--(6.9,2.3) ;
                     \draw[-] (7.1,1.3)--(7.1,1.4) (7.1,1.9)--(7.1,2) ;
                     \draw[->] (7.1,2.0)  arc (180:0:0.1);
                 \draw[densely dotted] (6.7,1.1) rectangle (7.4,2.2);
                     \draw[->] (7.3,-0.1)  arc (180:360:0.3);
                     \draw[-] (7.3,-0.1)--(7.3, 0.3) (7.3,0.8)--(7.3,2.0)  ;
                     \draw (7.2,0.3) rectangle (8.0,0.8);
                     \draw  (7.6,0.5)node{$S_2$};
                     \draw[-]  (7.9,-0.1)--(7.9, 0.3) (7.9,0.8)--(7.9,2.3);
                    \draw[->] (7.9,2.3)  arc (0:180:0.5);
\draw[color=white, very thick] (7.2,-0.1)--(8.0,-0.1) (7.2,1.0)--(8.0,1.0) (6.8,1.3)--(8.0,1.3) 
(6.8,2.0)--(8.0, 2.0)  (6.8,2.3)--(8.0,2.3) ;
\draw  (8.4,1.1)node{$=$};
                   \draw (8.8,.4) rectangle (9.2,.9); 
                      \draw  (9.0,.6)node{$S_1$};
                     \draw[->] (8.9,.3)  arc (180:360:0.1);
                     \draw[-] (8.9,.3)--(8.9,.4) (8.9,.9)--(8.9,1.4) (8.9,1.9)--(8.9,2.3);
                     \draw[-] (9.1,.3)--(9.1,.4) (9.1,.9)--(9.1,1) ;
                     \draw[->] (9.1,1.0)  arc (180:0:0.1);
                 \draw[densely dotted] (8.7,.1) rectangle (9.4,1.2);
                     \draw[->] (9.3,-0.1)  arc (180:360:0.3);
                     \draw[-] (9.3,-0.1)--(9.3,1.0)  ;
                     \draw (8.8,1.4) rectangle (10.0,1.9);
                     \draw  (9.4,1.6)node{$S_2$};
                     \draw[-]  (9.9,-0.1)--(9.9, 1.4) (9.9,1.9)--(9.9,2.3);
                    \draw[->] (9.9,2.3)  arc (0:180:0.5);
\draw[color=white, very thick] (8.7,-0.1)--(10.1,-0.1)  (8.7,.3)--(10.1,.3) 
(8.7,1.0)--(10.1, 1.0)  (8.7,1.3)--(10.1,1.3) (8.7,2.3)--(10.1,2.3) ;
\draw  (10.4,1.1)node{$=$};
                    \draw[->] (11.0,-0.1)  arc (180:360:0.3);
                     \draw[<-] (11.0,-0.1)--(11.0,0.3)  ;
                     \draw[->] (11.6,-0.1)--(11.6,0.3) ;
                  \draw (10.9,0.3) rectangle (11.7,0.8);
    \draw  (11.3,0.5)node{$S_1$};
                     \draw[<-] (11.0,0.8)--(11,1.4)  ;
                     \draw[->]  (11.6,0.8)--(11.6,1.4) ;
                 \draw (10.9,1.4) rectangle (11.7,1.9);
    \draw  (11.3,1.6)node{$S_2$};
                     \draw[-] (11.0,1.9)--(11.0,2.3) (11.6,1.9)--(11.6,2.3) ;
                    \draw[<-] (11.0,2.3)  arc (180:0:0.3);
\draw[color=white, very thick] (11-0.1,-0.1)--(11.7,-0.1) (11-0.1,1.1)--(11.7,1.1)  (11-0.1,2.3)--(11.7,2.3) ;
         \end{tikzpicture}
\caption{$K_1\sharp K_2 = K_2\sharp K_1$}
\label{proof of commutativity}
\end{center}
\end{figure}

To show that $\sharp$ is continuous,
we define  by $\widehat{\mathcal T}^{\rm seq}$
the set of finite consistent sequences of profinite fundamental tangles
and  by $\widehat{\mathcal K'}^{\rm seq}$
the set of  finite consistent sequences of profinite fundamental tangles $\gamma_n\cdots\gamma_2\cdot\gamma_1$
with  a single connected component and
with $(\gamma_n,\gamma_1)=(\opannihilation,\creation)$.
We note that the quotient set of  $\widehat{\mathcal T}^{\rm seq}$
by the equivalence of finite sequences of the moves (T1)-(T6)
is  equal to the set $\widehat{\mathcal T}$ of profinite tangles.
We also note that
 $\widehat{\mathcal K'}^{\rm seq}$ is projected onto 
its subset  $\widehat{\mathcal K}$ of profinite knots.
The set  $\widehat{\mathcal T}^{\rm seq}$ 
carries 
a structure of a topological space by 
the profinite topologies on $\widehat{B}_n$ (Note \ref{topology on profinite knots}).
It induces a subspace topology on $\widehat{\mathcal K'}^{\rm seq}$.
The map 
$\widehat{\mathcal T}^{\rm seq}\twoheadrightarrow\widehat{\mathcal T}$
is continuous, 
hence so the projection
$\widehat{\mathcal K'}^{\rm seq}\twoheadrightarrow\widehat{\mathcal K}$ is.
By the topology it is easy to see that
the map
$$
\sharp:\widehat{\mathcal K'}^{\rm seq}\times\widehat{\mathcal K'}^{\rm seq}\to\widehat{\mathcal K'}^{\rm seq}
$$
caused by \eqref{connected sum} is continuous.
Because the following diagram is commutative
$$
\begin{CD}
\widehat{\mathcal K'}^{\rm seq}\times \widehat{\mathcal K'}^{\rm seq} @>\sharp>>
\widehat{\mathcal K'}^{\rm seq} \\
@VVV    @VVV \\
\widehat{\mathcal K}\times \widehat{\mathcal K}@>\sharp>>
\widehat{\mathcal K}  \\
\end{CD}
$$
and the projection
$\widehat{\mathcal K'}^{\rm seq}\twoheadrightarrow\widehat{\mathcal K}$
is continuous,
the lower map is also continuous.

(3).
The first statement is obvious.
Let ${\mathcal K'}^{\rm seq}$ be the subset of 
$\widehat{\mathcal K'}^{\rm seq}$ which consists of
consistent finite sequences of \lq topological fundamental tangles', that is,
sequences of elements in $A$, $C$ and $B$
with $b_n\in B_n\subset {\widehat B}_n$.
There are a natural inclusion 
 ${\mathcal K'}^{\rm seq}\to\widehat{\mathcal K'}^{\rm seq}$
and a surjection
 ${\mathcal K'}^{\rm seq}\to{\mathcal K}$.
Since the map $B_n\to\widehat{B}_n$ is with dense image,
so the inclusion 
 ${\mathcal K'}^{\rm seq}\to\widehat{\mathcal K'}^{\rm seq}$ is.
Therefore the map 
${\mathcal K}\to\widehat{\mathcal K}$ is
also with dense image.
It is because the following diagram is commutative
$$
\begin{CD}
{\mathcal K'}^{\rm seq} @>>> \widehat{\mathcal K'}^{\rm seq} \\
@VVV    @VVV \\
{\mathcal K}@>>>\widehat{\mathcal K}  \\
\end{CD}
$$
and the projection
$\widehat{\mathcal K'}^{\rm seq}\twoheadrightarrow\widehat{\mathcal K}$
is continuous.
\end{proof}

The following two lemmas are required to prove Theorem \ref{topological monoid} (1).

\begin{lem}\label{transpose lemma}
For a profinite tangle $T$ with  $s(T)=t(T)=\uparrow$,
define its transpose $\opT$ (occasionally also denoted by ${}^tT$ ) by
$$\opT:=
a_{0,1}^{\opannihilation\downarrow}
\cdot(e_1^\downarrow\otimes T\otimes e_1^\downarrow)\cdot 
c_{1,0}^{\downarrow\opcreation}.
$$
Then we have
$$
\opT=a_{1,0}^{\downarrow\annihilation}
\cdot(e_1^\downarrow\otimes T\otimes e_1^\downarrow)\cdot 
c_{0,1}^{\creation\downarrow}.
$$
Similar claim holds for a profinite tangle $T$ with  $s(T)=t(T)=\downarrow$
by reversing all arrows.
\end{lem}

Figure \ref{opT} describes our lemma.
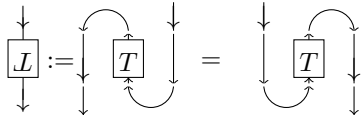
\begin{figure}[h]
\begin{center}
         \begin{tikzpicture}
                      \draw (0,0.3) rectangle (0.4,0.8);
                      \draw (0.2,0.5) node[rotate=180]{$T$};
                      \draw[-] (0.2,0) node{$\downarrow$} --(0.2, 0.3) (0.2,0.8)--(0.2,1.1)node{$\downarrow$} ;
\draw  (0.7,0.5)node{$:=$};
                      \draw[<-] (1.0,-0.2)--(1.0,0.9);
                       \draw (1.0,0.4) node{$\downarrow$};
                      \draw (1.4,0.3) rectangle (1.8,0.8);
                     \draw (1.6,0.5) node{$T$};
                     \draw[->] (1.6,0.2)--(1.6,0.3);
\draw[->] (1.6,0.8)--(1.6,0.9);
                     \draw[<-] (1.6,0.2)  arc (180:360:0.3);
                     \draw[->] (1.6,0.9)  arc (0:180:0.3);
                      \draw[<-] (2.2,0.2)--(2.2,1.3);
                       \draw (2.2,1.1) node{$\downarrow$};
\draw  (2.7,0.5)node{$=$};
                      \draw[<-] (3.4,0.2)--(3.4,1.3);
                       \draw (3.4,1.1) node{$\downarrow$};
                     \draw (3.8,0.3) rectangle (4.2,0.8);
                   \draw (4,0.5) node{$T$};
                     \draw[->] (4,0.2)--(4,0.3) ;
\draw[->](4,0.8)--(4,0.9);
                     \draw[<-] (4,0.2)  arc (360:180:0.3);
                     \draw[->] (4,0.9)  arc (180:0:0.3);
                      \draw[<-] (4.6,-0.2)--(4.6,0.9);
                       \draw (4.6,0.4) node{$\downarrow$};
\draw[color=white, very thick] (0.9,0.2)--(4.7,0.2) (0.9,0.9)--(4.7,0.9) ;
         \end{tikzpicture}
\caption{transpose of $T$}
\label{opT}
\end{center}
\end{figure}

\begin{proof}
A proof is depicted in Figure \ref{proof of transpose}.
We note that we apply (T6) in the first equality and (T2) and (T4) in the second equality.
\end{proof}
\begin{figure}[h]
\begin{center}
        \begin{tikzpicture}
                      \draw[<-] (-0.2,-0.4)--(-0.2,1.0) ;
                     \draw[->] (0.2,1.0)  arc (0:180:0.2);
                     \draw[<-] (0.2,0.2)  arc (180:360:0.2);
                      \draw (0,0.4) rectangle (0.4,0.8);
                      \draw (0.2,0.6) node{$T$};
                      \draw[->] (0.2,0.2)--(0.2, 0.4)  ;
                      \draw[->] (0.2,0.8)--(0.2,1.0) ;
                      \draw[<-] (0.6,0.2)--(0.6,1.6) ;
               \draw[color=white, very thick] (-0.3,0.2)--(0.7,0.2) (-0.3,1.0)--(0.7,1.0)  ;
\draw  (1.0,0.6)node{$=$};
                      \draw[<-] (1.4,-0.4)--(1.4,1.0) ;
                      \draw[<-] (1.4,1.0)--(1.8,1.4) ;
        \draw[color=white, line width=5pt]  (1.8, 1.0)--(1.4,1.4) ;
                     \draw[->] (1.8,1.0)--(1.4,1.4) ;
                     \draw[<-] (1.8,1.4)  arc (0:180:0.2);
                      \draw[<-] (1.8,0.2)--(2.2,-0.2) ;
        \draw[color=white, line width=5pt]  (1.8, -0.2)--(2.2,0.2) ;
                     \draw[<-] (1.8,-0.2)--(2.2,0.2) ;
                     \draw[->] (1.8,-0.2)  arc (180:360:0.2);
                      \draw (1.6,0.4) rectangle (2.0,0.8);
                      \draw (1.8,0.6) node{$T$};
                      \draw[->] (1.8,0.2)--(1.8, 0.4)  ;
                      \draw[->] (1.8,0.8)--(1.8,1.0) ;
                      \draw[<-] (2.2,0.2)--(2.2,1.6) ;
               \draw[color=white, thick] (1.3,-0.2)--(2.3,-0.2) (1.3,0.2)--(2.3,0.2) (1.3,1.0)--(2.3,1.0) (1.3,1.4)--(2.3,1.4) ;
\draw  (2.6,0.6)node{$=$};
                      \draw[<-] (3,0.2)--(3,1.6) ;
                     \draw[->] (3.4,1.0)  arc (180:0:0.2);
                     \draw[<-] (3.4,0.2)  arc (0:-180:0.2);
                      \draw (3.2,0.4) rectangle (3.6,0.8);
                      \draw (3.4,0.6) node{$T$};
                      \draw[->] (3.4,0.2)--(3.4, 0.4)  ;
                      \draw[->] (3.4,0.8)--(3.4,1.0) ;
                      \draw[<-] (3.8,-0.4)--(3.8,1.0) ;
               \draw[color=white, very thick](2.9,0.2)--(3.9,0.2) (2.9,1.0)--(3.9,1.0) ;
         \end{tikzpicture}
\caption{proof of Lemma \ref{transpose lemma}}
\label{proof of transpose}
\end{center}
\end{figure}

By (T5) and the above lemma, we see that ${}^{tt}T$ is isotopic to $T$.

\begin{lem}\label{CCAwithT}
For a profinite tangle $T$ with  $s(T)=t(T)=\uparrow$,
the equalities in Figure \ref{CAT} hold.
The same claim also holds for a profinite tangle $T$ with  $s(T)=t(T)=\downarrow$
by reversing all arrows.
\end{lem}

\begin{figure}[h]
\begin{tabular}{cccc}
\begin{minipage}{0.25\hsize}
\begin{center}
         \begin{tikzpicture}
                      \draw[<-] (0,0)--(0,1.1);
                     \draw (.4,0.3) rectangle (.8,0.8);
                     \draw (.6,0.5) node{$T$};
                     \draw[->] (.6,0.)--(.6,0.3) ;
                     \draw[->] (.6,0.8)--(.6,1.1);
                     \draw[->] (.6,1.1)  arc (0:180:0.3);
\draw  (1.1,0.5)node{$=$};
                     \draw (1.3,0.3) rectangle (1.7,0.8);
                     \draw (1.5,0.5) node[rotate=180]{$T$};
                     \draw[<-](1.5,0)--(1.5,0.3);
                     \draw[<-] (1.5,0.8)--(1.5,1.1);
                     \draw[<-] (1.5,1.1)  arc (180:0:0.3);
                     \draw[->] (2.1,0)--(2.1,1.1);
\draw[color=white,  thick]  (-0.1,1.1)--(2.2,1.1) ;
         \end{tikzpicture}
\end{center}
\end{minipage}
\begin{minipage}{0.25\hsize}
\begin{center}
         \begin{tikzpicture}
                     \draw (-0.2,0.3) rectangle (0.2,0.8);
                     \draw (.0,0.5) node{$T$};
                     \draw[->](.0,0)--(.0,0.3);
                     \draw[->] (.0,0.8)--(.0,1.1);
                     \draw[<-] (0,0)  arc (180:360:0.3);
                      \draw[<-] (0.6,0)--(0.6,1.1);
\draw  (1.,0.2)node{$=$};
                   \draw[->] (1.5,0)--(1.5,1.1);  
                   \draw (1.9,0.3) rectangle (2.3,0.8);
                     \draw (2.1,0.5) node[rotate=180]{$T$};
                     \draw[<-](2.1,0)--(2.1,0.3);
                     \draw[<-] (2.1,0.8)--(2.1,1.1);
                     \draw[<-] (1.5,0)  arc (180:360:0.3);
\draw[color=white,  thick] (-0.1,0)--(2.2,0) ;
         \end{tikzpicture}
\end{center}
\end{minipage}
\begin{minipage}{0.25\hsize}
\begin{center}
         \begin{tikzpicture}
                     \draw (-.2,0.3) rectangle (.2,0.8);
                     \draw (.0,0.5) node{$T$};
                     \draw[->] (.0,0.)--(.0,0.3) ;
                     \draw[->] (.0,0.8)--(.0,1.1);
                     \draw[<-] (.6,1.1)  arc (0:180:0.3);
                      \draw[<-] (0.6,0)--(0.6,1.1);
\draw  (1.1,0.5)node{$=$};
                    \draw[->] (1.5,0)--(1.5,1.1); 
                    \draw (1.9,0.3) rectangle (2.3,0.8);
                     \draw (2.1,0.5) node[rotate=180]{$T$};
                     \draw[<-](2.1,0)--(2.1,0.3);
                     \draw[<-] (2.1,0.8)--(2.1,1.1);
                     \draw[->] (1.5,1.1)  arc (180:0:0.3);
\draw[color=white,  thick]  (-0.1,1.1)--(2.2,1.1) ;
         \end{tikzpicture}
\end{center}
\end{minipage}
\begin{minipage}{0.25\hsize}
\begin{center}
         \begin{tikzpicture}
                      \draw[<-] (0,0)--(0,1.1);
                     \draw (0.4,0.3) rectangle (0.8,0.8);
                     \draw (0.6,0.5) node{$T$};
                     \draw[->](.6,0)--(.6,0.3);
                     \draw[->] (.6,0.8)--(.6,1.1);
                     \draw[->] (0,0)  arc (180:360:0.3);
\draw  (1.,0.2)node{$=$};
                   \draw (1.3,0.3) rectangle (1.7,0.8);
                     \draw (1.5,0.5) node[rotate=180]{$T$};
                     \draw[<-](1.5,0)--(1.5,0.3);
                     \draw[<-] (1.5,0.8)--(1.5,1.1);
                     \draw[->] (1.5,0)  arc (180:360:0.3);
                    \draw[->] (2.1,0)--(2.1,1.1);  
\draw[color=white,  thick] (-0.1,0)--(2.2,0) ;
         \end{tikzpicture}
\end{center}
\end{minipage}
\end{tabular}
\caption{Creation and annihilation commute with $T$.}
\label{CAT}
\end{figure}
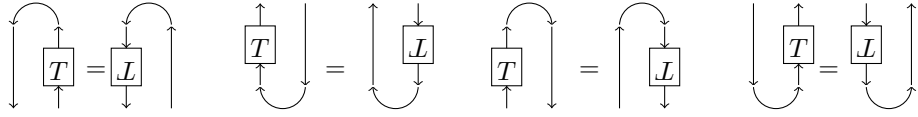

\begin{proof}
It can be checked by direct computation
using Lemma \ref{transpose lemma}.
\end{proof}

In knot theory, the so-called Alexander-Markov's theorem
is fundamental  on constructions of knot invariants.
The theorem is to translate knots and links into purely algebraic objects:

\begin{thm}[Alexander-Markov's theorem]
There is a one-to-one correspondence 
$$
{\mathcal L}\longleftrightarrow \sqcup_{n}B_n/ \sim_{\tiny\text{M}}
$$
between
the set  ${\mathcal L}$ of isotopy classes of oriented links and
the (disjoint) union $\sqcup_{n}B_n$ of  braids groups 
modulo the equivalence $\sim_{\tiny\text{M}}$ given by the following Markov moves
$$
\text{(M1). }  b_1\cdot b_2 \sim_{\tiny\text{M}} b_2\cdot b_1 \quad (b_1, b_2\in B_n), \qquad
\text{(M2). } 
b\in B_n\sim_{\tiny\text{M}} 
b\sigma_n^{\pm 1} \in B_{n+1} \quad (b\in B_{n}) 
$$
\end{thm}
For more on the theorem, consult \cite{CDM, O} for example.
The question below is to ask 
a validity of 
profinite analogue  of Alexander-Markov's theorem.

\begin{prob}
Is there
a \lq profinite analogue' of the Alexander-Markov's theorem which holds
for the set $\widehat{\mathcal L}$ of isotopy classes of profinite links ?
\end{prob}

There are  
several proofs of Alexander-Markov's theorem 
for topological links (\cite{Bi, Tr, V, Y} etc).
But they look heavily based on a certain finiteness property,
which  we (at least the author)  may not expect the validity
for profinite links.

\subsection{Pro-$l$ knots}\label{pro-l knots}
Pro-$l$ tangle diagrams, pro-$l$ knot diagrams and isotopy among them
which are pro-$l$ analogues of our  corresponding notions
given in the previous subsection,
are introduced in Definition \ref{definition of pro-l knots}.
A natural map
from profinite tangles (resp. profinite knots) to pro-$l$ tangles (resp. pro-$l$ knots)
is constructed in Proposition \ref{from profinite to pro-l}.
Proalgebraic tangles and proalgebraic knots are recalled in Definition \ref{definition of proalgebraic knots}.
A natural map
from pro-$l$ tangles (resp. pro-$l$ knots) to proalgebraic tangles (resp. proalgebraic knots)
is given in Proposition \ref{from pro-l to proalgebraic}.
It is explained
that the Kontsevich invariant factors through these natural maps  in Remark \ref{Kontsevich factorization}.
Our discussion in this subsection will serve for a proof of non-triviality of the
Galois action constructed in \S \ref{absolute Galois action}.

Let $l$ be a prime. We may include $l=2$.

\begin{nota}
A topological group $G$ is called a {\it pro-$l$ group}
if it is a projective limit $\varprojlim G_i$ of a projective system of finite $l$-groups $\{G_i\}_{i\in I}$.
For a discrete group $\Gamma$,
its {\it pro-$l$ completion} $\widehat{\Gamma}^l$ is the pro-$l$ group
defined by the projective limit
$$
\widehat{\Gamma}^l=\varprojlim \Gamma/N
$$
where $N$ runs  over all normal subgroups of $\Gamma$ with finite indices of power of $l$.
\end{nota}

For more on pro-$l$ groups, consult \cite{RZ} for example.
We note there is a natural homomorphism $\Gamma\to\widehat{\Gamma}^l$.

\begin{nota}
Let $n\geqslant 2 $.
Let ${\widehat P}_n\rtimes B_n$ be the semi-direct product of ${\widehat P}_n$
and $B_n$ with respect to the $B_n$-action on  ${\widehat P}_n$ given by $p\mapsto bpb^{-1}$ 
($p\in\widehat{P}_n$ and $b\in B_n$).   
consider the inclusion $P_n\hookrightarrow {\widehat P}_n\rtimes B_n$
sending $p\mapsto (p,p^{-1})$. 
Then it is easy to see 
the homomorphism sending $(p,b)\mapsto pb$
yields an isomorphism:
\begin{equation*}
({\widehat P}_n\rtimes B_n) / P_n \simeq \widehat{B}_n.
\end{equation*}
\end{nota}

\begin{defn}
(1). The {\it pro-$l$ pure braid group} ${\widehat P}_n^l$ is the pro-$l$ completion of   $P_n$.

(2). The {\it pro-$(l)$ braid group} ${\widehat B}_n^{(l)}$ is defined to be the induced quotient
$$
{\widehat B}_n^{(l)}:=
(\widehat{P}_n^{l}\rtimes B_n) / P_n .
$$ 
\end{defn}

We encode 
a topological group structure on ${\widehat B}_n^{(l)}$
by the pro-$l$ topology on  ${\widehat P}_n^l$ and the discrete topology on $B_n$.
We note that this  ${\widehat B}_n^{(l)}$  appears also in \cite{LS}.

\begin{rem}
(1).
Our  ${\widehat B}_n^{(l)}$ is different from the pro-$l$ completion ${\widehat B}_n^{l}$ of $B_n$.

(2).
There is an exact sequence:
$$
1\to {\widehat P}_n^l\to {\widehat B}_n^{(l)} \to {\frak S}_n\to 1.
$$

(3).
There are natural group homomorphisms:
\begin{equation}\label{two natural group homomorphisms}
B_n\to \widehat{B}_n\to {\widehat B}_n^{(l)}.
\end{equation}
\end{rem}

The map \eqref{two natural group homomorphisms} is induced from 
$P_n\to \widehat{P}_n\to {\widehat P}_n^{l}$.

\begin{defn}\label{definition of pro-l knots}
(1).
A {\it pro-$l$ tangle diagram} means a consistent finite sequence of
{\it fundamental pro-$l$ tangle diagrams}, which are elements in $A$, $C$ 
(in Definition \ref{definition of fundamental profinite tangle}) or
$$
\widehat{B}^l:=\left\{b_n^\epsilon \bigm| b_n^\epsilon=
\left(b_n,\epsilon=(\epsilon_i)_{i=1}^{n}\right)
\in \widehat{B}_n^{(l)}\times \{\uparrow, \downarrow\}^{n}, n=1,2,3,4,\dots
\right\}.
$$
A {\it pro-$l$ knot diagram} means a pro-$l$ tangle diagram without endpoints
(their sources and targets are both empty)
and with a single connected component.

(2).
Two pro-$l$ tangle diagrams $T_1$ and $T_2$ are said to be {\it isotopic}
if they are related by a finite number of the moves 
replacing profinite  tangles and profinite braids by pro-$l$ tangles and elements in ${\widehat B}^l$
in (T1)-(T6) and $c\in\widehat{\mathbb Z}$ by $c\in{\mathbb Z}_l$ in (T6).
\footnote{
We note that, for $\sigma_{i}\in {\widehat B}_N^{(l)}$ and $c\in{\mathbb Z}_l$,
the power $\sigma_i^{2c+1}$ makes sense in ${\widehat B}_{N}^{(l)}$
because, by $\sigma_i^2\in \widehat{P}_N^l$, we have  $\sigma_i^{2c}\in {\widehat P}_N^l$.
}

(3). A {\it pro-$l$ tangle} (resp. {\it pro-$l$}) means an isotopic class of pro-$l$ tangle  (resp. knot) diagram.
We denote the set of  pro-$l$ tangles by $\widehat{\mathcal T}^l$
and the set of pro-$l$ knots by $\widehat{\mathcal K}^l$.
\end{defn}

Both $\widehat{\mathcal T}^l$ and $\widehat{\mathcal K}^l$ carry a structure of topological space by the method in
Note \ref{topology on profinite knots}. 

\begin{prop}\label{from profinite to pro-l}
(1).
The set $\widehat{\mathcal K}^l$ forms a topological monoid with respect to 
the connected sum.

(2).
There are continuous maps:
\begin{align}
\widehat{\mathcal T} &\to \widehat{\mathcal T}^l,  \label{factorization0} \\
\widehat{\mathcal K} &\to \widehat{\mathcal K}^l. \label{factorization1}
\end{align}

(3).
The map \eqref{factorization1} is monoid homomorphisms.
The image of its composition with $h$ in Theorem \ref{profinite realization theorem}.(2)
\begin{equation}\label{factorization10}
h_l:{\mathcal K} \overset{h}{\to} \widehat{\mathcal K} \to \widehat{\mathcal K}^l. 
\end{equation}
is with dense image
in $\widehat{\mathcal K}^l$.
\end{prop}

\begin{proof}
(1).
It is obtained by the  same arguments to the proof of Theorem \ref{topological monoid}.

(2).
The map \eqref{factorization0} is induced from the second map in \eqref{two natural group homomorphisms},
whose continuity implies ours.
It preserves  each connected component, which yields the map \eqref{factorization1}.

(3).
To see that they form homomorphisms are immediate.
The density can be proved by the same arguments to the proof of Theorem \ref{topological monoid}.
\end{proof}


\begin{defn}
Let $R$ be a commutative ring.

(1).
Let  $I$ be the two-sided ideal of the group algebra ${R}[B_n]$ of $B_n$
generated by $\sigma_i-\sigma_i^{-1}$ ($1\leqslant i\leqslant n-1$).
The topological $R$-algebra $\widehat{{R}[B_n]}$ of {\it  proalgebraic braids}
means its completion 
with respect to the $I$-adic  filtration, i.e.
$$
\widehat{{R}[B_n]}:=\underset{N}{\varprojlim} \ {R}[B_n]/I^N.
$$

(2).
Put $I_0:=I\cap{R}[P_n]$. Then $I_0$ is the augmentation ideal of ${R}[P_n]$
(cf. \cite{KT98}).
The topological $R$-algebra $\widehat{{R}[P_n]}$ of {\it  proalgebraic pure braids}
means its completion 
$$
\widehat{{R}[P_n]}:=\underset{N}{\varprojlim} \ {R}[P_n]/I_0^N.
$$
It is a subalgebra of $\widehat{{R}[B_n]}$.
\end{defn}

It is direct to see that 
both algebras naturally equip structures of co-commutative Hopf algebras.

\begin{rem}
(i). There is a short exact sequence
$$
0\to \widehat{{R}[P_n]} \to \widehat{{R}[B_n]} \to R[{\frak S_n}]\to 0.
$$

(ii). We remark that
(the group-like part of) $\widehat{{R}[P_n]}$ is the unipotent (Malcev) completion  of $P_n$ and
(the group-like part of) $\widehat{{R}[B_n]}$ is Hain's \cite{H} relative completion 
of $B_n$ with respect to the natural projection $B_n\to \frak S_n$.

(iii). The natural morphisms $P_n\to R[P_n]$ and $B_n\to R[B_n]$ yield injections
\begin{equation}\label{injections for braids}
P_n\hookrightarrow\widehat{{R}[P_n]} \quad\text{and}\quad B_n\hookrightarrow\widehat{{R}[B_n]}
\end{equation}
(cf. \cite{CDM} \S 12).
\end{rem}

Since $\widehat{P}_n^{l}$ is a pro-$l$ group
and $I_0$ is the augmentation ideal,
we have a natural continuous homomorphism
$\widehat{P}_n^{l}\to \widehat{{\mathbb Q}_l[P_n]}$ (cf. \cite{CDM} for example).

\begin{prop}
There is a natural continuous group homomorphism
\begin{equation}\label{map to proalgebraic braids}
\widehat{B}_n^{(l)}\to \widehat{{\mathbb Q}_l[B_n]}.
\end{equation}
\end{prop}

\begin{proof}
It can be directly checked that the map induced from
the above $\widehat{P}_n^{l}\to \widehat{{\mathbb Q}_l[P_n]}$
($\subset \widehat{{\mathbb Q}_l[B_n]} $)
and the natural map $B_n\hookrightarrow {\mathbb Q}_l[B_n]\to \widehat{{\mathbb Q}_l[B_n]} $
holds the property.
\end{proof}

Next we discuss the corresponding notions in tangles and knots settings.
The following notions go back to the idea of Vassiliev.

\begin{defn}[\cite{KT98}]
\label{definition of proalgebraic knots}
Let $R$ be a commutative ring.
Let ${\mathcal T}_{\epsilon,\epsilon'}$  be the full set of isotopy classes of oriented tangles with
type $({\epsilon,\epsilon'})$.

(1).
Let ${R}[{\mathcal T}_{\epsilon,\epsilon'}]$ be the free $R$-module of finite formal sums
of elements of ${\mathcal T}_{\epsilon,\epsilon'}$.
A singular oriented tangle 
\footnote{
It is an \lq oriented tangle' 
which  is allowed to have  a finite number of transversal double points (see \cite{KT98} for precise).
}
determines an element of ${R}[{\mathcal T}_{\epsilon,\epsilon'}]$ by the desingularization
of each double point
by the following relation
$$
\diaProjected=\diaCrossP -\diaCrossN.
$$
Let ${\mathcal T}_n$ ($n\geqslant 0$) be the $R$-submodule of $R[{\mathcal T}_{\epsilon,\epsilon'}]$ generated by
all singular oriented tangles with type $(\epsilon,\epsilon')$
and with $n$ double points.
The descending filtration $\{{\mathcal T}_n\}_{n\geqslant 0}$
is called the {\it singular filtration}.
The topological $R$-module $\widehat{{R}[{\mathcal T}_{\epsilon,\epsilon'}]}$ of {\it proalgebraic tangles} means its completion
with respect to the singular filtration:
$$
\widehat{{R}[{\mathcal T}_{\epsilon,\epsilon'}]}:=
\underset{N}{\varprojlim} \ {R}[{\mathcal T}_{\epsilon,\epsilon'}]/{\mathcal T}_N.
$$

(2). 
Let ${R}[{\mathcal K}]$ be the  $R$-submodule of ${R}[{\mathcal T}_{\emptyset,\emptyset}]$ generated by
elements of ${\mathcal K}$.
By Proposition \ref{knot forms a monoid},
it forms a commutative $R$-algebra.
Put ${\mathcal K}_n:={\mathcal T}_n\cap{R}[{\mathcal K}]$ ($n\geqslant 0$).
Then ${\mathcal K}_n$ forms an ideal of $R[{\mathcal K}]$
and the descending filtration $\{{\mathcal K}_n\}_{n\geqslant 0}$
is called the {\it singular knot filtration} (cf. loc.cit.).
The topological commutative $R$-algebra $\widehat{{R}[{\mathcal K}]}$ of {\it proalgebraic knots} means its completion
with respect to the singular knot filtration:
$$
\widehat{{R}[{\mathcal K}]}:=
\underset{N}{\varprojlim} \ {R}[{\mathcal K}]/{\mathcal K}_N.
$$
Actually it equips a structure of co-commutative and commutative Hopf algebra.
\end{defn}

The maps below are  tangle and knot analogues of the map \eqref{map to proalgebraic braids}.

\begin{prop}\label{from pro-l to proalgebraic}
(1).
There are continuous maps:
\begin{align}
\widehat{\mathcal T}^l\to \widehat{{\mathbb Q_l}[{\mathcal T}]}, \label{factorization20} \\
\widehat{\mathcal K}^l\to \widehat{{\mathbb Q_l}[{\mathcal K}]}. \label{factorization2}
\end{align}

(2).
The map \eqref{factorization2} is a continuous monoid homomorphism
and its image lies on the set  $\widehat{{\mathbb Q_l}[{\mathcal K}]}^\times$
of invertible elements.
\end{prop}

\begin{proof}
(1).
Since an element $b_n^\epsilon \in B_n\times \{\uparrow, \downarrow\}^n$
($n\geqslant 1$),
a braid $b_n\in B_n$ with an orientation $\epsilon\in  \{\uparrow, \downarrow\}^n$
(namely its  source),
is naturally regarded as a special type of an oriented tangle,
each orientation $\epsilon$ yields a natural inclusion
$$
{\mathbb Q_l}[B_n]\hookrightarrow {\mathbb Q_l}[{\mathcal T}].
$$
On the embedding, we have
$
{\mathcal T}_m \cap  {\mathbb Q_l}[B_n]= I^m
$
for $m\geqslant 0$.
Therefore the above map and the map \eqref{map to proalgebraic braids} induce
$$
\widehat{B}_n^{(l)}\to \  {\mathbb Q_l}[{\mathcal T}]/{\mathcal T}_m.
$$
Hence it determines the map of sets
\begin{equation}\label{B to T}
\widehat{B}^l\to {\mathbb Q_l}[{\mathcal T}]/{\mathcal T}_m.
\end{equation}
We also have the natural maps of sets 
\begin{equation}\label{A,C to T}
A\to {\mathbb Q_l}[{\mathcal T}]/{\mathcal T}_m \  \text{ and } \ 
C\to {\mathbb Q_l}[{\mathcal T}]/{\mathcal T}_m.
\end{equation}

As is described in the proof of Theorem \ref{profinite realization theorem},
the set $\mathcal T$ of (topological) oriented tangles is described by 
the set of consistent finite sequences of elements of $A$, $B$ and $C$
modulo the (discrete) Turaev moves.
By ${\mathbb Q}_l$-linearly extending the description,
we obtain the same description of  ${\mathbb Q_l}[{\mathcal T}]$.
Since our three maps \eqref{B to T} and \eqref{A,C to T}
are consistent with the moves, we obtain
$$
\widehat{\mathcal T}^l\to \  {\mathbb Q_l}[{\mathcal T}]/{\mathcal T}_m.
$$
(Again we note that, for $\sigma_{i}\in {B}_N$ and $c\in{\mathbb Z}_l$,
the power $\sigma_i^{c}$ makes sense in ${\mathbb Q_l}[B_N]/I^m$
by the formula
$$
\sigma_i^{c}:= \exp\{\frac{c}{2}\log \sigma_i^2\}
$$
when $l\neq 2$ or when $l=2$ and $c\in 2{\mathbb Z}_2$, and
$$
\sigma_i^{c}:=\sigma_i\cdot \exp\{\frac{c-1}{2}\log \sigma_i^2\}
$$
when $l=2$ and $c\not\in 2{\mathbb Z}_2$.
Here $\exp$ and $\log$ are defined by the usual Taylor expansions.
The RHS is well-defined by $\sigma_i^2-1\in I$.)

It yields the map \eqref{factorization20} which is continuous.
Since this map preserves each connected component, 
the map \eqref{factorization2} is also obtained.

(2).
It is immediate to see that it forms a continuous homomorphism.

Each oriented knot, an element of ${\mathcal K}$, is congruent to the unit,
the trivial knot $\orientedcircle\in {\mathbb Q_l}[{\mathcal K}]$, modulo ${\mathcal K}_1$,
because any knot can be untied by a finite times of changing crossings
(consult for unknotting number, say, in \cite{CDM}).
Therefore the image of $h_l({\mathcal K})$ ($\subset \widehat{\mathcal K}^l$)
is contained in the subspace
$\orientedcircle +{\mathcal K}_1\cdot\widehat{{\mathbb Q_l}[{\mathcal K}]}$.
Hence the image of $\widehat{\mathcal K}^l$ should lie on the subspace.
It is
because the subspace is open in $\widehat{{\mathbb Q_l}[{\mathcal K}]}$,
our map \eqref{factorization2} is continuous as shown above and
$h_l({\mathcal K})$ is dense in $\widehat{\mathcal K}_l$ by Proposition \ref{from profinite to pro-l}.(3).
All elements of the subspace are invertible
because it is known that the quotient $\widehat{{\mathbb Q_l}[{\mathcal K}]}/ {\mathcal K}_1$
is $1$-dimensional and generated by $\orientedcircle$.
Thus the claim is obtained.
\end{proof}

The author is not aware if our above two maps are injective or not.

\begin{prop}\label{from profinite to pro-algebraic}
(1). For each prime $l$, there are continuous maps:
\begin{align}
\widehat{\mathcal T}\to \widehat{{\mathbb Q_l}[{\mathcal T}]}, \label{factorization30} \\
\widehat{\mathcal K}\to \widehat{{\mathbb Q_l}[{\mathcal K}]}. \label{factorization3}
\end{align}

(2).
The map \eqref{factorization3} is a continuous monoid homomorphism
and its image lies on the set  $\widehat{{\mathbb Q_l}[{\mathcal K}]}^\times$
of  invertible elements.

(3).  The image of the composition of \eqref{factorization3}  with $h$ in
Theorem \ref{profinite realization theorem}.(2)
\begin{equation}\label{from K to Q_lK}
{\mathcal K} \overset{h}{\to} \widehat{\mathcal K} \to \widehat{{\mathbb Q}_l[{\mathcal K}]}. 
\end{equation}
lies on the rational invertible part $\widehat{{\mathbb Q}[{\mathcal K}]}^\times$
($\subset \widehat{{\mathbb Q}_l[{\mathcal K}]}$).
Furthermore the resulting morphism 
\begin{equation}\label{from K to QK}
i:{\mathcal K}\to \widehat{{\mathbb Q}[{\mathcal K}]}
\end{equation}
is independent of a prime $l$.
\end{prop}

\begin{proof}
(1) and (2) follow from
Proposition \ref{from profinite to pro-l} and \ref{from pro-l to proalgebraic}.
Our claim (3) follows from that
the map \eqref{from K to Q_lK} is induced from the natural inclusion
${\mathcal K}\hookrightarrow{R[{\mathcal K}]}$
with $R=\mathbb Q$.
\end{proof}

Relating to Problem \ref{profinite knot invariant},
the following is obtained as an application of 
Proposition \ref{from profinite to pro-algebraic}.

\begin{rem}\label{finite-type remark}
(1).
Finite type knot invariants (resp. and their projective limits)
valued on $\mathbb Q$
are knot invariants which factor through
${\mathcal K}\to \widehat{{\mathbb Q}[{\mathcal K}]}\to 
{\mathbb Q}[{\mathcal K}]/{\mathcal K}_N$
for some $N$ (resp. ${\mathcal K}\to \widehat{{\mathbb Z}[{\mathcal K}]}$).
By using the map \eqref{from K to Q_lK},
we can extend all finite type knot invariants and their projective limits,
such as each coefficient of the Jones polynomial substituting $e^h$
(cf. \cite{CDM}),
into profinite knot invariants valued on $\mathbb Q_l$ for each prime $l$.
It is easy to see that the same holds for tangle invariants.

(2). The linking number is 
an invariant of two components links which values on $\mathbb Z$ and
which is known to be of finite type (\cite{CDM}, etc).
Thus it can be extended to an invariant of two components profinite links
which values on $\mathbb Q_l$, actually on $\mathbb Z_l$,
for each prime $l$.
The Milnor $\bar\mu$-invariant  \cite{Mi1,Mi2} is an invariant of links
which is known as
a higher generalization of the linking number. 
It is defined on $\mathbb Z$ but it has a certain indeterminacy.
The author is not sure if it can be also extended to an invariant of profinite links, but
he expects that the works \cite{HL, Ba2, Lin}
concerning its associated invariant of string links
would help to detect that.
\end{rem}

As a complementation of Remark \ref{perfect remark}, we have

\begin{rem}\label{Kontsevich factorization}
The Kontsevich invariant $Z:{\mathcal K}\to\widehat{CD}$
is a knot invariant which is a composition of $i$ with
the isomorphism $\widehat{{\mathbb Q}[{\mathcal K}]}\simeq\widehat{CD}$
constructed in \cite{K}.
Here $\widehat{CD}$ stands for the $\mathbb Q$-vector space (completed by degree)
of chord diagrams
modulo 4T- and FI-relations (consult also \cite{CDM, O}).
The invariant is conjectured to be perfect, i.e., the map $Z$ is conjectured to be injective
(cf. \cite{O-Prob} Conjecture 3.2).
The conjecture is equivalent to saying that  the above map $i$ in \eqref{from K to QK}
is injective.
Hence it naturally leads us to conjecture 
that  $h:{\mathcal K} \to \widehat{\mathcal K}$ 
is injective (Conjecture \ref{injectivity on T})
because the non-injectivity of $h$  imply  the non-injectivity of $i$,
that is, the non-perfectness of
the Kontsevich invariant, by the above proposition.
We may also say that Conjecture \ref{injectivity on T} is an assertion on a knot analogue of
the injectivity of the maps \eqref{injections for braids}.
\end{rem}

\subsection{Action of the absolute Galois group}\label{absolute Galois action}
The group $\mathrm{Frac}\widehat{\mathcal K}$ of profinite knots is introduced as
the group of fraction of the topological monoid  $\widehat{\mathcal K}$
in Definition \ref{definition of GK}
and its basic property is shown in Theorem \ref{theorem of GK}.
A continuous action of the profinite Grothendieck-Teichm\"{u}ler group  $\widehat{GT}$ 
(cf. Definition \ref{definition of GT}) on $\mathrm{Frac}\widehat{\mathcal K}$
is established
in Definition \ref{GT-action on profinite knots}-Theorem \ref{GT-action theorem}. 
As a result of our construction, 
an action of
the absolute Galois group $G_{\mathbb Q}$ of the rational number field 
on $\mathrm{Frac}\widehat{\mathcal K}$
is obtained (Theorem \ref{Galois representation theorem}).
We post several projects and problems on this Galois representation in the end.

\begin{defn}\label{definition of GK}
The {\it group of profinite knots} $\mathrm{Frac}\widehat{\mathcal K}$ is defined to be the group of fraction
of the monoid  $\widehat{\mathcal K}$, i.e.,
the quotient space of $\widehat{\mathcal K}^2$ by the equivalent relations
$(r,s)\approx (r',s')$ if  $r\sharp s'\sharp t\sim r'\sharp s\sharp t$
for some profinite knot $t$,
i.e. $r\sharp s'\sharp t=r'\sharp s\sharp t$
holds in $\widehat{\mathcal K}$.
Occasionally we denote the equivalent class  $[(r,s)]$ by $\frac{r}{s}$.
\end{defn}

For $[(r_1,s_1)]$ and $[(r_2,s_2)]\in \mathrm{Frac}\widehat{\mathcal K}$,
define its product by
$$
[(r_1,s_1)]\sharp [(r_2,s_2)]:=
[(r_1\sharp r_2,s_1\sharp s_2)]\in \mathrm{Frac}\widehat{\mathcal K},
\qquad
\text{i.e.}\qquad
\frac{r_1}{s_1}\sharp \frac{r_2}{s_2}=\frac{r_1\sharp r_2}{s_1\sharp s_2}
\in \mathrm{Frac}\widehat{\mathcal K}.
$$
We encode $\mathrm{Frac}\widehat{\mathcal K}$ with the quotient topology of
$\widehat{\mathcal K}^2$. 

\begin{thm}\label{theorem of GK}
(1).
The product $\sharp$ is well-defined on  $\mathrm{Frac}\widehat{\mathcal K}$.
The set $\mathrm{Frac}\widehat{\mathcal K}$ forms a topological commutative group.

(2). 
It is a non-trivial group. Actually it is an infinite group.
\end{thm}


\begin{proof}
(1).
It is easy to see that $\sharp$ is well-defined and
$\mathrm{Frac}\widehat{\mathcal K}$ forms a commutative group
with unit
$$e=(\orientedcircle, \orientedcircle)$$
by Theorem \ref{topological monoid}.

Consider the  commutative diagram
$$
\begin{CD}
\widehat{\mathcal K}^2\times \widehat{\mathcal K}^2
@>\sharp>> \widehat{\mathcal K}^2 \\
@VVV    @VVV \\
\mathrm{Frac}\widehat{\mathcal K}\times \mathrm{Frac}\widehat{\mathcal K}
@>\sharp>> \mathrm{Frac}\widehat{\mathcal K}. \\
\end{CD}
$$
Since the upper map is continuous by Theorem \ref{topological monoid}
and the surjection
$\widehat{\mathcal K}^2\to \mathrm{Frac}\widehat{\mathcal K}$ is continuous by definition,
it follows that the map $\sharp$ is continuous.

Let $\tau: \widehat{\mathcal K}^2\to\widehat{\mathcal K}^2$
be the switch map sending $(r,s)\mapsto (s,r)$.
It is easy to see that it is continuous and
it induces the inverse map on $\mathrm{Frac}\widehat{\mathcal K}$.
Then by the commutative diagram
$$
\begin{CD}
\widehat{\mathcal K}^2
@>\tau>> \widehat{\mathcal K}^2 \\
@VVV    @VVV \\
\mathrm{Frac}\widehat{\mathcal K}@>>> \mathrm{Frac}\widehat{\mathcal K}, \\
\end{CD}
$$
the inverse map is also continuous.

(2).
By Proposition \ref{from profinite to pro-l} and \ref{from pro-l to proalgebraic},
there is a continuous monoid homomorphism
\begin{equation}\label{K to Kl to QlK}
\widehat{\mathcal K} \to \widehat{\mathcal K}^l\to \widehat{{\mathbb Q_l}[{\mathcal K}]}
\end{equation}
for a prime $l$.
By Proposition \ref{from pro-l to proalgebraic},
the image lies on $\widehat{{\mathbb Q_l}[{\mathcal K}]}^\times$.
Whence it induces a continuous group homomorphism
\begin{equation}\label{GK to QlK}
\mathrm{Frac}\widehat{\mathcal K} \to\widehat{{\mathbb Q_l}[{\mathcal K}]}^\times.
\end{equation}
Thus it is enough to show that the image of the composition of the maps \eqref{K to Kl to QlK}
and $h:{\mathcal K}\to \widehat{\mathcal K}$
is infinite set.
The claim is obvious because 
this map is equal to $i$ in \eqref{from K to QK} and
the Kontsevich invariant takes infinite number of values (cf. Remark \ref{Kontsevich factorization})
\end{proof}


A reason why we introduce  $\mathrm{Frac}\widehat{\mathcal K}$ is that
we need to treat the inverse of  $\Lambda_f$ in Figure \ref{error term}
when we let $\widehat{GT}$ act on profinite knots (cf. Definition \ref{GT-action on profinite knots}).

We note that the natural morphism 
\begin{equation}\label{arithmetic realization map}
h':{\mathcal K}\to \mathrm{Frac}\widehat{\mathcal K}
\end{equation}
sending $K\mapsto [(K,\orientedcircle)]$
is a homomorphism as monoid.
By abuse of notations, we occasionally denote the image $h'(K)$ by the same symbol $K$.
Related to  Conjecture \ref{injectivity on T},
\begin{prob}
Is the 
map $h'$ injective?
\end{prob}

On a structure of $\mathrm{Frac}\widehat{\mathcal K}$,
we pose

\begin{prob}
Is $\mathrm{Frac}\widehat{\mathcal K}$ a profinite group?
\end{prob}

By \cite{RZ}, to  show that $\mathrm{Frac}\widehat{\mathcal K}$ is a profinite group, 
we must show that it is
compact, Hausdorff and totally-disconnected.
The author is not aware of any one of their validities.
It is worthy to note that the set $\widehat{\mathcal T}$ 
of isotopy classes of profinite tangles
is not compact,
hence not a profinite space.
It is because the map $|\pi_0|:\widehat{\mathcal T}\to{\Bbb N}$
taking the number of connected components of each profinite tangles
is continuous and surjective to the non-compact space $\Bbb N$. 

\begin{defn}\label{GT-action on profinite knots}
Let  $(r,s)$ be a pair of profinite knot diagrams 
with $r=\gamma_{1,m}\cdots\gamma_{1,2}\cdot\gamma_{1,1}$
and $s=\gamma_{2,n}\cdots\gamma_{2,2}\cdot\gamma_{2,1}$
($\gamma_{i,j}$: profinite fundamental tangle diagram).
For $\sigma=(\lambda,f)\in \widehat{GT}$
(hence $\lambda\in\widehat{\mathbb Z}^\times, f\in \widehat{F}_2$),
define its action by
\begin{equation}\label{GT-action on GK}
\sigma \left(\frac{r}{s}\right):=
\frac{\sigma(r)}{\sigma(s)}:=
\frac{
\{\sigma(\gamma_{1,m})\cdots\sigma(\gamma_{1,2})\cdot\sigma(\gamma_{1,1})\}
\sharp(\Lambda_f)^{\sharp\alpha(s)}    }
{
\{\sigma(\gamma_{2,n})\cdots\sigma(\gamma_{2,2})\cdot\sigma(\gamma_{2,1})\}
\sharp(\Lambda_f)^{\sharp\alpha(r)}  
}
\in  \mathrm{Frac}\widehat{\mathcal K}.
\end{equation}
It is well-defined by Proposition \ref{knots to knots proposition} and Theorem \ref{GT-action theorem}.
Here
\begin{equation*}
\sigma(r):=\frac{ \{\sigma(\gamma_{1,m})\cdots\sigma(\gamma_{1,2})\cdot\sigma(\gamma_{1,1})\}}
{(\Lambda_f)^{\sharp\alpha(r)}} 
\text{  and  }
\sigma(s):=\frac{
\{\sigma(\gamma_{2,n})\cdots\sigma(\gamma_{2,2})\cdot\sigma(\gamma_{2,1})\} }
{(\Lambda_f)^{\sharp\alpha(s)}}
\in \mathrm{Frac}\widehat{\mathcal K}
\end{equation*}
are defined as follows:
\end{defn}

\begin{enumerate}
\item When $\gamma_{i,j}=a_{k,l}^\epsilon$,
we define
$$
\sigma(\gamma_{i,j}):=\gamma_{i,j}\cdot
f_{1\cdots k,k+1,k+2}^{s(\gamma_{i,j})}
$$
Here
$f_{1\cdots k,k+1,k+2}^{s(\gamma_{i,j})}
=\mathrm{ev}_{1,\epsilon_1}(f^{\uparrow\epsilon_2\epsilon_3})\otimes e_l^{\epsilon_4}$
with $s(\gamma_{i,j})=\epsilon_1\epsilon_2\epsilon_3\epsilon_4\in
\{\uparrow,\downarrow\}^{k+l+2}$
($\epsilon_1\in\{\uparrow,\downarrow\}^k$,
$\epsilon_2,\epsilon_3\in\{\uparrow,\downarrow\}$,
$\epsilon_4\in\{\uparrow,\downarrow\}^l$).
It is also described by
$f_{1\cdots k,k+1,k+2}^{s(\gamma_{i,j})}=
\left(f_{1\cdots k,k+1,k+2}\otimes e_l,
{s(\gamma_{i,j})}\right)
\in\widehat{B}$
with
$$f_{1\cdots k,k+1,k+2}\otimes e_l
=f(x_{1\cdots k,k+1},x_{k+1,k+2})
\in\widehat{B}_{k+l+2}
$$
where $x_{1\cdots k,k+1}$ and $x_{k+1,k+2}$ are regarded as elements of $\widehat{B}_{k+l+2}$.
We mean $f_{1\cdots k,k+1,k+2}\otimes e_l$ by the trivial braid $e_{l+2}\in\widehat{B}_{l+2}$ when $k=0$.
Figure \ref{GT-action on a} depicts the action.
Here the thickened black band stands for the trivial braid $e_k$ with $k$-strings.

\begin{figure}[h]
\begin{center}
\begin{tikzpicture}
\draw  (-0.6, 0.4) node{\Huge{$f($}};
\draw[-] (0.5,0)--(0.5,0.5);
                   \draw[color=white, line width=7pt](0,0) ..controls(0.1,0.5) and (0.9,0)        ..(1,0.5);
                   \draw[color=black, line width=5pt] (0,0) ..controls(0.1,0.5) and (0.9,0)        ..(1,0.6);
                   \draw[color=black, line width=5pt] (1,0.5) ..controls(0.9,1.0) and (0.1,0.5)       ..(0,1);
                   \draw[color=white, line width=7pt] (0.5,0.5)--(0.5,1);
                   \draw[-] (0.5,0.5)--(0.5,1);
                   \draw[-] (1.2,0)--(1.2,1);
\draw (1.5,0) node{,};
                   \draw[color=black, line width=5pt] (2.2,0)--(2.2,1);
                   \draw[-] (3,0)--(3,0.5);
                   \draw[color=white, line width=5pt](2.5,0) ..controls(2.6,0.5) and (3.4,0)        ..(3.5,0.5);
                   \draw[color=black] (2.5,0) ..controls(2.6,0.5) and (3.4,0)        ..(3.5,0.5);
                   \draw[color=black] (3.5,0.5) ..controls(3.4,1.0) and (2.6,0.5)       ..(2.5,1);
                   \draw[color=white, line width=5pt] (3,0.5)--(3,1);
                   \draw[-] (3,0.5)--(3,1);
\draw  (3.7, 0.4) node{\Huge{$)$}};
\draw (-1.0,-0.2) rectangle (4.0,1.2);
                   \draw[color=black, line width=5pt] (1.0,-0.5)--(1.0,-0.2) (1.0,1.2)--(1.0,2.0);
\draw[decorate,decoration={brace, mirror}] (0.85,-0.6) -- (1.15,-0.6) node[midway,below]{$k$};
                   \draw (2,-0.5)--(2,-0.2) (2,1.2)--(2,1.5);
                   \draw (2.8,-0.5)--(2.8,-0.2) (2.8,1.2)--(2.8,1.5);
                   \draw (2,1.5)  arc (180:0:0.4);
                  \draw (4.3,-0.5)--(4.3,2.0);
                  \draw (4.4,-0.5)--(4.4,2.0);
                  \draw[dotted] (4.45,0.6)--(4.75,0.6);
                  \draw[dotted] (4.45,1.7)--(4.75,1.7);
\draw[decorate,decoration={brace, mirror}] (4.3,-0.6) -- (4.8,-0.6) node[midway,below]{$l$};
                  \draw (4.8,-0.5)--(4.8,2.0);
\draw[color=white, very thick]  (0,1.5)--(5,1.5) ;
          \end{tikzpicture}
\caption{$\sigma(a_{k,l}^\epsilon)$}
\label{GT-action on a}
\end{center}
\end{figure}
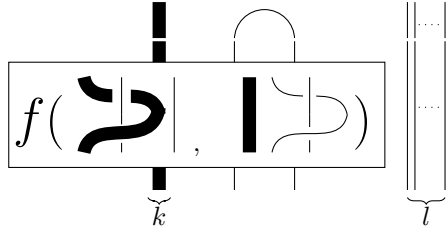
\item When $\gamma_{i,j}=b_{n}^\epsilon=(b_n,\epsilon)\in \widehat{B}$,  
we define
$$\sigma(\gamma_{i,j}):=(\sigma(b_n),\epsilon)
$$
which is nothing but the image of $b_{n}\in \widehat{B}_n$
by the $\widehat{GT}$-action on $\widehat{B}_n$ explained in \S \ref{Galois action on profinite braids}.

\item When $\gamma_{i,j}=c_{k,l}^\epsilon$,
we define
$$
\sigma(\gamma_{i,j}):=
f_{1\cdots k,k+1,k+2}^{-1,t(\gamma_{i,j})}
\cdot\gamma_{i,j}
$$
with
$f_{1\cdots k,k+1,k+2}^{-1,t(\gamma_{i,j})}
=\left(f_{1\cdots k,k+1,k+2}^{-1}\otimes e_l,
t(\gamma_{i,j})\right)
\in\widehat{B}$.
Figure \ref{GT-action on c} depicts the action.
\begin{figure}[h]
\begin{center}
\begin{tikzpicture}
\draw  (-0.6, 0.4) node{\Huge{$f($}};
                   \draw[-] (0.5,0)--(0.5,0.5);
                   \draw[color=white, line width=7pt](0,0) ..controls(0.1,0.5) and (0.9,0)        ..(1,0.5);
                   \draw[color=black, line width=5pt] (0,0) ..controls(0.1,0.5) and (0.9,0)        ..(1,0.6);
                   \draw[color=black, line width=5pt] (1,0.5) ..controls(0.9,1.0) and (0.1,0.5)       ..(0,1);
                   \draw[color=white, line width=7pt] (0.5,0.5)--(0.5,1);
                   \draw[-] (0.5,0.5)--(0.5,1);
                   \draw[-] (1.2,0)--(1.2,1);
\draw (1.5,0) node{,};
                   \draw[color=black, line width=5pt] (2.2,0)--(2.2,1);
                   \draw[-] (3,0)--(3,0.5);
                   \draw[color=white, line width=5pt](2.5,0) ..controls(2.6,0.5) and (3.4,0)        ..(3.5,0.5);
                   \draw[color=black] (2.5,0) ..controls(2.6,0.5) and (3.4,0)        ..(3.5,0.5);
                   \draw[color=black] (3.5,0.5) ..controls(3.4,1.0) and (2.6,0.5)       ..(2.5,1);
                   \draw[color=white, line width=5pt] (3,0.5)--(3,1);
                   \draw[-] (3,0.5)--(3,1);
\draw  (4, 0.4) node{\Huge{$)^{-1}$}};
\draw (-1.0,-0.2) rectangle (4.5,1.2);
                   \draw[color=black, line width=5pt] (1.0,-1.0)--(1.0,-0.2) (1.0,1.2)--(1.0,1.5);
\draw[decorate,decoration={brace, mirror}] (0.85,-1.1) -- (1.15,-1.1) node[midway,below]{$k$};
                   \draw (2,-0.5)--(2,-0.2) (2,1.2)--(2,1.5);
                   \draw (2.8,-0.5)--(2.8,-0.2) (2.8,1.2)--(2.8,1.5);
                   \draw (2,-0.5)  arc (180:360:0.4);
                  \draw (4.8,-1.0)--(4.8,1.5);
                  \draw (4.9,-1.0)--(4.9,1.5);
                  \draw[dotted] (4.95,0.6)--(5.25,0.6);
                  \draw[dotted] (4.95,-0.7)--(5.25,-0.7);
\draw[decorate,decoration={brace, mirror}] (4.8,-1.1) -- (5.3,-1.1) node[midway,below]{$l$};
                  \draw (5.3,-1.0)--(5.3,1.5);
\draw[color=white, very thick]  (0,-0.5)--(5.5,-0.5) ;
          \end{tikzpicture}
\caption{$\sigma(c_{k,l}^\epsilon)$}
\label{GT-action on c}
\end{center}
\end{figure}
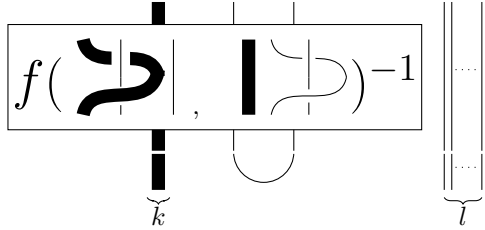
\end{enumerate}

The symbol $\Lambda_f$ represents the profinite tangle
$$
a_{0,0}^{\opannihilation}\cdot a_{0,2}^{\opannihilation\downarrow\uparrow}
\cdot (e_1^{\downarrow}\otimes f)\cdot c_{1,1}^{\downarrow\opcreation\uparrow}\cdot c_{0,0}^{\creation}
$$
(cf. Figure \ref{error term}).
\begin{figure}[h]
\begin{center}
         \begin{tikzpicture}
                      \draw[<-] (1.0,-0.2)--(1.0,1.2);
                       \draw (1.0,0.4) node{$\downarrow$};
                     \draw[->] (1.0,-0.2)  arc (180:360:0.6);
                      \draw (1.2,0.5) rectangle (2.4,0.9);
                     \draw (1.8,0.7) node{$f$};
                     \draw[->] (1.4,0.2)--(1.4,0.5);
                    \draw[->] (1.4,0.9)--(1.4,1.2);
                     \draw[<-] (1.4,0.2)  arc (180:360:0.2);
                     \draw[->] (1.4,1.2)  arc (0:180:0.2);
                     \draw[<-] (1.8,1.6)  arc (180:0:0.2);
                     \draw[<-] (1.8,0.2)--(1.8,0.5);
                     \draw[<-] (1.8,0.9)--(1.8,1.6);
                      \draw[->] (2.2,-0.2)--(2.2,0.5);
                      \draw[->] (2.2,0.9)--(2.2,1.6);
                       \draw (2.2,0) node{$\uparrow$};
                      \draw (2.2,1) node{$\uparrow$};
\draw[color=white, very thick] (0.9,-0.2)--(2.3,-0.2) (0.9,0.2)--(2.3,0.2) (0.9,1.2)--(2.3,1.2) (0.9,1.6)--(2.3,1.6) ;
         \end{tikzpicture}
\caption{$\Lambda_f$}
\label{error term}
\end{center}
\end{figure}

The symbol $\alpha(r)$ (resp. $\alpha(s)$) means
the number of annihilations;
the cardinality of the set
$\{j|\gamma_{i,j}\in A\}$ for $i=1$ (resp. $i=2$)
and $(\Lambda_f)^{\sharp\alpha(r)}$ (resp. $(\Lambda_f)^{\sharp\alpha(s)}$)
means the $\alpha(r)$-th (resp. the $\alpha(s)$-th ) power  of $\Lambda_f$ with respect to $\sharp$.
Particularly we have
\begin{equation}\label{oriented circle and lambda}
\sigma(\orientedcircle)
\sharp{\Lambda_f}={\orientedcircle} \ 
\in \mathrm{Frac}\widehat{\mathcal K} 
\end{equation}

\begin{prop}\label{knots to knots proposition}
Let $\sigma=(\lambda,f)\in\widehat{GT}$.
\begin{enumerate}
\item If  $r=\gamma_{1,m}\cdots\gamma_{1,2}\cdot\gamma_{1,1}$ 
($\gamma_{1,j}$: a profinite fundamental tangle) is a profinite knot, then
$\{\sigma(\gamma_{1,m})\cdots\sigma(\gamma_{1,2})\cdot\sigma(\gamma_{1,1})\}$
is again a profinite knot.
\item The profinite tangle $\Lambda_f$ (Figure \ref{error term}) is a profinite knot.
\end{enumerate}
\end{prop}

\begin{proof}
(1).
When $\gamma=b_{n}^\epsilon$,
since the projection 
$$
p:\widehat{B}_n\to\frak S_n
$$
is $\widehat{GT}$-equivalent
(the action on $\frak S_n$ means the trivial action),
the skeleton never change, i.e.
${\Bbb S}(\sigma(\gamma))={\Bbb S}(\gamma)$.

When $\gamma=a_{k,l}^\epsilon$ (resp. $c_{k,l}^\epsilon$),
the skeleton ${\Bbb S}(\sigma(\gamma))$ is obtained by connecting  $k+l+2$
straight bars
on the top  (resp. bottom) of ${\Bbb S}(\gamma)$
because $f\in \widehat{F}_2\subset\widehat{P}_3$.

Therefore
$\{\sigma(\gamma_{1,m})\cdots\sigma(\gamma_{1,2})\cdot\sigma(\gamma_{1,1})\}$
is again a profinite knot.

(2). By Figure \ref{error term},
it is easy because $f\in \widehat{F}_2\subset\widehat{P}_3$.
\end{proof}


\begin{thm}\label{GT-action theorem}
The equation \eqref{GT-action on GK}  determines a well-defined 
$\widehat{GT}$-action on $\mathrm{Frac}\widehat{\mathcal K}$.
Namely,

(1).
$\sigma(\frac{r_1}{s_1})=\sigma(\frac{r_2}{s_2})\in \mathrm{Frac}\widehat{\mathcal K}$
if $r_1\sim r_2$ and $s_1\sim s_2$, i.e. 
if $r_1=r_2$ and $s_1=s_2$ in $\widehat{\mathcal K}$.

(2).
$\sigma(\frac{r_1}{s_1})=\sigma(\frac{r_2}{s_2})\in \mathrm{Frac}\widehat{\mathcal K}$
if $(r_1,s_1)\approx (r_2,s_2)$, i.e.
if $\frac{r_1}{s_1}=\frac{r_2}{s_2}$ in  $\mathrm{Frac}\widehat{\mathcal K}$.

(3).
$\sigma_1(\sigma_2(x))=(\sigma_1\circ\sigma_2)(x)$
for any $\sigma_1,\sigma_2\in \widehat{GT}$ and
$x\in \mathrm{Frac}\widehat{\mathcal K}$.\\
Furthermore $\mathrm{Frac}\widehat{\mathcal K}$ forms a topological
$\widehat{GT}$-module. Namely,

(4). the action is compatible with the group structure, i.e.
$$
\sigma(e)=e, \qquad
\sigma(x\sharp y)=\sigma(x)\sharp\sigma(y), \qquad
\sigma(x^{-1})=\sigma(x)^{-1}
$$
for any $\sigma\in \widehat{GT}$ and
$x,y\in \mathrm{Frac}\widehat{\mathcal K}$.

(5). the action is continuous.
\end{thm}

\begin{proof}
(1).
Firstly we prove that
$\sigma\left( (r_1,s)\right)=\sigma\left( (r_2,s)\right)\in \widehat{\mathcal K}^2$
when $r_1$ is isotopic to $r_2$
for $\sigma=(\lambda,f)\in\widehat{GT}$.
We may further assume that $r_1$ is obtained from $r_2$ by a single operation of 
one of the moves (T1)-(T6).
\begin{itemize}
\item
If it is (T1), it is clear.
\item
If it is (T2), it is immediate because $\sigma(b_2)\cdot \sigma(b_1)=\sigma(b_2b_1)$
holds for $b_1,b_2\in\widehat{B}_n$.
\item
If it is (T3),
we may further assume that its $T_1$ and $T_2$ in (T3) are both fundamental profinite tangles.
Then
by Proposition \ref{action-change-basepoint-braids}
and Proposition \ref{action-change-basepoint-knots}
\begin{align*}
\sigma(e^{t(T_1)}_{n_1}\otimes T_2)&\cdot \sigma(T_1\otimes e^{s(T_2)}_{m_2})\\
&=
f_{[n_1],[n_2],[0]}^{-1, (t(T_1),t(T_2))}\cdot
(e^{t(T_1)}_{n_1}\otimes \sigma(T_2))\cdot 
f_{[n_1],[m_2],[0]}^{(t(T_1),s(T_2))}\cdot 
(\sigma(T_1)\otimes e^{s(T_2)}_{m_2}),\\
\intertext{by (T4)}
&=
f_{[n_1],[n_2],[0]}^{-1, (t(T_1),t(T_2))}\cdot
(e^{t(T_1)}_{n_1}\otimes \sigma(T_2))\cdot 
(\sigma(T_1)\otimes e^{s(T_2)}_{m_2})\cdot
f_{[m_1],[m_2],[0]}^{(s(T_1),s(T_2))} \\
&=
f_{[n_1],[n_2],[0]}^{-1, (t(T_1),t(T_2))}\cdot
(\sigma(T_1)\otimes e^{t(T_2)}_{n_2})\cdot
(e^{s(T_1)}_{m_1}\otimes \sigma(T_2))\cdot
f_{[m_1],[m_2],[0]}^{(s(T_1),s(T_2))} \\
&=
(\sigma(T_1)\otimes e^{t(T_2)}_{n_2})\cdot
f_{[m_1],[n_2],[0]}^{-1, (s(T_1),t(T_2))}\cdot
(e^{s(T_1)}_{m_1}\otimes \sigma(T_2))\cdot
f_{[m_1],[m_2],[0]}^{(s(T_1),s(T_2))} \\
&=\sigma(T_1\otimes e^{t(T_2)}_{n_2})\cdot \sigma(e^{s(T_1)}_{m_1}\otimes T_2).
\end{align*}
Whence (T3) is preserved by the $\widehat{GT}$-action.

\item
If it is (T4), we may assume that $T$ in (T4) is a fundamental profinite tangle.
Then by Proposition \ref{action-evaluation}
and Proposition \ref{action-change-basepoint-knots}
\begin{align*}
\sigma(\mathrm{ev}_{k,t(T)}&(b_l^\epsilon))\cdot
\sigma(e_{k-1}^{s_1}\otimes T\otimes e_{l-k}^{s_2})\\
=&
f_{[k'-1],[n],[l-k']}^{-1, (t_1,t(T),t_2)}\cdot
\mathrm{ev}_{k,t(T)}(\sigma(b_l^\epsilon))\cdot
f_{[k-1],[n],[l-k]}^{(s_1,t(T),s_2)} \\
& \quad\cdot
f_{[k-1],[n],[l-k]}^{-1,(s_1,t(T),s_2)}\cdot
(e_{k-1}^{s_1}\otimes \sigma(T)\otimes e_{l-k}^{s_2})\cdot
f_{[k-1],[m],[l-k]}^{(s_1,s(T),s_2)}
\\
=& f_{[k'-1],[n],[l-k']}^{-1,(t_1,t(T),t_2)}\cdot
\mathrm{ev}_{k,t(T)}(\sigma(b_l^\epsilon))\cdot
(e_{k-1}^{s_1}\otimes \sigma(T)\otimes e_{l-k}^{s_2})\cdot
f_{[k-1],[m],[l-k]}^{(s_1,s(T),s_2)},\\
\intertext{by (T4)}
=& f_{[k'-1],[n],[l-k']}^{-1,(t_1,t(T),t_2)}\cdot
(e_{k'-1}^{t_1}\otimes \sigma(T)\otimes e_{l-k'}^{t_2})\cdot
\mathrm{ev}^{k',s(T)}(\sigma(b_l^\epsilon))\cdot
f_{[k-1],[m],[l-k]}^{(s_1,s(T),s_2)},\\
=&f_{[k'-1],[n],[l-k']}^{-1,(t_1,t(T),t_2)}\cdot
(e_{k'-1}^{t_1}\otimes \sigma(T)\otimes e_{l-k'}^{t_2})\cdot
f_{[k'-1],[m],[l-k']}^{(t_1,s(T),t_2)} \\
& \quad\cdot
f_{[k'-1],[m],[l-k']}^{-1,(t_1,s(T),t_2)}\cdot
\mathrm{ev}^{k',s(T)}(\sigma(b_l^\epsilon)) \cdot
f_{[k-1],[m],[l-k]}^{(s_1,s(T),s_2)} \\
=&\sigma(e_{k'-1}^{t_1}\otimes T\otimes e_{l-k'}^{t_2})\cdot \sigma(\mathrm{ev}^{k',s(T)}(b_l^\epsilon)). 
\end{align*}
Whence (T4) is preserved by the action.

\item
If it is (T5), we have
\begin{align*}
\sigma(a_{k+1,l-1}^{\epsilon'})&\cdot \sigma(c_{k,l}^{\epsilon})
=a_{k+1,l-1}^{\epsilon'}\cdot 
f_{1\cdots k+1, k+2,k+3}^{s(c_{k,l}^{\epsilon})}\cdot
f_{1\cdots k, k+1,k+2}^{-1, s(c_{k,l}^{\epsilon})}\cdot
c_{k,l}^{\epsilon}. \\
\intertext{
By the pentagon equation \eqref{pentagon equation}
}
=&a_{k+1,l-1}^{\epsilon'}\cdot 
f^{-1,s(c_{k,l}^{\epsilon})}_{1\cdots k, k+1, k+2 \ k+3}\cdot
f_{k+1,k+2,k+3}^{s(c_{k,l}^{\epsilon})}\cdot
f_{1\cdots k, k+1 \ k+2, k+3}^{s(c_{k,l}^{\epsilon})}
\cdot c_{k,l}^{\epsilon}, \\
\intertext{by (T4) and Lemma \ref{projection of f} }
=&a_{k+1,l-1}^{\epsilon'}\cdot  f_{k+1,k+2,k+3}^{s(c_{k,l}^{\epsilon})}\cdot c_{k,l}^{\epsilon}. 
\end{align*}
It looks that (T5) is not preserved by the $\widehat{GT}$-action.
But actually it means that 
$\sigma(r_1)$ is obtained by an insertion of $f_{k+1,k+2,k+3}^{s(c_{k,l}^{\epsilon})}$
between $a_{k+1,l-1}^{\epsilon'}$ and $c_{k,l}^{\epsilon}$ in $\sigma(r_2$).
Thus 
$\sigma(r_1)=\sigma(r_2)\sharp\Lambda_f$.
Because $\alpha(r_1)=\alpha(r_2)+1$, we may say that
(T5) is compatible with the action
by \eqref{GT-action on GK}.
The second equality can be proved in the same way.

\item
If it is (T6), again by Proposition \ref{action-change-basepoint-braids}
and Proposition \ref{action-change-basepoint-knots},
\begin{align*}
\sigma(a_{k,l}^{\epsilon})&\cdot \sigma((\sigma_{k+1}^{\epsilon'})^{c})\\
=&a_{k,l}^{\epsilon}\cdot f_{1\cdots k, k+1,k+2}^{s(a_{k,l}^{\epsilon})}\cdot
 f_{1\cdots k, k+1,k+2}^{-1, s(a_{k,l}^{\epsilon})}\cdot
(\sigma_{k+1}^{\epsilon'})^{\lambda c}\cdot
f_{1\cdots k, k+1,k+2}^{s(a_{k,l}^{\bar\epsilon})} \\
=&a_{k,l}^{\epsilon}\cdot 
(\sigma_{k+1}^{\epsilon'})^{\lambda c}\cdot
f_{1\cdots k, k+1,k+2}^{s(a_{k,l}^{\bar\epsilon})}, \\
\intertext{by $\lambda\equiv 1\pmod 2$ and (T6)}
=&a_{k,l}^{\bar\epsilon} \cdot
f_{1\cdots k, k+1,k+2}^{s(a_{k,l}^{\bar\epsilon})}
=\sigma(a_{k,l}^{\bar\epsilon}). 
\end{align*}

The case for $c_{k,l}^\epsilon$ can be checked in the same way.
\end{itemize}

Secondly we prove that
$\sigma\left( (r,s_1)\right)=\sigma\left( (r,s_2)\right)\in \widehat{\mathcal K}^2$
when $s_1$ is isotopic to $s_2$. 
But it can be proved in a similar way to the above.
Hence our claim of (1) is obtained.

(2).
By definition,
$$r_1\sharp s_2\sharp t=r_2\sharp s_1\sharp t$$
in ${\mathcal K}$ for some profinite knot $t$.
By the definition of $\sharp$,
\begin{equation}
\sigma(k_1\sharp k_2)\sharp\sigma(\orientedcircle)
=\sigma(k_1)\sharp\sigma(k_2)
\end{equation}
equivalently
\begin{equation}\label{sigma-connected sum}
\sigma(k_1\sharp k_2)=\sigma(k_1)\sharp\sigma(k_2)\sharp\Lambda_f,
\end{equation}
holds in $\mathrm{Frac}\widehat{\mathcal K}$ for any profinite knot $k_1$ and $k_2$.
Therefore our claim is immediate because
$$
\sigma(r_1\sharp s_2\sharp t)=\sigma(r_1)\sharp\sigma(s_2)\sharp\sigma(t)\sharp
\Lambda_f\sharp\Lambda_f
$$
and
$$
\sigma(r_2\sharp s_1\sharp t)=\sigma(r_2)\sharp\sigma(s_1)\sharp\sigma(t)\sharp
\Lambda_f\sharp\Lambda_f.
$$

(3).
For $\sigma_1=(\lambda_1,f_1)$ and $\sigma_2=(\lambda_2,f_2)\in \widehat{GT}$,
put $\sigma_3=\sigma_2\circ\sigma_1\in \widehat{GT}$.
Hence, by \eqref{product of GT},
$\sigma_3=(\lambda_3,f_3)$ with
$\lambda_3=\lambda_2\lambda_1$ and
\begin{equation}\label{f3}
f_3=f_2\cdot \sigma_2(f_1)=
f_2\cdot f_1(x^{\lambda_2},f_2^{-1}y^{\lambda_2}f_2)
(=:f_2\circ f_1).
\end{equation}

Firstly we note that
$$
\sigma_3(\gamma_{i,j})=\sigma_2\left(\sigma_1(\gamma_{i,j})\right).
$$
When $\gamma_{i,j}=a_{k,l}^\epsilon$ or $c_{k,l}^\epsilon$, 
the equality is derived from \eqref{f3}.
When  $\gamma_{i,j}=b_{n}^\epsilon$, it is easy because of the
$\widehat{GT}$-action on $\widehat{B}_n$.

Secondly by definition we have
\begin{align*}
\sigma_2(\Lambda_{f_1})
&=
\{\sigma_2(a_{0,0})\cdot 
\sigma_2(a_{0,2})\cdot 
\sigma_2(e_1\otimes f_1)\cdot 
\sigma_2(c_{1,1})\cdot 
\sigma_2(c_{0,0})\}/\Lambda_{f_2}^{\sharp 2}. \\
\intertext{By Proposition \ref{action-change-basepoint-braids} and 
Proposition \ref{action-change-basepoint-knots} }
&=
\{a_{0,0}\cdot 
a_{0,2}\cdot
(f_2^{-1})_{1,2,3}\cdot
(f_2^{-1})_{1,23,4}\cdot
(e_1\otimes \sigma_2(f_1)) \\
&\qquad\qquad
\cdot (f_2)_{1,23,4}\cdot
(f_2)_{1,2,3}\cdot
(f_2^{-1})_{1,2,3}\cdot
c_{1,1}\cdot 
c_{0,0}\}/\Lambda_{f_2}^{\sharp 2} \\
&=
\{a_{0,0}\cdot 
a_{0,2}\cdot
(f_2^{-1})_{1,2,3}\cdot
(f_2^{-1})_{1,23,4}\cdot
(e_1\otimes \sigma_2(f_1)) \cdot 
(f_2)_{1,23,4}\cdot
c_{1,1}\cdot 
c_{0,0}\}/\Lambda_{f_2}^{\sharp 2}, \\
\intertext{by the pentagon equation \eqref{pentagon equation} }
&=
\{a_{0,0}\cdot 
a_{0,2}\cdot
(f_2^{-1})_{12,3,4}\cdot
(f_2^{-1})_{1,2,34}\cdot
(f_2)_{2,3,4}\cdot
(e_1\otimes \sigma_2(f_1)) \\
& \qquad
\cdot (f_2)_{1,23,4}\cdot
c_{1,1}\cdot 
c_{0,0}\}/\Lambda_{f_2}^{\sharp 2}, \\
\intertext{by a successive application of (T6) and Lemma \ref{projection of f}}
&=
\{a_{0,0}\cdot 
a_{0,2}\cdot
(f_2)_{2,3,4}\cdot
(e_1\otimes \sigma_2(f_1)) \cdot
c_{1,1}\cdot 
c_{0,0}\}/\Lambda_{f_2}^{\sharp 2} \\
&=
\{a_{0,0}\cdot 
a_{0,2}\cdot
(e_1\otimes f_2\cdot\sigma_2(f_1)) \cdot
c_{1,1}\cdot 
c_{0,0}\}/\Lambda_{f_2}^{\sharp 2}\\
&=\Lambda_{f_2\circ f_1}
\sharp\sigma_2(\orientedcircle)^{\sharp 2}.
\end{align*}
We note that in the above computation we omit the symbol $\epsilon$
of orientation.

Finally
\begin{align*}
\sigma_2\left(\sigma_1 \left(\frac{r}{s}\right)\right)
&=
\sigma_2\left(
\frac{
\{\sigma_1(\gamma_{1,m})\cdots\sigma_1(\gamma_{1,2})\cdot\sigma_1(\gamma_{1,1})\}
\sharp(\Lambda_{f_1})^{\sharp\alpha(s)}    }
{
\{\sigma_1(\gamma_{2,n})\cdots\sigma_1(\gamma_{2,2})\cdot\sigma_1(\gamma_{2,1})\}
\sharp(\Lambda_{f_1})^{\sharp\alpha(r)}  
}
\right) \\
&=
\frac{
\sigma_2(\{\sigma_1(\gamma_{1,m})\cdots\sigma_1(\gamma_{1,2})
\cdot\sigma_1(\gamma_{1,1})\})
\sharp\sigma_2(\Lambda_{f_1})^{\sharp\alpha(s)}
\sharp\Lambda_{f_2}^{\sharp\alpha(s)}
    }
{
\sigma_2(\{\sigma_1(\gamma_{2,n})\cdots\sigma_1(\gamma_{2,2})
\cdot\sigma_2(\gamma_{2,1})\})
\sharp\sigma_2(\Lambda_{f_1})^{\sharp\alpha(r)} 
\sharp\Lambda_{f_2}^{\sharp\alpha(r)}
} \\
&=
\frac{
\{\sigma_2(\sigma_1(\gamma_{1,m}))\cdots
\cdot\sigma_2(\sigma_1(\gamma_{1,1}))\}
\sharp\Lambda_{f_2}^{\sharp \alpha(s)}
\sharp(\Lambda_{f_3})^{\sharp\alpha(s)}
\sharp\sigma_2(\orientedcircle)^{\sharp 2\alpha(s)}
\sharp \Lambda_{f_2}^{\sharp\alpha(s)}
    }
{
\{\sigma_2(\sigma_1(\gamma_{2,n}))\cdots
\cdot\sigma_2(\sigma_1(\gamma_{2,1}))\}
\sharp\Lambda_{f_2}^{\sharp \alpha(r)}
\sharp(\Lambda_{f_3})^{\sharp\alpha(r)} 
\sharp\sigma_2(\orientedcircle)^{\sharp 2\alpha(r)}
\sharp \Lambda_{f_2}^{\sharp\alpha(r)}
} \\
&=
\frac{
\{\sigma_3(\gamma_{1,m})\cdots\sigma_3(\gamma_{1,2})
\cdot\sigma_3(\gamma_{1,1})\}
\sharp(\Lambda_{f_3})^{\sharp\alpha(s)}
    }
{
\{\sigma_3(\gamma_{2,n})\cdots\sigma_3(\gamma_{2,2})
\cdot\sigma_3(\gamma_{2,1})\}
\sharp(\Lambda_{f_3})^{\sharp\alpha(r)} 
} 
=\sigma_3 \left(\frac{r}{s}\right)
\end{align*}
by \eqref{oriented circle and lambda} and \eqref{sigma-connected sum}.

(4).
Let $x=r_1/s_1$ and $y=r_2/s_2$ with profinite knots $r_1,r_2,s_1,s_2$.
Then by  \eqref{sigma-connected sum} it is easy to see
\begin{align*}
\sigma(x\sharp y)&=
\sigma(\frac{r_1\sharp r_2}{s_1\sharp s_2})=
\frac{\sigma({r_1\sharp r_2})}{\sigma({s_1\sharp s_2})} 
=\frac{\sigma(r_1)\sharp\sigma(r_2)\sharp\Lambda_f}
{\sigma(s_1)\sharp\sigma(s_2)\sharp\Lambda_f} 
=\frac{\sigma(r_1)\sharp\sigma(r_2)}{\sigma(s_1)\sharp\sigma(s_2)} \\
&=\frac{\sigma(r_1)}{\sigma(s_1)}\sharp\frac{\sigma(r_2)}{\sigma(s_2)}
=\sigma(\frac{r_1}{s_1})\sharp\sigma(\frac{r_2}{s_2})
=\sigma(x)\sharp\sigma(y).
\end{align*}
The inverse is also easy to check.

(5).
We recall that
$\widehat{\mathcal K'}^{\rm seq}$
(cf. the proof of Theorem \ref{topological monoid}. (2))
is the set of  finite consistent sequences of profinite fundamental tangles $\gamma_n\cdots\gamma_2\cdot\gamma_1$
with  a single connected component and
with $(\gamma_n,\gamma_1)=(\opannihilation,\creation)$.
We define the map
$$
A:\widehat{GT}\times \widehat{\mathcal K'}^{\rm seq}\times\widehat{\mathcal K'}^{\rm seq}\to
\widehat{\mathcal K'}^{\rm seq}\times \widehat{\mathcal K'}^{\rm seq}
$$
by
$$
A(\sigma,r,s)=
\left(\{\sigma(\gamma_{1,m})\cdots\sigma(\gamma_{1,2})\cdot\sigma(\gamma_{1,1})\}
\sharp(\Lambda_f)^{\sharp\alpha(s)},
\{\sigma(\gamma_{2,n})\cdots\sigma(\gamma_{2,2})\cdot\sigma(\gamma_{2,1})\}
\sharp(\Lambda_f)^{\sharp\alpha(r)}\right) 
$$
for $\sigma=(\lambda,f)$,
$r=\gamma_{1,m}\cdots\gamma_{1,2}\cdot\gamma_{1,1}$
and $s=\gamma_{2,n}\cdots\gamma_{2,2}\cdot\gamma_{2,1}$
($\gamma_{i,j}$: profinite fundamental tangle).
We know that the $\widehat{GT}$-action on $\widehat{B}_n$
and the map $\widehat{GT}\to\widehat{B}_3$: $\sigma=(\lambda,f)\mapsto f$ are continuous,
so the map $A$ is continuous.
Since the diagram below is commutative
$$
\begin{CD}
\widehat{GT}\times \widehat{\mathcal K'}^{\rm seq}\times\widehat{\mathcal K'}^{\rm seq}
@>A>> \widehat{\mathcal K'}^{\rm seq}\times\widehat{\mathcal K'}^{\rm seq} \\
@VVV    @VVV \\
\widehat{GT} \times \mathrm{Frac}\widehat{\mathcal K}
@>>> \mathrm{Frac}\widehat{\mathcal K} \\
\end{CD}
$$
and the projection
$\widehat{\mathcal K'}^{\rm seq}\times\widehat{\mathcal K'}^{\rm seq} 
\twoheadrightarrow \mathrm{Frac}\widehat{\mathcal K}$
is continuous,
the lower map is also continuous.
\end{proof}

The following is required to prove Theorem \ref{GT-action theorem}.

\begin{prop}\label{action-change-basepoint-knots}
%

Let $k,l,m_1,m_2\geqslant 0$ and
$\epsilon_i\in\{\uparrow,\downarrow\}^{m_i}$ ($i=1,2$).
For any $\sigma\in\widehat{GT}$,
$$
\sigma(a_{m_1+k,l+m_2}^{\epsilon_1\epsilon\epsilon_2})
=
f_{[m_1],[k+l],[m_2]}^{-1,\epsilon_t}\cdot
\left(e_{m_1}^{\epsilon_1}\otimes
\sigma(a_{k,l}^{\epsilon})\otimes e_{m_2}^{\epsilon_2}\right)\cdot
f_{[m_1],[k+l+2],[m_2]}^{\epsilon_s}
$$
with $\epsilon_t=t(a_{m_1+k,l+m_2}^{\epsilon_1\epsilon\epsilon_2})$ and
$\epsilon_s=s(a_{m_1+k,l+m_2}^{\epsilon_1\epsilon\epsilon_2})$.
And
$$
\sigma(c_{m_1+k,l+m_2}^{\epsilon_1\epsilon\epsilon_2})
=f_{[m_1],[k+l+2],[m_2]}^{-1, \epsilon_t}
\cdot\left(e_{m_1}^{\epsilon_1}\otimes
\sigma(c_{k,l}^{\epsilon})\otimes e_{m_2}^{\epsilon_2}\right)\cdot
f_{[m_1],[k+l],[m_2]}^{\epsilon_s}
$$
with $\epsilon_t=t(a_{m_1+k,l+m_2}^{\epsilon_1\epsilon\epsilon_2})$ and
$\epsilon_s=s(a_{m_1+k,l+m_2}^{\epsilon_1\epsilon\epsilon_2})$.
\end{prop}

Here
$f_{[m_1],[M],[m_2]}^\epsilon\in\widehat{B}$ means
$(f_{[m_1],[M],[m_2]},\epsilon)\in\widehat{B}_{m_1+M+m_2}\times
\{\uparrow,\downarrow\}^{m_1+M+m_2}$
with  (see also \eqref{f[]})
\begin{align*}
f_{[m_1],[M],[m_2]}:=&
f_{1\cdots m_1, m_1+1\cdots m_1+M-1,m_1+M}\cdot
f_{1\cdots m_1, m_1+1\cdots m_1+M-2,m_1+M-1}\cdot \\
& \qquad
\cdots
f_{1\cdots m_1, m_1+1,m_1+2}
\in\widehat{B}_{m_1+M+m_2}. 
\end{align*}

\begin{proof}
%
We prove the first equality.
To avoid the complexity,
we again omit the symbol of orientations. 
By Definition \ref{GT-action on profinite knots}.(1),

\begin{align}\label{GT-action on tensor for annihilation}
&f_{[m_1],[k+l],[m_2]}^{-1}\cdot
(e_{m_1}\otimes\sigma(a_{k,l})\otimes e_{m_2})
\cdot f_{[m_1],[k+l+2],[m_2]} \\
&\quad=f_{[m_1],[k+l],[m_2]}^{-1}\cdot
a_{m_1+k,l+m_2}\cdot  f_{m_1+1\cdots m_1+k, m_1+k+1, m_1+k+2}
\cdot f_{[m_1],[k+l+2],[m_2]}.
\notag
\end{align}

\begin{itemize}
\item
When $M\geqslant k+3$,
by (T4), \\
$f_{m_1+1\cdots m_1+k, m_1+k+1, m_1+k+2}$
commutes with
$f_{1\cdots m_1, m_1+1\cdots m_1+M-1,m_1+M}$
and
$$
a_{m_1+k,l+m_2}\cdot
f_{1\cdots m_1, m_1+1\cdots m_1+M-1,m_1+M}
=
f_{1\cdots m_1, m_1+1\cdots m_1+M-3,m_1+M-2}
\cdot a_{m_1+k,l+m_2}.
$$
Therefore
$$
\eqref{GT-action on tensor for annihilation}=
f_{[m_1],[k],[l+m_2]}^{-1}\cdot
a_{m_1+k,l+m_2}\cdot  f_{m_1+1\cdots m_1+k, m_1+k+1, m_1+k+2}
\cdot f_{[m_1],[k+2],[l+m_2]}.
$$

\item
When $M=k+1,k+2$, our calculation goes as follows.
\begin{align*}
\eqref{GT-action on tensor for annihilation}&
=f_{[m_1],[k],[l+m_2]}^{-1}\cdot
a_{m_1+k,l+m_2}\cdot  f_{m_1+1\cdots m_1+k, m_1+k+1, m_1+k+2}\cdot \\
&\quad f_{1\cdots m_1, m_1+1\cdots m_1+k+1,m_1+k+2}\cdot
f_{1\cdots m_1, m_1+1\cdots m_1+k,m_1+k+1}\cdot
f_{[m_1],[k],[l+m_2+2]}, \\
\intertext{by the pentagon equation \eqref{pentagon equation}, }
&=f_{[m_1],[k],[l+m_2]}^{-1}\cdot
a_{m_1+k,l+m_2}\cdot  f_{1\cdots m_1, m_1+1\cdots m_1+k, m_1+k+1\ m_1+k+2}\cdot \\
&\quad f_{1\cdots m_1+k, m_1+k+1, m_1+k+2}\cdot
f_{[m_1],[k],[l+m_2+2]}, \\
\intertext{by (T4) and Lemma \ref{projection of f},}
&=f_{[m_1],[k],[l+m_2]}^{-1}\cdot
a_{m_1+k,l+m_2}\cdot 
f_{1\cdots m_1+k, m_1+k+1, m_1+k+2}\cdot
f_{[m_1],[k],[l+m_2+2]}. 
\end{align*}

\item
When $M\leqslant k$, by (T4) again, \\
$f_{1\cdots m_1+k, m_1+k+1, m_1+k+2}$ commutes with
$f_{1\cdots m_1, m_1+1\cdots m_1+M-1,m_1+M}$ and
$$
a_{m_1+k,l+m_2}\cdot
f_{1\cdots m_1, m_1+1\cdots m_1+M-1,m_1+M}
=
f_{1\cdots m_1, m_1+1\cdots m_1+M-1,m_1+M}
\cdot a_{m_1+k,l+m_2}.
$$
Therefore
\begin{align*}
\eqref{GT-action on tensor for annihilation}
&=f_{[m_1],[k],[l+m_2]}^{-1}\cdot
a_{m_1+k,l+m_2}\cdot f_{[m_1],[k],[l+m_2+2]}\cdot
f_{1\cdots m_1+k, m_1+k+1, m_1+k+2}, \\
&=a_{m_1+k,l+m_2}\cdot 
f_{1\cdots m_1+k, m_1+k+1, m_1+k+2}
=\sigma(a_{m_1+k,l+m_2}).
\end{align*}
\end{itemize}
Hence we get the equality.

The second equality can be proved in the same way.
\end{proof}

Thus by Theorem \ref{GT-action theorem}, the $\widehat{GT}$-action
\begin{equation}\label{GT-representation on GK}
\widehat{GT}\to \ \mathrm{Aut} \mathrm{Frac}\widehat{\mathcal K}
\end{equation}
is established.

We note that due to the creation-annihilation relations (T5)
we have to pass to the fractional group of $\widehat{\mathcal K}$
to construct $\widehat{GT}$-action on profinite knots.

\begin{rem}
In \cite{KT98}, it is explained that the category $\widehat{\mathcal T}(R)$
($R$: a commutative ring containing $\mathbb Q$)
of pro-$R$-algebraic framed tangles forms an $R$-linear ribbon category.
From which they deduced  an action of  the proalgebraic 
Grothendieck-Teichm\"{u}ller group $GT({R})$  \cite{Dr} on the space $\widehat{R[\mathcal K]}$
of pro-$R$-algebraic knots by a categorical arguments.
In our forthcoming paper \cite{F13}, 
it will be shown that our action \eqref{GT-representation on GK}
is a lift of their action.
An analogous deduction of our Theorem \ref{GT-action theorem}
by such categorical arguments
might be expectable.
However a completely same argument does not seem to work.
We may have a \lq ribbon' category $\widehat{\mathcal T}$ of profinite 
(framed) tangles
but a difficulty here is that the inverse
$(\Lambda_f)^{-1}$  does not look to exist generally 
in $\widehat{\mathcal T}$,
unlike the case of $\widehat{\mathcal T}(R)$.
(That is  why we introduced the group $\mathrm{Frac}\widehat{\mathcal K}$
of the fraction of the monoid $\widehat{\mathcal K}$.
A technical care to remedy this might be required.)
\end{rem}

\begin{rem}\label{factorization remark}
The Kontsevich knot invariant \cite{Ko}  is obtained by
integrating a formal analogue of the Knizhnik-Zamolodchikov (KZ) equation.
Bar-Natan \cite{B}, Kassel-Turaev \cite{KT98},  Le-Murakami \cite{LM} and
Pieunikhin \cite{P} gave a combinatorial
reconstruction of the invariant by using  an associator \cite{Dr}.
An {\it associator} means a pair $(\mu,\varphi)$ with $\mu\in {R}^\times$
and an $R$-coefficient
non-commutative formal power series $\varphi$ with two variables
satisfying  specific  relations which are analogues of our pentagon and hexagon equations
\eqref{pentagon equation}-\eqref{hexagon equation}
(\cite{Dr}, see also \cite{F10}).
One of striking results
\footnote{
It might be amazing to know that Drinfeld indicated it in \cite{Dr90}.
}
in Le-Murakami \cite{LM} is the rationality of the Kontsevich invariant
which follows from 
that the resulting invariant is, in fact,
independent of $\varphi$ (but depends on $\mu$).
Stimulated to their result,
Kassel and Turaev \cite{KT98} showed that their $GT(R)$-action on  $\widehat{R[\mathcal K]}$
factors through the cyclotomic action (cf. Appendix of \cite{KT98}).
The algebra  $\widehat{R[\mathcal K]}$ looks  \lq too linear'.
In contrast, in our profinite setting, it is totally unclear if
our above $\widehat{GT}$-action \eqref{GT-representation on GK} 
would depend only on $\lambda\in\widehat{\mathbb Z}^\times$ of
$(\lambda,f)\in \widehat{GT}$,
namely, the action would factor through $\widehat{\mathbb Z}^\times$.
We remind that their proof of the above independency is 
based on certain linear algebraic arguments, 
actually a vanishing of a (Harrison) cohomology of a chain complex associated with chords.
But here in our situation,
we are working not on their proalgebraic  setting but on the profinite setting
where such vanishing result has not been established.
And we do not know whether such a factorization would occur in our setting or not.
\end{rem}

Finally we obtain a Galois representation on knots
as an important consequence of Theorem \ref{GT-action theorem}.

\begin{thm}\label{Galois representation theorem}
Fix an embedding from $\overline{\mathbb Q}$ in to $\mathbb C$.
The group $\mathrm{Frac}\widehat{\mathcal K} $ of profinite knots admits
a non-trivial topological $G_{\mathbb Q}$-module structure.
Namely there is a non-trivial continuous Galois representation
\begin{equation}\label{GQ-action on GK}
\rho_0:G_{\mathbb Q}\to
\mathrm{Aut}\ \mathrm{Frac}\widehat{\mathcal K} .
\end{equation}
\end{thm}

\begin{proof}
By Theorem \ref{GT-action theorem},  it is straightforward because
in Theorem \ref{GQ to GT} we see that the absolute Galois group
$G_{\mathbb Q}$ 
is mapped to $\widehat{GT}$.
It is proved that $\mathrm{Frac}\widehat{\mathcal K}$ is nontrivial in Theorem \ref{theorem of GK}.
The non-triviality of $\rho_0$ is a consequence of the example below
because generally we have $K\neq \overline{K}$ in $\mathrm{Frac}\widehat{\mathcal K}$:
For instance, the left trefoil (the knot in the left below of Figure \ref{Example of connected sum})
and the right trefoil (its mirror image) are mapped to different elements
by the map \eqref{GK to QlK}
because they are known to be separated by the Kontsevich invariant.
(cf. Remark \ref{Kontsevich factorization}.)
\end{proof}


\begin{eg}\label{mirror image}
Especially when $\sigma\in G_{\mathbb Q}$ is equal to the complex conjugation morphism $\varsigma_0$,
it corresponds to $(\lambda,f)=(-1,1)\in\widehat{GT}$
whose action on $\widehat{B}_n$ is given by $\sigma_i\mapsto \sigma_i^{-1}$
($1\leqslant i\leqslant n-1$)
(cf. Example \ref{complex conjugation example}).
Whence the action of $\varsigma_0$ on  $\mathrm{Frac}\widehat{\mathcal K}$
is particularly described by
$$
\rho_0(\varsigma_0)\left(K\right)=\overline{K}
$$
for $K\in{\mathcal K}$ because $\Lambda_1=\orientedcircle$.
Here we denote the image of the 
map $h'$
\eqref{arithmetic realization map} on  an oriented knot $K$ 
by the same symbol $K$ and we mean the mirror image of the knot $K$ by $\overline{K}$.
The easiest example is that the right trefoil knot is mapped to the left trefoil knot
by the complex conjugation.
\end{eg}

There is another type of involution. 
For each profinite oriented knot $K$, we define $\mathrm{rev}(K)$ to be a profinite 
oriented knot which is obtained by reversing the orientation of $K$.
It is easy to see that it induces a well-defined involution
$$
\mathrm{rev}:\mathrm{Frac}\widehat{\mathcal K}\to \mathrm{Frac}\widehat{\mathcal K}
$$
which is an automorphism as a topological group.

\begin{prob}
Is $\mathrm{rev}$ defined over $\mathbb Q$?
Namely does 
\begin{equation}\label{rev-sigma=sigma=rev}
\mathrm{rev}\circ\sigma=\sigma\circ\mathrm{rev}
\end{equation}
hold for all $\sigma\in G_{\mathbb Q}$?
\end{prob}

This problem would be proven affirmatively
if we could show that $\Lambda_f=\mathrm{rev}(\Lambda_f)$
in $\mathrm{Frac}\widehat{\mathcal K}$
for all $\sigma=(\lambda,f)\in  G_{\mathbb Q}$.
In the proalgebraic setting
(cf. Remark \ref{factorization remark} and \cite{F13}),
$\mathrm{rev}$ can be defined similarly for $\widehat{R[\mathcal K]}$.
We can show the validity of an analogue of 
\eqref{rev-sigma=sigma=rev} for the action of
$\sigma\in GT(R)$ on $\widehat{R[\mathcal K]}$
by transmitting the $GT(R)$-action on $\widehat{R[\mathcal K]}$ into
a $GRT(R)$-action on $\widehat{CD}$. 

Extending other standard operations on knots, such as mutation and cabling,
into those on  profinite knots and examining their Galois behaviors
is also worthy to pursue.

\begin{project}
In \S \ref{Galois action on profinite braids}, the actions of $\widehat{GT}$ and
$G_{\mathbb Q}$ on the profinite braid group $\widehat{B}_n$
are discussed.
In Remark \ref{algebraic geometry interpretation}
it is explained in the language of algebraic geometry that
the $G_{\mathbb Q}$-action on $\widehat{B}_n$ is caused by 
the homotopy exact sequence of the scheme-theoretic fundamental group
of the quotient variety $\mathrm{Conf}^n_{{\frak S}_n}$
of the configuration space $\mathrm{Conf}^n$.
Whilst as for our  Galois action on knots in Theorem \ref{Galois representation theorem},
the author is not sure whether there is such kind of  its \lq purely'
algebraic-geometrical  interpretation  (without a usage of $\widehat{GT}$-factorization) or not.
In other word, it is not clear if there exists any (co)homology theory $H_\star$ 
(or any fundamental group  theory $\pi_1^\star$) and 
any (pro-)variety $X$ defined over $\mathbb Q$
such that
$$
\mathrm{Frac}\widehat{\mathcal K}= H_\star(X_{\overline{\mathbb Q}})
$$
(or $\mathrm{Frac}\widehat{\mathcal K}=\pi_1^\star(X_{\overline{\mathbb Q}})$)
and the right hand side naturally carries a $G_{\mathbb Q}$-action
which yields our $G_{\mathbb Q}$-action on $\mathrm{Frac}\widehat{\mathcal K}$.
It would be our future research.
\end{project}

Asking the validity of an analogue of Bely\u \i's theorem \cite{Be} in \eqref{Belyi theorem}
is particularly significant.

\begin{prob}\label{Belyi-type problem}
Is the action of the absolute Galois group  \eqref{GQ-action on GK} on profinite knots faithful? \\
If not, then what is the corresponding kernel field? And what is the arithmetic meaning of this?
\end{prob}

This is also related to the problem discussed in Remark \ref{factorization remark} above.
In \cite{F13}, it will be explained that the corresponding kernel field
is bigger than the maximal abelian extension ${\mathbb Q}(\mu_\infty)$ of $\mathbb Q$.

\begin{prob}
What is the $G_{\mathbb Q}$-invariant subspace of $\mathrm{Frac}\widehat{\mathcal K}$?
\end{prob}

In \cite{F13}, we will settle a similar problem formulated in the proalgebraic setting.

Asking the same type of questions for each given knot is also worthy to discuss.

\begin{prob}
(1).
What is the Galois stabilizer of each given knot?
And what is the corresponding Galois extension field of $\mathbb Q$?

(2). Suppose that two topological knots $K_1$ and $K_2$ and
an open normal subgroup $\mathcal N$ of $\mathrm{Frac}\widehat{\mathcal K}$
with finite index are given.
When  do two cosets  $K_1\cdot {\mathcal N}$ and $K_2\cdot {\mathcal N}$ 
lie on a same Galois orbit?
If so, then which $\sigma\in G_{\mathbb Q}$ connect them?
\end{prob}

Example \ref{mirror image} tells that 
a knot and its mirror image lie on a same Galois orbit and that
the stabilizer of an amphicheiral knot, say,
the figure eight knot (the knot in the left above of Figure \ref{Example of connected sum}),
contains the subgroup $\{id, \varsigma_0 \}$ of order $2$.

\begin{project}
We can also consider Galois action not only on profinite knots 
but also on pro-solvable knots by replacing profinite pure braid groups $\hat{P}_n$
by their pro-solvable completions 
on the definition of profinite knots.
The direction might be also worthy to pursue.
Ihara's profinite beta function $B_\sigma$ \cite{Ih} 
which arises from a description of the action of the absolute Galois group on the double commutator quotient of $\widehat{F}_2$
might help to see the action on a certain quotient.
\end{project}

As is explained in Remark \ref{Kontsevich factorization},
we have an isomorphism 
$\widehat{{\mathbb Q}[{\mathcal K}]}\simeq\widehat{CD}$
between the linear space of pro-algebraic knots
and  that of chord diagrams.

\begin{prob}
 What is  a profinite analogue of  the linear space of chord diagrams?
Do we have a profinite analogue of the above isomorphism for 
$\widehat{\mathcal K}$ and this?
\end{prob}

Grothendieck's dessins d'enfants \cite{G} are combinatorial diagrams
which  describe the action of the  absolute Galois group on $\widehat{F}_2$.
The author wonders
if a certain projective system of dessins attains the above profinite analogue of
the space of chord diagrams.


\begin{project}
There are various notions of equivalences for (framed) knots (and links)
such as
the Kirby moves (the Fenn-Rourke moves), the knot cobordism, the knot concordance, etc.
Extending these notions into those for our profinite links and
examining their behaviors under our Galois action is
worthy to pursue.
Particularly  the Kirby moves are known (consult the standard textbook such as \cite{O})
to yield a one to one correspondence
between the set of framed links modulo the equivalence generated by the moves
and the  set of isomorphism classes of
closed connected oriented 3-manifolds (three dimensional manifolds) 
by Dehn surgery.
Giving a nice formulation of profinite analogues of Kirby moves 
and a description of their Galois behaviors 
looks significant for a {\it realization} of
the analogy between number rings and 3-manifolds,
which is one of the most fundamental ones posted in arithmetic topology
by Kapranov \cite{Kap}, Morishita \cite{Mo} and Reznikov \cite{R}.
%
\end{project}

%

\appendix
\section{Two-bridge profinite knots}\label{Two-bridge profinite knots}
We introduce profinite analogues of two-bridge knots
and observe that the subgroup of $G\widehat{\mathcal K}$
generated by them is stable under our Galois action.

We recall that
a {\it two-bridge knot} (or link) is a  topological knot (or link) which can be isotoped so that the natural height function given by the $z$-coordinate has only two maxima and two minima as critical points.

\begin{defn}
A profinite knot (resp. link) is called a {\it two-bridge profinite knot (resp. link)}
when it is isotopic to
the presentation
$a_{0,0}^{\epsilon_1}\cdot a_{0,2}^{\epsilon_2}
\cdot b_4^{\epsilon_3}\cdot c_{1,1}^{\epsilon_4}\cdot c_{0,0}^{\epsilon_5}
$
with $b_4\in \widehat{B}_4$,
which is depicted in Figure \ref{two bridge knot} with its orientation ignored.
\end{defn}
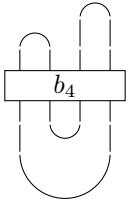
\begin{figure}[h]
\begin{center}
         \begin{tikzpicture}
                      \draw[-] (1.0,-0.2)--(1.0,0.5) (1.0,0.9)--(1.0,1.2);
                     \draw[-] (1.0,-0.2)  arc (180:360:0.6);
                      \draw (0.8,0.5) rectangle (2.4,0.9);
                     \draw (1.6,0.7) node{$b_4$};
                     \draw[-] (1.4,0.2)--(1.4,0.5);
                    \draw[-] (1.4,0.9)--(1.4,1.2);
                     \draw[-] (1.4,0.2)  arc (180:360:0.2);
                     \draw[-] (1.4,1.2)  arc (0:180:0.2);
                     \draw[-] (1.8,1.6)  arc (180:0:0.2);
                     \draw[-] (1.8,0.2)--(1.8,0.5);
                     \draw[-] (1.8,0.9)--(1.8,1.6);
                      \draw[-] (2.2,-0.2)--(2.2,0.5) (2.2,0.9)--(2.2,1.6);
\draw[color=white, very thick] (0.9,-0.2)--(2.3,-0.2) (0.9,0.2)--(2.3,0.2) (0.9,1.2)--(2.3,1.2) (0.9,1.6)--(2.3,1.6) ;
         \end{tikzpicture}
\caption{Two-bridge profinite knot}
\label{two bridge knot}
\end{center}
\end{figure}

We note that each isotopy class of
two-bridge (topological) knot naturally determines 
an isotopy class of  two-bridge profinite knot.

We recall that a two-bridge (topological) knot
is also called as a {\it rational knot} because each isotopy class of
two-bridge knot (or link) is characterized by a rational number 
via its continued fraction expansion (cf. \cite{Li}).
As its profinite analogue,
it is interesting to see if
the set of isotopy classes of two-bridge profinite knots 
can be also parametrized by any numbers or not.

We denote the subgroup of $G\widehat{\mathcal K}$
generated by two-bridge profinite knots by ${\mathcal Br}(2)$.

\begin{prop}
The subgroup ${\mathcal Br}(2)$ is stable under our Galois action.
\end{prop}

\begin{proof}
It can be verified by 
showing that ${\mathcal Br}(2)$ is stable by our $\widehat{GT}$-action.
Let $\sigma=(\lambda,f)\in\widehat{GT}$ and let $K$ be a  two-bridge profinite knot
with the above presentation. Then
\begin{align*}
\sigma(\frac{K}{\orientedcircle})&=\frac{\sigma(K)}{\sigma(\orientedcircle)}
=\frac{\{\sigma(a_{0,0}^{\epsilon_1})\cdot\sigma( a_{0,2}^{\epsilon_2})
\cdot \sigma(b_4^{\epsilon_3})\cdot \sigma(c_{1,1}^{\epsilon_4})\cdot \sigma(c_{0,0}^{\epsilon_5})\}\sharp\Lambda_f}
{(\Lambda_f)^{\sharp 2}}\\
&=\frac{a_{0,0}^{\epsilon_1}\cdot a_{0,2}^{\epsilon_2}
\cdot (\sigma(b_4^{\epsilon_3})\cdot f^{-1}_{123})\cdot c_{1,1}^{\epsilon_4}\cdot c_{0,0}^{\epsilon_5}}
{\Lambda_f}.
\end{align*}
Since both numerator and denominator of the last term are two-bridge knots
and $\widehat{GT}$ acts on  $G\widehat{\mathcal K}$ as a group automorphism,
our claim is obtained.
\end{proof}

Considering profinite analogues of  other types of (topological) knots
such as  hyperbolic knots, torus knots, etc.
and considering their Galois behaviors might be interesting problems.

\section{Profinite framed knots}\label{profinite framed knots}
We will extend our definition of profinite analogue of knots into that of
framed knots and show that the absolute Galois group acts on the group generated by
them.
The construction is given in a same way to our  arguments in \S \ref{Galois action on profinite knots}.

\begin{defn}\label{definition of profinite framed tangle}
The set $\widehat{\mathcal{FT}}$ of {\it profinite framed tangles} 
(isotopy classes of profinite  framed tangle diagrams) is 
defined to be the set of profinite tangle diagrams divided by the {\it framed isotopy},
the equivalence given by  a finite number of  the moves (FT1)-(FT6).
Here  (FT1)-(FT5)  are same to (T1)-(T5) given in \S \ref{Galois action on profinite knots}
while (FT6) is given below.
The set $\widehat{\mathcal{FK}}$ (resp. $\widehat{\mathcal{FL}}$)
of {\it profinite framed knots (resp. links)}
is the subset of $\widehat{\mathcal{FT}}$ which consists of
framed isotopy classes of  profinite knot (resp. link) diagrams.
\end{defn}

(FT6) {\it First framed Reidemeister move}:
for $c\in\widehat{\mathbb Z}$,
$c_{k,l+2}^{\epsilon_1}$, $c_{k+2,l}^{\epsilon_2}\in C$, 
$a_{k+1,l+1}^{\epsilon_3}\in A$ and
$\sigma_{k+1}^{\epsilon_1'}$, $\sigma_{k+3}^{\epsilon_2'}\in\widehat{B}$
representing $\sigma_{k+1}\in {\widehat B}_{k+l+2}$ and
$\sigma_{k+3}\in {\widehat B}_{k+l+4}$
such that the sequence of the left hand side is consistent,
\begin{equation}\label{FT6a}
a_{k+1,l+1}^{\epsilon_3}\cdot(\sigma_{k+3}^{\epsilon_2'})^{-c}\cdot c_{k+2,l}^{\epsilon_2}\cdot(\sigma_{k+1}^{\epsilon_1'})^c\cdot c_{k,l}^{\epsilon_1}
=c_{k,l}^{\epsilon_3'}.
\end{equation}
Here $\epsilon_3'$ is chosen to be  $t(a_{k+1,l+1}^{\epsilon_3})=t(c_{k,l}^{\epsilon_3'})$.

For $c\in\widehat{\mathbb Z}$,
$c_{k+1,l+1}^{\epsilon_1}\in C$, 
$a_{k,l}^{\epsilon_2}$, $a_{k+2,l}^{\epsilon_3}\in A$ and
$\sigma_{k+1}^{\epsilon_2'}$, $\sigma_{k+3}^{\epsilon_3'}\in\widehat{B}$
representing $\sigma_{k+1}\in {\widehat B}_{k+l+2}$ and
$\sigma_{k+3}\in {\widehat B}_{k+l+4}$
such that the sequence of the left hand side is consistent,
\begin{equation}\label{FT6b}
a_{k,l}^{\epsilon_2}\cdot(\sigma_{k+1}^{\epsilon_2'})^{-c}\cdot a_{k+2,l}^{\epsilon_3}\cdot(\sigma_{k+3}^{\epsilon_3'})^c\cdot c_{k+1,l+1}^{\epsilon_1}
=a_{k,l}^{\epsilon_1'}.
\end{equation}
Here again $\epsilon_1'$ is chosen to be  $s(c_{k+1,l+1}^{\epsilon_3})=s(a_{k,l}^{\epsilon_1'})$.
\begin{figure}[h]
\begin{tabular}{cc}
\begin{minipage}{0.5\hsize}
\begin{center}
          \begin{tikzpicture}
                    \draw[-] (0,0.2) --(0, 3.0) ;
                    \draw[dotted] (0.1,0.4) --(0.3, 0.4) (0.1,1)--(0.3,1)  (0.1,1.6)--(0.3,1.6)  (0.1,2.2)--(0.3,2.2)
 (0.1,2.8)--(0.3,2.8);
                    \draw[-] (0.4,0.2) --(0.4, 3.0);
                     \draw[decorate,decoration={brace,mirror}] (-0.1,0.2) -- (0.5,0.2) node[midway,below]{$k$};
\draw (0.6,0.5)  arc (180:360:0.2);
                    \draw[-] (.6,1.2) --(1, .8);
                    \draw[draw=white,double=black, very thick] (0.6,0.8) --(1, 1.2);
\draw  (0.5,0.8) rectangle (1.1,1.2) node[right]{$c$};
                    \draw[-] (0.6,0.5)--(0.6,.8) (1.0,0.5)--(1.0,.8);
                    \draw[-] (0.6,1.2)--(0.6,3.0) (1.0,1.2)--(1.0, 2.6);
\draw (1.0,2.6)  arc (180:0:0.25);
\draw (1.5,1.8)  arc (180:360:0.2);
                    \draw[-] (1.5,2) --(1.9, 2.4);
                    \draw[draw=white,double=black, very thick] (1.9,2) --(1.5, 2.4);

\draw  (1.4,2) rectangle (2,2.4) node[right]{$c$};
                    \draw[-] (1.5,1.8)--(1.5,2) (1.5,2.4)--(1.5,2.6);
                    \draw[-] (1.9,1.8)--(1.9,2) (1.9,2.4)--(1.9,3.0);

                    \draw[-] (2.4,0.2)--(2.4,3.0) ;
                    \draw[-] (2.9,0.2) --(2.9,3.0) ;
 \draw[dotted] (2.5,0.4) --(2.8, 0.4) (2.5,1)--(2.8,1)  (2.5,1.6)--(2.8,1.6)  (2.5,2.2)--(2.8,2.2)
 (2.5,2.8)--(2.8,2.8);
                    \draw[decorate,decoration={brace,mirror}] (2.3,0.2) -- (3.0,0.2) node[midway,below]{$l$};
               \draw[color=white, very thick] (-0.1,0.6)--(3,0.6)  (-0.1,1.4)--(3,1.4)  (-0.1,1.8)--(3,1.8)
 (-0.1,2.6)--(3,2.6) ;

\draw  (3.3,1.5)node{$=$};
                    \draw[-] (3.9,0.2) --(3.9,3);
                    \draw[dotted] (4,1.5) --(4.2, 1.5) ;
                    \draw[-] (4.3,0.2) --(4.3,3);
                     \draw[decorate,decoration={brace,mirror}] (3.8,0.2) -- (4.4,0.2) node[midway,below]{$k$};
\draw (4.4,3)  arc (180:360:0.3);
                    \draw[-] (5.1,0.2) --(5.1,3) ;
                    \draw[-] (5.6,0.2) --(5.6,3) ;
                   \draw[dotted] (5.2,1.5) --(5.5, 1.5) ;
                    \draw[decorate,decoration={brace,mirror}] (5,0.2) -- (5.7,0.2) node[midway,below]{$l$};
\draw (6.0,0.2) node{,};
           \end{tikzpicture}
\end{center}
\end{minipage}
\begin{minipage}{0.5\hsize}
\begin{center}
          \begin{tikzpicture}
                    \draw[-] (0,0.2) --(0, 3.0) ;
                    \draw[dotted] (0.1,0.4) --(0.3, 0.4) (0.1,1)--(0.3,1)  (0.1,1.6)--(0.3,1.6)  (0.1,2.2)--(0.3,2.2)
 (0.1,2.8)--(0.3,2.8);
                    \draw[-] (0.4,0.2) --(0.4, 3.0);
                     \draw[decorate,decoration={brace,mirror}] (-0.1,0.2) -- (0.5,0.2) node[midway,below]{$k$};
\draw (0.6,2.6)  arc (180:0:0.2);
                    \draw[-] (.6,2) --(1, 2.4);
                    \draw[draw=white,double=black, very thick] (0.6,2.4) --(1, 2);
\draw  (0.5,2) rectangle (1.1,2.4) node[right]{$c$};
                    \draw[-] (0.6,0.2)--(0.6,2) (0.6,2.4)--(0.6,2.6);
                    \draw[-] (1,0.6)--(1,2) (1.0,2.4)--(1.0, 2.6);
\draw (1.0,.6)  arc (180:360:0.25);
\draw (1.5,1.4)  arc (180:0:0.2);
                    \draw[-] (1.5,1.2) --(1.9, 0.8);
                    \draw[draw=white,double=black, very thick] (1.9,1.2) --(1.5, 0.8);
\draw  (1.4,0.8) rectangle (2,1.2) node[right]{$c$};

                    \draw[-] (1.5,.6)--(1.5,.8) (1.5,1.2)--(1.5,1.4);
                    \draw[-] (1.9,.2)--(1.9,.8) (1.9,1.2)--(1.9,1.4);

                    \draw[-] (2.4,0.2)--(2.4,3.0) ;
                    \draw[-] (2.9,0.2) --(2.9,3.0) ;
 \draw[dotted] (2.5,0.4) --(2.8, 0.4) (2.5,1)--(2.8,1)  (2.5,1.6)--(2.8,1.6)  (2.5,2.2)--(2.8,2.2)
 (2.5,2.8)--(2.8,2.8);
                    \draw[decorate,decoration={brace,mirror}] (2.3,0.2) -- (3.0,0.2) node[midway,below]{$l$};
               \draw[color=white, very thick] (-0.1,0.6)--(3,0.6)  (-0.1,1.4)--(3,1.4)  (-0.1,1.8)--(3,1.8)
 (-0.1,2.6)--(3,2.6) ;

\draw  (3.3,1.5)node{$=$};
                    \draw[-] (3.9,0.2) --(3.9,3);
                    \draw[dotted] (4,1.5) --(4.2, 1.5) ;
                    \draw[-] (4.3,0.2) --(4.3,3);
                     \draw[decorate,decoration={brace,mirror}] (3.8,0.2) -- (4.4,0.2) node[midway,below]{$k$};
\draw (4.4,0.3)  arc (180:0:0.3);
                    \draw[-] (5.1,0.2) --(5.1,3) ;
                    \draw[-] (5.6,0.2) --(5.6,3) ;
                   \draw[dotted] (5.2,1.5) --(5.5, 1.5) ;
                    \draw[decorate,decoration={brace,mirror}] (5,0.2) -- (5.7,0.2) node[midway,below]{$l$};
           \end{tikzpicture}
\end{center}
\end{minipage}
\end{tabular}
\caption{(FT6): First framed Reidemeister move}
\label{FT6}
\end{figure}
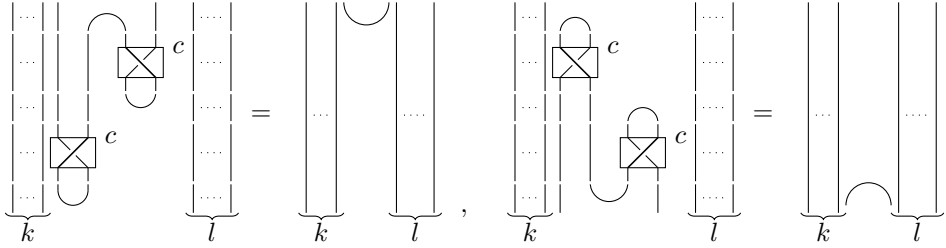

We note that, in the first (resp. second) equation,
$c_{k,l}^{\epsilon_1}=c_{k,l}^{\epsilon_3'}$
(resp. $a_{k,l}^{\epsilon_2}=a_{k,l}^{\epsilon_1'}$)
if and only if $c\equiv 0\pmod 2$.
It is easy to see that our previous (T6) implies (FT6).

\begin{rem}
There is a natural projection
$\widehat{\mathcal{FT}}\to \widehat{\mathcal{T}}$
and hence
\begin{equation}\label{forget frame}
\widehat{\mathcal{FK}}\to \widehat{\mathcal{K}}.
\end{equation}
It is clear by definition.
\end{rem}

\begin{defn}
We introduce special tangles, {\it twists}, $\varphi^\uparrow$ and $\varphi^\downarrow$
by 
$$\varphi^\uparrow:=a_{1,0}^{\uparrow\annihilation}
\cdot\sigma_1^{\uparrow\uparrow\downarrow}\cdot c_{1,0}^{\uparrow\opcreation},
\qquad
\varphi^\downarrow:=a_{1,0}^{\downarrow\opannihilation}
\cdot\sigma_1^{\downarrow\downarrow\uparrow}\cdot c_{1,0}^{\downarrow\creation}.
$$
The picture of $\varphi^\uparrow$ is depicted in the first term of Figure \ref{twist}.
\end{defn}

It is an easy exercise to derive the lemma below
from (FT1)-(FT6).

\begin{lem}\label{various twists}
(1).
The following equalities hold
(cf. Figure \ref{twist}):
$$
\varphi^\uparrow
=a_{0,1}^{\opannihilation\uparrow}\cdot\sigma_2^{\downarrow\uparrow\uparrow}\cdot
c_{0,1}^{\creation\uparrow}
=a_{0,1}^{\annihilation\uparrow}\cdot (\sigma_2^{-1})^{\uparrow\uparrow\downarrow}\cdot
c_{1,0}^{\uparrow\opcreation}
=a_{0,1}^{\opannihilation\uparrow}\cdot (\sigma_1^{-1})^{\uparrow\downarrow\uparrow}\cdot
c_{1,0}^{\uparrow\creation}.
$$
\begin{figure}[h]
\begin{center}
        \begin{tikzpicture}
                      \draw[->] (0,-0.1)--(0,.3) ;
                      \draw[->] (0,.6)--(0,1.0);
                    \draw[-] (0,.6) --(0.3, 0.3);
                    \draw[draw=white,double=black, very thick] (0,.3) --(0.3, 0.6);
                     \draw[<-] (0.3,0.3)  arc (180:360:0.2);
                     \draw[->] (0.3,0.6)  arc (180:0:0.2);
                    \draw[<-] (0.7,.3) --(0.7, 0.6);
\draw  (1.,0.4)node{$=$};
                    \draw[<-] (1.3,.3) --(1.3, 0.6);
                     \draw[->] (1.3,0.3)  arc (180:360:0.2);
                     \draw[<-] (1.3,0.6)  arc (180:0:0.2);
                    \draw[-] (1.7,0.6) --(2.0, 0.3);
                    \draw[draw=white,double=black, very thick] (1.7,0.3) --(2., 0.6);
                      \draw[->] (2.0,-0.1)--(2.0,.3); 
                      \draw[->]  (2.0,.6)--(2.0,1.0);
\draw  (2.3,0.4)node{$=$};
                    \draw[->] (2.6,-0.1) --(2.6, 0.6);
                     \draw[->] (2.6,0.6)  arc (180:0:0.2);
                    \draw[-] (3.0,0.3) --(3.4, 0.6);
                    \draw[draw=white,double=black, very thick] (3.4,0.3) --(3., 0.6);
                     \draw[<-] (3,0.3)  arc (180:360:0.2);
                      \draw[->]  (3.4,.6)--(3.4,1.0);
\draw  (3.7,0.4)node{$=$};
                    \draw[->] (4.0,-0.1) --(4, 0.3);
                    \draw[-] (4.0,0.3) --(4.4, 0.6);
                    \draw[draw=white,double=black, very thick] (4.4,0.3) --(4., 0.6);
                     \draw[<-] (4,0.6)  arc (180:0:0.2);
                     \draw[->] (4.4,0.3)  arc (180:360:0.2);
                    \draw[->] (4.8,0.3) --(4.8, 1.0);
\draw[color=white,  thick] (-0.1,0.3)--(4.9,0.3) (-0.1,0.6)--(4.9,0.6) ;
        \end{tikzpicture}
\caption{twist $\varphi^\uparrow$}
\label{twist}
\end{center}
\end{figure}

(2).
We have the following equalities (cf. Figure \ref{switch twists figure})
$$
(\varphi^\downarrow\otimes \uparrow)\cdot\creation=
(\downarrow\otimes\varphi^\uparrow)\cdot \creation,\qquad
\annihilation\cdot (\varphi^\uparrow\otimes \downarrow)
=\annihilation\cdot(\uparrow\otimes\varphi^\downarrow).
$$
We also have  the same equalities  obtained by reversing all arrows.
\end{lem}
\begin{figure}[h]
\begin{tabular}{cc}
\begin{minipage}{0.5\hsize}
\begin{center}
         \begin{tikzpicture}
                     \draw[<-](.0,0)--(.0,0.4);
                     \draw[<-] (.0,0.8)--(.0,1.3);
                    \draw[-] (.0,.8) --(.2, 0.4);
                    \draw[draw=white,double=black, very thick] (.0,.4) --(.2, 0.8);
                     \draw[->] (0.2,0.4)  arc (180:360:0.1);
                     \draw[<-] (0.2,0.8)  arc (180:0:0.1);
                    \draw[->] (.4,.4) --(.4, 0.8);
                     \draw[->] (0,0)  arc (180:360:0.4);
                      \draw[->] (0.8,0)--(0.8,1.3);
\draw[color=white,  thick] (-0.1,0)--(1.3,0) (-0.1,0.4)--(1.3,0.4)  (-0.1,0.8)--(1.3,0.8);
\draw  (1.4,0.5)node{$=$};
                     \draw[->] (2,0)  arc (180:360:0.4);
                      \draw[<-] (2.,0)--(2.,1.3);
                     \draw[->](2.8,0)--(2.8,0.4);
                     \draw[->] (2.8,0.8)--(2.8,1.3);
                    \draw[-] (2.8,.8) --(3.0, 0.4);
                    \draw[draw=white,double=black, very thick] (2.8,.4) --(3.0, 0.8);
                     \draw[<-] (3,0.4)  arc (180:360:0.1);
                     \draw[->] (3.,0.8)  arc (180:0:0.1);
                    \draw[<-] (3.2,.4) --(3.2, 0.8);
\draw[color=white,  thick] (1.9,0)--(3.3,0) (1.9,0.4)--(3.3,0.4)  (1.9,0.8)--(3.3,0.8);
                    \draw (3.6,-0.3) node{,};
         \end{tikzpicture}
\end{center}
\end{minipage}
\begin{minipage}{0.5\hsize}
\begin{center}
         \begin{tikzpicture}
                     \draw[->](.0,-0.1)--(.0,0.4);
                     \draw[->] (.0,0.8)--(.0,1.2);
                    \draw[-] (.0,.8) --(.2, 0.4);
                    \draw[draw=white,double=black, very thick] (.0,.4) --(.2, 0.8);
                     \draw[<-] (0.2,0.4)  arc (180:360:0.1);
                     \draw[->] (0.2,0.8)  arc (180:0:0.1);
                    \draw[<-] (.4,.4) --(.4, 0.8);
                     \draw[->] (0,1.2)  arc (180:0:0.4);
                      \draw[<-] (0.8,-0.1)--(0.8,1.2);
\draw[color=white,  thick]  (-0.1,0.4)--(1.3,0.4)  (-0.1,0.8)--(1.3,0.8) (-0.1,1.2)--(1.3,1.2);
\draw  (1.4,0.5)node{$=$};
                     \draw[->] (2,1.2)  arc (180:0:0.4);
                      \draw[->] (2.,-0.1)--(2.,1.2);
                     \draw[<-](2.8,-0.1)--(2.8,0.4);
                     \draw[<-] (2.8,0.8)--(2.8,1.2);
                    \draw[-] (2.8,.8) --(3.0, 0.4);
                    \draw[draw=white,double=black, very thick] (2.8,.4) --(3.0, 0.8);
                     \draw[->] (3,0.4)  arc (180:360:0.1);
                     \draw[<-] (3.,0.8)  arc (180:0:0.1);
                    \draw[->] (3.2,.4) --(3.2, 0.8);
\draw[color=white,  thick] (1.9,1.2)--(3.3,1.2) (1.9,0.4)--(3.3,0.4)  (1.9,0.8)--(3.3,0.8);
         \end{tikzpicture}
\end{center}
\end{minipage}
\end{tabular}
\caption{Lemma \ref{various twists}.(2).}
\label{switch twists figure}
\end{figure}

The space $\widehat{\mathcal{FT}}$ (hence the subspace $\widehat{\mathcal{FK}}$)
carries a structure of topological space by 
profinite topology on $\widehat{B}_n$ ($n=1,2,\dots$)
and the discrete topology on $A$ and on $C$. 

We have a framed version of  our Theorem A in \S \ref{introduction}.

\begin{thm}
(1). The space $\widehat{\mathcal{FK}}$ carries a structure of a topological commutative monoid
whose product structure
$$
\sharp^{\rm{fr}}: \widehat{\mathcal{FK}}\times \widehat{\mathcal{FK}} \to \widehat{\mathcal{FK}}. 
$$
is given by the connected sum \eqref{connected sum}
for $K_1=\alpha_m\cdots\alpha_1$ and $K_2=\beta_n\cdots\beta_1$.

(2). Let $\mathcal{FK}$ denote the set of framed isotopy classes of  (topological) oriented framed knots
\footnote{
Framed knot means a knot equipped with a framing, that is,
a nonzero normal vector field on it.
}.
Then there is a natural map
$$
h^{\rm{fr}}: {\mathcal{FK}}\to \widehat{\mathcal{FK}}.
$$
The map is with dense image and
is a monoid homomorphism with respect to the connected sum.
\end{thm}

The map $h^{\rm{fr}}$  should be conjectured to be injective
by the same reason to Conjecture \ref{injectivity on T}.

\begin{proof}
The proof can be done in an almost same way to that of Theorem \ref{profinite realization theorem} and
\ref{topological monoid}.
We outline it by strengthening a slight difference.

(1). 
Each framed isotopy class of profinite knot contains a profinite knot 
$K=\gamma_m\cdots \gamma_1$
for some $m$ with  $(\gamma_m,\gamma_1)=(\opannihilation, \creation)$
by (FT6)  with $c\equiv 1\pmod 2$. 
Hence the connected sum could extend to $\widehat{\mathcal{FK}}$
once we have the well-definedness.
To show the well-definedness,  it is enough to show that 
$K_1\sharp K_2$ is framed isotopic to $K_1'\sharp K_2$
when $K_1'$ is obtained  from $K_1$
by operating (FT3) or (FT6) on $\alpha_1$ by a similar reason to the proof of Theorem \ref{topological monoid}.
Consider the latter case (FT6).
It suffices to show that $K_1\sharp K_2$ is framed isotopic to
$K_3=\alpha_{m}\cdots\alpha_2\cdot  a_{1,1}\cdot\sigma_3^{-2c}\cdot c_{2,0}\cdot\sigma_1^{2c}\cdot\beta_{n-1}\cdots\beta_1$
for $c\in\widehat{\mathbb Z}$ (cf. Figure \ref{proofFT6}).
It 
can be shown in a similar way to Figure  \ref{proofT6}.
\begin{figure}[h]
\begin{center}
        \begin{tikzpicture}
                     \draw[-] (0.,0)  arc (180:360:0.25);
                      \draw[-] (0,0)--(0,0.2)  (0,0.6)--(0,1.0) (0,1.4)--(0,3.3) ;
                      \draw[-] (0.5,0)--(0.5,0.2) (0.5,0.6)--(0.5,1.0) (0.5,1.4)--(0.5,2.7);
 \draw  (0.3,.4)node{$S_2$};
 \draw (-0.1,0.2) rectangle (0.6,0.6); 
                   \draw[-] (0,1.4) --(0.5, 1);
                    \draw[draw=white,double=black, very thick] (0,1) --(0.5, 1.4);
 \draw  (-0.1,1.0) rectangle (0.6,1.4) node[right]{$2c$};
                     \draw[-] (1.1,1.9)  arc (180:360:0.25);
                      \draw[-] (1.1,1.9)--(1.1,2.1)  (1.1,2.5)--(1.1,2.7) ;
                      \draw[-] (1.6,1.9)--(1.6,2.1)  (1.6,2.5)--(1.6,3.3) ;
 \draw (1,2.1) rectangle (1.7,2.5) node[right]{$2c$};
                   \draw[-] (1.1,2.1) --(1.6, 2.5);
                    \draw[draw=white,double=black, very thick] (1.1,2.5) --(1.6, 2.1);
 \draw[-] (0.5,2.7)  arc (180:0:0.3);
 \draw  (0.8,3.5)node{$S_2$};
 \draw (-0.1,3.3) rectangle (1.7,3.7); 
                    \draw[-](0.3,3.7)--(0.3,3.9) (1.3,3.7)--(1.3, 3.9);
 \draw[-] (0.3,3.9)  arc (180:0:0.5);
\draw[color=white,  thick] (-0.1,0)--(1.7,0) (-0.1,0.8)--(1.7,0.8) (-0.1,1.6)--(1.7,1.6) (-0.1,1.9)--(1.7,1.9)
 (-0.1,2.7)--(1.7,2.7) (-0.1,3.1)--(1.7,3.1) (-0.1,3.9)--(1.7,3.9);
        \end{tikzpicture}
\caption{$K_3$}
\label{proofFT6}
\end{center}
\end{figure}

Next consider the former case (FT3).  
Since $K_1$ is a profinite knot, $m_1$ and $m_2$ are both $0$.
By the above argument in case (FT6),
we may assume that both $T_1$  and $T_2$  in Figure \ref{T3}
should start from $\creation$
(i.e. $\alpha_1=\beta_1=\creation$).
Define  $T$  as in Figure \ref{definition of T}.
A successive application of
commutativity of  profinite braids with $T$ shown in (FT4)
and that of creations and annihilations with $T$ shown in Lemma \ref{CCA with framed T}
lead the framed isotopy equivalence shown in Figure \ref{proofT3}.
Hence we get the the well-defined product $\sharp^{\rm{fr}}$.

The proofs of associativity, commutativity, continuity for  $\sharp^{\rm{fr}}$
can be done in a quite same way to the proof of Theorem \ref{topological monoid}.

(2).
The result in \cite{B} indicates that 
the set $\mathcal{FT}$ of (framed) isotopy classes of framed tangles
is described by the set of consistent  finite sequences of  
fundamental tangles,
elements of $A$,  $B$
and $C$
modulo the (discrete) framed Turaev moves,
which is equivalent to 
the moves
replacing profinite tangles and braids by (discrete) tangles and braids
in (FT1)-(FT6) and $c\in\widehat{\mathbb Z}$ by $c\in{\mathbb Z}$ in (FT6).
Because we have a natural map $B_n\to\widehat{B}_n$ and
the Turaev moves are special case of our 6 moves,
we have a natural map $\mathcal{FT}\to\widehat{\mathcal{FT}}$,
which yields our map $h^{\rm{fr}}$.
It is easy to see that it is homomorphic and is with dense image.
\end{proof}

A framed analogue of Lemma \ref{transpose lemma} and  \ref{CCAwithT} hold.

\begin{lem}\label{CCA with framed T}
Let $T$ be a profinite tangle $T$ with  $s(T)=t(T)=\uparrow$.

(1). Its transpose $\opT$ (see. Lemma \ref{transpose lemma} )
is equal to
$
a_{1,0}^{\downarrow\annihilation}
\cdot(e_1^\downarrow\otimes T\otimes e_1^\downarrow)\cdot 
c_{0,1}^{\creation\downarrow}
$ (cf. Figure \ref{opT}).

(2).
The equalities in Figure \ref{CAT} hold for $T$.

The same claim also holds for a profinite tangle $T$ with  $s(T)=t(T)=\downarrow$
by reversing all arrows.
\end{lem}

\begin{proof}
(1). A proof is depicted in Figure \ref{framed transpose lemma}.
We use (FT6) in the first equality.
The second and the fourth equalities follow from  Lemma \ref{various twists}.(1).
The third equality is obtained by the \lq commutativitiy' of twists with
annihilation, creation and profinite braids
assured by Lemma \ref{various twists}.(2) and  (FT4).
The fifth equation follows from (FT2) and (FT4).

\begin{figure}[h]
\begin{center}
        \begin{tikzpicture}
                      \draw[<-] (-0.2,-0.4)--(-0.2,1.0) ;
                     \draw[->] (0.2,1.0)  arc (0:180:0.2);
                     \draw[<-] (0.2,0.2)  arc (180:360:0.2);
                      \draw (0,0.4) rectangle (0.4,0.8);
                      \draw (0.2,0.6) node{$T$};
                      \draw[->] (0.2,0.2)--(0.2, 0.4)  ;
                      \draw[->] (0.2,0.8)--(0.2,1.0) ;
                      \draw[<-] (0.6,0.2)--(0.6,1.6) ;
               \draw[color=white, very thick] (-0.3,0.2)--(0.7,0.2) (-0.3,1.0)--(0.7,1.0)  ;
\draw  (1.0,0.6)node{$=$};
                      \draw[<-] (1.4,-1.3)--(1.4,1.0) ;
    \draw[->] (1.8,1.0)  arc (0:180:0.2);
                      \draw[->] (1.8,0.8)--(1.8,1.0) ;
                      \draw (1.6,0.4) rectangle (2.0,0.8);
                      \draw (1.8,0.6) node{$T$};
                      \draw[->] (1.8,-0.7)--(1.8, 0.4)  ;
                      \draw[-] (1.8,-0.7)--(2.1,-1) ;
        \draw[draw=white,double=black, very thick]  (1.8, -1)--(2.1,-0.7) ;
                     \draw[->] (1.8,-1)  arc (180:360:0.15);
                      \draw[<-] (2.1,-0.7)--(2.1,0) ;
    \draw[<-] (2.1,0)  arc (180:0:0.15);
                      \draw[->] (2.4,-0.2)--(2.4,0) ;
                      \draw[-] (2.4,-0.5)--(2.7,-0.2) ;
        \draw[draw=white,double=black, very thick]  (2.4, -0.2)--(2.7,-0.5) ;
    \draw[->] (2.4,-0.5)  arc (180:360:0.15);
                      \draw[<-] (2.7,-0.2)--(2.7,1.6) ;
               \draw[color=white, thick] (1.3,-1)--(2.8,-1) (1.3,-0.7)--(2.8,-0.7) (1.3,-0.5)--(2.8,-0.5) 
                    (1.3,-0.2)--(2.8,-0.2) (1.3,0)--(2.8,0) (1.3,0.2)--(2.8,0.2) (1.3,1.0)--(2.8,1.0)  ;
\draw  (3.1, 0.6)node{$=$};
                      \draw[<-] (3.5,-1.3)--(3.5,1.0) ;
    \draw[->] (3.9,1.0)  arc (0:180:0.2);
                      \draw[->] (3.9,0.8)--(3.9,1.0) ;
                      \draw (3.7,0.4) rectangle (4.1,0.8);
                      \draw (3.9,0.6) node{$T$};
                      \draw[->] (3.9,-0.6)--(3.9, 0.4)  ;
     \draw[->] (3.9,-1)  arc (180:360:0.2);
                      \draw (4.1,-0.4) rectangle (4.5,0);
                      \draw (4.3,-0.2) node{$\varphi^\downarrow$};
                      \draw[<-] (4.3,-0.6)--(4.3, -0.4) ;
                      \draw[<-]  (4.3,0)--(4.3,1.6) ;
                      \draw[-] (3.9,-0.6)--(4.3,-1) ;
        \draw[draw=white,double=black, very thick]  (3.9, -1)--(4.3,-0.6) ;
               \draw[color=white, thick] (3.4,-1)--(4.4,-1) (3.4,-0.6)--(4.4,-0.6) (3.4,0.2)--(4.4,0.2) (3.4,1)--(4.4,1) ;
\draw  (4.7, 0.6)node{$=$};
                      \draw[<-] (5.1,-1.3)--(5.1,1.2) ;
                      \draw (4.9,1.2) rectangle (5.3,1.6);
                      \draw (5.1,1.4) node{$\varphi^\downarrow$};
                      \draw[<-] (5.1,1.6)--(5.1,1.8) ;
    \draw[->] (5.5,1.8)  arc (0:180:0.2);
                      \draw[->] (5.5,0.8)--(5.5,1.8) ;
                      \draw (5.3,0.4) rectangle (5.7,0.8);
                      \draw (5.5,0.6) node{$T$};
                      \draw[->] (5.5,-0.6)--(5.5, 0.4)  ;
     \draw[->] (5.5,-1)  arc (180:360:0.2);
                      \draw[<-] (5.9,-0.6)--(5.9, 2.1) ;
                      \draw[-] (5.5,-0.6)--(5.9,-1) ;
        \draw[draw=white,double=black, very thick]  (5.5, -1)--(5.9,-0.6) ;
               \draw[color=white, thick] (5,-1)--(6,-1)  (5,-0.6)--(6,-0.6) (5,0.2)--(6,0.2) (5,1)--(6,1) (5,1.8)--(6,1.8);
\draw  (6.4, 0.6)node{$=$};
                      \draw[<-] (6.8,-0.4)--(6.8,1.0) ;
                     \draw[-] (6.8,1)--(7.2,1.4) ;
        \draw[draw=white,double=black, very thick]  (6.8, 1.4)--(7.2,1) ;
                     \draw[<-] (7.2,1.4)  arc (0:180:0.2);
                     \draw[-] (7.2,0.2)--(7.6,-0.2) ;
        \draw[draw=white,double=black, very thick]  (7.2, -0.2)--(7.6,0.2) ;
                     \draw[->] (7.2,-0.2)  arc (180:360:0.2);
                      \draw (7.0,0.4) rectangle (7.4,0.8);
                      \draw (7.2,0.6) node{$T$};
                      \draw[->] (7.2,0.2)--(7.2, 0.4)  ;
                      \draw[->] (7.2,0.8)--(7.2,1.0) ;
                      \draw[<-] (7.6,0.2)--(7.6,1.6) ;
               \draw[color=white, very thick] (6.7,-0.2)--(7.7,-0.2) (6.7,0.2)--(7.7,0.2) (6.7,1)--(7.7,1)(6.7,1.4)--(7.7,1.4)  ;
\draw  (8, 0.6)node{$=$};
                      \draw[<-] (8.4,0.2)--(8.4,1.6) ;
                     \draw[->] (8.4,0.2)  arc (180:360:0.2);
                      \draw (8.6,0.4) rectangle (9,0.8);
                      \draw (8.8,0.6) node{$T$};
                      \draw[->] (8.8,0.2)--(8.8, 0.4)  ;
                      \draw[->] (8.8,0.8)--(8.8,1.0) ;
                      \draw[<-] (9.2,-0.4)--(9.2,1.0) ;
                     \draw[<-] (9.2,1.0)  arc (0:180:0.2);
               \draw[color=white, very thick] (8.3,0.2)--(9.3,0.2) (8.3,1.0)--(9.3,1.0)  ;
         \end{tikzpicture}
\caption{Proof of Lemma \ref{CCA with framed T}.(1).}
\label{framed transpose lemma}
\end{center}
\end{figure}

(2). It is a  direct consequence of (1).
\end{proof}

\begin{defn}\label{definition of GFK}
The {\it group of profinite framed knots} $\mathrm{Frac}\widehat{\mathcal{FK}}$ is defined to be the group of fraction
of the monoid  $\widehat{\mathcal{FK}}$
as in the same way to Definition \ref{definition of GK}.
\end{defn}

We also have a framed version of  our Theorem B in \S \ref{introduction}.

\begin{thm}
Fix an embedding from the algebraic closure $\overline{\mathbb Q}$ of 
the rational number field $\mathbb Q$
into the complex number field $\mathbb C$.
Then the group $\mathrm{Frac}\widehat{\mathcal{FK}} $ of profinite knots admits
a non-trivial topological $G_{\mathbb Q}$-module structure.
Namely, 
there is a non-trivial continuous Galois representation on profinite knots
$$
\rho_0^{\rm{fr}}:G_{\mathbb Q}\to \mathrm{Aut}\ \mathrm{Frac}\widehat{\mathcal{FK}} .
$$
\end{thm}

\begin{proof}
This is achieved by establishing a consistent 
continuous action of the profinite Grothendieck-Teichm\"{u}ller group 
$\widehat{GT}$ on $\mathrm{Frac}\widehat{\mathcal{FK}}$
and using the inclusion from 
$G_{\mathbb Q}$
into $\widehat{GT}$.

We introduce this $\widehat{GT}$-action on $\mathrm{Frac}\widehat{\mathcal{FK}}$
by using the action on each profinite fundamental tangle
given in Definition \ref{GT-action on profinite knots}.
The proof of  its well-definedness, i.e. to check  that the action preserves (FT1)-(FT6),
can be done in the same way to that of Theorem \ref{GT-action theorem}
but for (FT6), which is proved below:
Let $\sigma=(\lambda,f)\in\widehat{GT}$.
Let $r_1$ and $r_2$ be two profinite knots.
Assume that $r_1$ is obtained from  $r_2$ by a single operation of  the move (FT6).
Then by Proposition \ref{action-change-basepoint-braids}
and Proposition \ref{action-change-basepoint-knots},
we have
\begin{align*}
\sigma(a&_{k,l}^{\epsilon_2})\cdot\sigma((\sigma_{k+1}^{\epsilon_2'})^{-c})\cdot \sigma(a_{k+2,l}^{\epsilon_3})\cdot\sigma((\sigma_{k+3}^{\epsilon_3'})^c)\cdot \sigma(c_{k+1,l+1}^{\epsilon_1}) \\
=&a_{k,l}^{\epsilon_2}\cdot f_{1\cdots k,k+1,k+2}^{s(a_{k,l}^{\epsilon_2})}\cdot
 f_{1\cdots k,k+1,k+2}^{-1,s(a_{k,l}^{\epsilon_2})}\cdot
(\sigma_{k+1}^{\epsilon'_2})^{-c\lambda}\cdot
 f_{1\cdots k,k+1,k+2}^{t(a_{k+2,l}^{\epsilon_3})} \\
&\quad 
\cdot a_{k+2,l}^{\epsilon_3}\cdot
f_{1\cdots k+2,k+3,k+4}^{s(a_{k+2,l}^{\epsilon_3})}\cdot
f_{1\cdots k+2,k+3,k+4}^{-1,s(a_{k+2,l}^{\epsilon_3})}\cdot
(\sigma_{k+3}^{\epsilon'_3})^{c\lambda}\cdot
f_{1\cdots k+2,k+3,k+4}^{t(c_{k+1,l+1}^{\epsilon_1})} \\
&\qquad 
\cdot f_{1\cdots k+1,k+2,k+3}^{-1,t(c_{k+1,l+1}^{\epsilon_1})}\cdot
c_{k+1,l+1}^{\epsilon_1}
\\
=&a_{k,l}^{\epsilon_2}\cdot 
(\sigma_{k+1}^{\epsilon'_2})^{-c\lambda}\cdot
 f_{1\cdots k,k+1,k+2}^{t(a_{k+2,l}^{\epsilon_3})} 
\cdot a_{k+2,l}^{\epsilon_3}\cdot
(\sigma_{k+3}^{\epsilon'_3})^{c\lambda}\cdot
f_{1\cdots k+2,k+3,k+4}^{t(c_{k+1,l+1}^{\epsilon_1})} \\
&\qquad 
\cdot f_{1\cdots k+1,k+2,k+3}^{-1,t(c_{k+1,l+1}^{\epsilon_1})}\cdot
c_{k+1,l+1}^{\epsilon_1}
\\
=&a_{k,l}^{\epsilon_2}\cdot 
(\sigma_{k+1}^{\epsilon'_2})^{-c\lambda}\cdot
a_{k+2,l}^{\epsilon_3}\cdot
(\sigma_{k+3}^{\epsilon'_3})^{c\lambda} \\
&\qquad 
\cdot f_{1\cdots k+2,k+3,k+4}^{t(c_{k+1,l+1}^{\epsilon_1})} 
\cdot f_{1\cdots k,k+1,k+2}^{t(c_{k+1,l+1}^{\epsilon_1})} 
\cdot f_{1\cdots k+1,k+2,k+3}^{-1,t(c_{k+1,l+1}^{\epsilon_1})}\cdot
c_{k+1,l+1}^{\epsilon_1}. \\
\intertext{By a successive application of \eqref{pentagon equation},}
=&a_{k,l}^{\epsilon_2}\cdot 
(\sigma_{k+1}^{\epsilon'_2})^{-c\lambda}\cdot
a_{k+2,l}^{\epsilon_3}\cdot
(\sigma_{k+3}^{\epsilon'_3})^{c\lambda} \\
&\qquad 
\cdot f_{1\cdots k,k+1,k+2}^{t(c_{k+1,l+1}^{\epsilon_1})} 
\cdot f_{1\cdots k+1,k+2, k+3 \ k+4}^{-1,t(c_{k+1,l+1}^{\epsilon_1})} 
\cdot f_{k+2,k+3,k+4}^{t(c_{k+1,l+1}^{\epsilon_1})}\cdot
c_{k+1,l+1}^{\epsilon_1}. \\
\intertext{By (FT4) and Lemma \ref{projection of f}, }
=&a_{k,l}^{\epsilon_2}\cdot 
(\sigma_{k+1}^{\epsilon'_2})^{-c\lambda}\cdot
f_{1\cdots k,k+1,k+2}^{t(a_{k+2,l}^{\epsilon_3})} \cdot
a_{k+2,l}^{\epsilon_3}\cdot
(\sigma_{k+3}^{\epsilon'_3})^{c\lambda} 
\cdot f_{k+2,k+3,k+4}^{t(c_{k+1,l+1}^{\epsilon_1})}\cdot
c_{k+1,l+1}^{\epsilon_1}. \\
\intertext{By  Lemma \ref{ITC},
twists $(\varphi^\uparrow)^c$ and $(\varphi^\downarrow)^c$ ($c\in \widehat{\mathbb Z}$)
\lq commute' with  profinite tangles. 
Thus 
}
=&a_{k,l}^{\epsilon_2}\cdot 
(\sigma_{k+1}^{\epsilon'_2})^{-c\lambda}\cdot
a_{k+2,l}^{\epsilon_3}\cdot
(\sigma_{k+3}^{\epsilon'_3})^{c\lambda} \cdot
c_{k+1,l+1}^{\epsilon_1}\cdot
f_{1\cdots k,k+1,k+2}^{s(a_{k,l}^{\epsilon'_1})} \\
&\qquad
\cdot a_{k+2,l}^{\epsilon_0}
\cdot f_{k+2,k+3,k+4}^{s(a_{k+2,l}^{\epsilon_0})}\cdot
c_{k+1,l+1}^{\epsilon'_0} \\
\intertext{with appropriate orientations $\epsilon_0$ and $\epsilon'_0$. Hence}
=&a_{k,l}^{\epsilon'_1}\cdot
f_{1\cdots k,k+1,k+2}^{s(a_{k,l}^{\epsilon'_1})} 
\cdot a_{k+2,l}^{\epsilon_0}
\cdot f_{k+2,k+3,k+4}^{s(a_{k+2,l}^{\epsilon_0})}\cdot c_{k+1,l+1}^{\epsilon'_0} \\
=&\sigma(a_{k,l}^{\epsilon'_1})
\cdot a_{k+2,l}^{\epsilon_0}
\cdot f_{k+2,k+3,k+4}^{s(a_{k+2,l}^{\epsilon_0})}\cdot c_{k+1,l+1}^{\epsilon'_0} .
\end{align*}
Therefore we have
$\sigma(r_1)=\sigma(r_2)\sharp\Lambda_f$.
Because $\alpha(r_1)=\alpha(r_2)+1$, we may say that
\eqref{FT6b} is preserved 
by the $\widehat{GT}$-action.
The proof for \eqref{FT6a} can be done in the same way.
Thus  the first framed Reidemeister move (FT6) is preserved by the $\widehat{GT}$-action.

The rest part of our claim; continuity, non-triviality, etc, 
was shown in the proof of Theorem \ref{GT-action theorem}.(2)-(5).
\end{proof}

\begin{lem}\label{infinite twists}
The map  sending $c\in{\mathbb N}$ to  the $c$-th power
$(\varphi^\uparrow)^c\in\widehat{\mathcal{FT}}$
of $\varphi^\uparrow$
extends additively and continuously into $c\in\widehat{\mathbb Z}$
and it is given by
\begin{equation}\label{infinite twist formula}
(\varphi^\uparrow)^c
=a_{0,1}\cdot (\sigma_1^{-c})^{\uparrow\downarrow\uparrow}\cdot
c_{1,0}^{\uparrow\creation}.
\end{equation}
with $a_{0,1}=a_{0,1}^{\opannihilation\uparrow}$ or $a_{0,1}^{\annihilation\uparrow}$
according to $c\equiv 1 $ or  $0 \pmod 2$
(cf. Figure \ref{infinite twist figure}).

The same statement also holds for $\varphi^\downarrow$.
\end{lem}
\begin{figure}[h]
\begin{center}
       \begin{tikzpicture}
                     \draw[->] (1.1,1.5)--(1.1,1.8);
                      \draw[-] (1.1,1.8)--(1.1,2)  (1.1,2.5)--(1.1,2.7) ;
                      \draw[-] (1.6,1.8)--(1.6,2.) (1.6,2.5)--(1.6,2.7) ;
 \draw (1,2.) rectangle (1.7,2.5) node[right]{$c$};
                   \draw[-] (1.1,2.) --(1.6, 2.5);
                    \draw[draw=white,double=black, very thick] (1.1,2.5) --(1.6, 2.);
 \draw[-] (1.6,2.7)  arc (0:180:0.25);
 \draw[->] (1.6,1.8)  arc (180:360:0.25);
                     \draw[->] (2.1,1.8)--(2.1,3);
\draw[color=white,  thick] (1,1.8)--(2.2,1.8) (1,2.7)--(2.2,2.7);
        \end{tikzpicture}
\caption{$(\varphi^\uparrow)^c$}
\label{infinite twist figure}
\end{center}
\end{figure}
\begin{proof}
The equation \eqref{infinite twist formula} for $c\in{\mathbb N}$ 
can be shown by induction:
Assume the validity for $c$ with $c\equiv 1\pmod 2$.
Then 
\begin{align*}
(\varphi^\uparrow)^{c+1}
&=(\varphi^\uparrow)^c\cdot \varphi^\uparrow
=a_{0,1}^{\opannihilation\uparrow}\cdot (\sigma_1^{-c})^{\uparrow\downarrow\uparrow}\cdot
c_{1,0}^{\uparrow\creation}\cdot \varphi^\uparrow, \\
\intertext{by Lemma \ref{various twists}.(2)}
&=a_{0,1}^{\opannihilation\uparrow}\cdot 
(e_{\downarrow}\otimes\varphi^\uparrow\otimes e_{\uparrow})\cdot
(\sigma_1^{-c})^{\uparrow\downarrow\uparrow}\cdot
c_{1,0}^{\uparrow\creation}, \\
\intertext{by Lemma \ref{various twists}.(1)}
&=a_{0,1}^{\annihilation\uparrow}\cdot 
(\sigma_1^{-c-1})^{\uparrow\downarrow\uparrow}\cdot
c_{1,0}^{\uparrow\creation}. 
\end{align*}
The case for $c$ with $c\equiv 0\pmod 2$ can be done in the same way.
By the expression of \eqref{infinite twist formula},
it is immediate to see  that  $(\varphi^\uparrow)^c$ consistently extends to
the case for
$c\in\widehat{\mathbb Z}$.
\end{proof}

As an analogue of Lemma \ref{various twists} (2), we have

\begin{lem}\label{ITC}
We have the following equalities
$$
((\varphi^\downarrow)^c\otimes \uparrow)\cdot\creation=
(\downarrow\otimes(\varphi^\uparrow)^c)\cdot \creation,\qquad
\annihilation\cdot ((\varphi^\uparrow)^c\otimes \downarrow)
=\annihilation\cdot(\uparrow\otimes(\varphi^\downarrow)^c)
$$
for $c\in\widehat{\mathbb Z}$.
We also have  the same equalities  obtained by reversing all arrows.
\end{lem}

\begin{proof}
A proof is depicted in Figure \ref{Proof of Lemma ITC}.
We note that we use Lemma \ref{various twists}.(1) and
(FT6) to deduce the second
and the third equality respectively.
\begin{figure}[h]
\begin{center}
       \begin{tikzpicture}
      \draw[->] (.1,1.5)  arc (180:360:0.7);
                     \draw[<-] (.1,1.5)--(.1,2);
                      \draw[-]  (.1,2.5)--(.1,2.7) ;
      \draw[-] (.1,2.7)  arc (180:0:0.25);
                      \draw[-] (.6,1.8)--(.6,2.) (.6,2.5)--(.6,2.7) ;
      \draw (0,2.) rectangle (.7,2.5) node[right]{$c$};
                   \draw[-] (.1,2) --(.6, 2.5);
                    \draw[draw=white,double=black, very thick] (.1,2.5) --(.6, 2.);
      \draw[<-] (.6,1.8)  arc (180:360:0.25);
                     \draw[<-] (1.1,1.8)--(1.1,3.1);
                     \draw[->] (1.5,1.5)--(1.5,3.1);
      \draw[color=white,  thick] (0,1.5)--(1.6,1.5) (0,1.8)--(1.6,1.8) (0,2.7)--(1.6,2.7) ;
\draw  (1.9, 1.8)node{$=$};
                   \draw[->] (2.3,1.3) --(2.8, .8);
                    \draw[draw=white,double=black, very thick] (2.3,0.8) --(2.8, 1.3);
      \draw[<-] (2.3,0.8)  arc (180:360:0.25);
                   \draw[->] (2.8,1.3) --(3.3, 1.8);
                    \draw[draw=white,double=black, very thick] (2.8,1.8) --(3.3, 1.3);
                    \draw[->] (2.3,3.1)--(2.3,1.3) ;
                      \draw[-] (2.8,1.8)--(2.8,2)  (2.8,2.5)--(2.8,2.7) ;
      \draw[-] (2.8,2.7)  arc (180:0:0.25);
                      \draw[-] (3.3,1.8)--(3.3,2.) (3.3,2.5)--(3.3,2.7) ;
      \draw (2.7,2.) rectangle (3.4,2.5) node[right]{$c$};
                   \draw[-] (2.8,2) --(3.3, 2.5);
                    \draw[draw=white,double=black, very thick] (2.8,2.5) --(3.3, 2.);
                   \draw[<-] (3.3,.5) --(3.3, 1.3);
      \draw[->] (3.3,0.5)  arc (180:360:0.25);
                   \draw[->] (3.8,0.5) --(3.8, 3.1);
      \draw[color=white,  thick] (2.2,.5)--(3.9,.5) (2.2,.8)--(3.9,.8) (2.2,1.3)--(3.9,1.3) (2.2,1.8)--(3.9,1.8)
(2.2,2.7)--(3.9,2.7);
\draw  (4.2, 1.8)node{$=$};
                   \draw[<-] (4.6,0.2) --(4.6, 3.1);
      \draw[->] (4.6,.2)  arc (180:360:0.25);
                   \draw[->] (5.1,.2) --(5.1, .5);
                   \draw[->] (5.1,1) --(5.6, .5);
                    \draw[draw=white,double=black, very thick] (5.6,1) --(5.1, .5);
      \draw[<-] (5.1,1)  arc (180:0:0.25);
      \draw[->] (5.6,0.5)  arc (180:360:0.25);
                      \draw[-] (6.1,0.5)--(6.1,1.5)  ;
                      \draw[-]  (6.1,2.)--(6.1,2.2) (6.6,2.)--(6.6,2.2);

                   \draw[-] (6.1,1.5) --(6.6, 2.);
                    \draw[draw=white,double=black, very thick] (6.6,1.5) --(6.1, 2.);
      \draw (6,1.5) rectangle (6.7,2.) node[right]{$c$};
                   \draw[-] (6.1,2.2) --(6.6, 2.7);
                    \draw[draw=white,double=black, very thick] (6.6,2.2) --(6.1, 2.7);
      \draw[-] (6.1,2.7)  arc (180:0:0.25);
                   \draw[->] (6.6,1.5) --(6.6, -0.1);
      \draw[->] (6.6,-0.1)  arc (180:360:0.3);
                   \draw[->] (7.2,-0.1) --(7.2, 3.1);
      \draw[color=white,  thick] (4.5,-0.1)--(7.3,-0.1)  (4.5,.2)--(7.3,.2)   (4.5,.5)--(7.3,.5)  (4.5,1)--(7.3,1)  (4.5,1.3)--(7.3,1.3)  (4.5,2.2)--(7.3,2.2)  (4.5,2.7)--(7.3,2.7)  ;
\draw  (7.6, 1.8)node{$=$};
                    \draw[->] (8,3.1)--(8,1.3) ;
      \draw[->] (8,1.3)  arc (180:360:0.25);
                      \draw[-] (8.5,1.3)--(8.5,2)  (8.5,2.5)--(8.5,2.7) ;
      \draw[-] (8.5,2.7)  arc (180:0:0.25);
                      \draw[-] (9,1.8)--(9,2.) (9,2.5)--(9,2.7) ;
      \draw (8.4,2.) rectangle (9.1,2.5) node[right]{$c$};
                   \draw[-] (8.5,2) --(9, 2.5);
                    \draw[draw=white,double=black, very thick] (8.5,2.5) --(9, 2.);
      \draw[->] (9,1.8)  arc (180:360:0.25);
                   \draw[->] (9.5,1.8) --(9.5, 3.1);
      \draw[color=white,  thick] (7.9,1.3)--(9.6,1.3)  (7.9,1.8)--(9.6,1.8)  (7.9,2.7)--(9.6, 2.7);
        \end{tikzpicture}
\caption{Proof of Lemma \ref{ITC}}
\label{Proof of Lemma ITC}
\end{center}
\end{figure}
\end{proof}
We note that  the $G_{\mathbb Q}$-module structure on 
the group $\mathrm{Frac}\widehat{\mathcal{K}} $ of profinite knots 
given in Theorem B in \S \ref{introduction} is
induced from  that on
the group $\mathrm{Frac}\widehat{\mathcal{FK}} $ of profinite framed knots
by \eqref{forget frame}.


\thanks{
{\it Acknowledgements}.
The author would like to express his profound gratitude to
Masanori Morishita who persistently encouraged this research.
His thanks are also directed to
Hiroshi Fujiwara who patiently helped him to draw many graphics in this paper,
Hitoshi Murakami, Toshitake Kohno and Tetsuya Ito
who generously explained him lots of valuable issues in low dimensional topology
related to this paper.
Part of the paper was written at Max Planck Institute for Mathematics.
He also thanks the institute for its hospitality.
This work was supported by 
Grant-in-Aid for Young Scientists (A) 24684001.
The referees efforts to make this paper better is gratefully acknowledged.
}


\end{document}